\documentclass[a4paper,11pt,DIV=11,
abstract=on
]{scrartcl}
\usepackage{mathtools}
\mathtoolsset{showonlyrefs}
\usepackage{amsmath, amsthm, amssymb}
\usepackage{fontenc}
\usepackage[utf8]{inputenc}
\usepackage[english]{babel}
\usepackage{graphicx}
\usepackage{subcaption,xargs}
\usepackage{booktabs}
\usepackage{enumitem,listings}
\usepackage{environ,ifthen,xcolor}
\usepackage{hyperref}
\usepackage[
  defernumbers=true,
  backend=biber,
  style=numeric-comp,
  firstinits=true,
  natbib=true,
  maxcitenames=5,
  sortlocale=en_EN,
  url=true, 
  doi=true,
  eprint=true,
  maxnames=99
]{biblatex}
\setkomafont{sectioning}{\rmfamily\bfseries}
\setkomafont{title}{\rmfamily}
\makeatletter
\newcommandx{\abs}[2][1=\@empty]{#1\lvert #2 #1\rvert}
\newcommandx{\norm}[3][1=\@empty,3=\@empty]{#1\lVert #2 #1\rVert_{#3}}
\renewcommand{\div}{\operatorname{div}}
\newcommandx{\PT}{\operatorname{PT}}
\newcommandx{\M}{\mathcal{M}}

\newcommand{\T}{\mathrm{T}}
\newcommand{\E}{\mathcal{E}}
\makeatother
\DeclareMathOperator*{\argmin}{arg\,min}

\newtheorem{theorem}{Theorem}[section]
\newtheorem{definition}[theorem]{Definition}

\newtheorem{lemma}[theorem]{Lemma}
\newtheorem{remark}[theorem]{Remark}

\newtheorem{corollary}[theorem]{Corollary}

\hypersetup{pdfauthor={Ronny Bergmann, Daniel Tenbrinck},
	pdftitle={A Graph Framework for Manifold-Valued Data},
	pdfsubject={arXiv Preprint},%
	pdfcreator = {pdflatex and TextMate},%
	pdfkeywords={},
	plainpages=false, pdfstartview=FitH, pdfview=FitH, pdfpagemode=UseOutlines,%
	bookmarksnumbered=true,bookmarksopen=false,bookmarksopenlevel=0,%
	colorlinks=true,
		linkcolor=green!50!black,
		citecolor=red!50!black,
		urlcolor=blue!50!black%
}
\title{A Graph Framework for Manifold-valued Data}
\author{Ronny Bergmann%
	\thanks{Felix-Klein-Zentrum, Fachbereich für Mathematik,%
	Technische Universität Kaiserslautern, 67663 Kaiserslautern, Germany.\newline
	\texttt{bergmann@mathematik.uni-kl.de}}%
	\and%
	Daniel Tenbrinck%
	\thanks{Institute for Computational and Applied Mathematics,
	Westfälische Wilhelms-Universität Münster, 48149 Münster, Germany.\newline
	\texttt{daniel.tenbrinck@uni-muenster.de}}%
}
\date{October 2, 2017}
\begin{document}
\maketitle
\begin{abstract}
Graph-based methods have been proposed as a unified framework for discrete
calculus of local and nonlocal image processing methods in the recent years.
In order to translate variational models and partial differential equations to
a graph, certain operators have been investigated and successfully
applied to real-world applications involving graph models.
So far the graph framework has been limited to real- and vector-valued functions on Euclidean domains.
In this paper we generalize this model to the case of manifold-valued data.
We introduce the basic calculus needed to formulate variational models and
partial differential equations for manifold-valued functions and discuss the
proposed graph framework for two particular families of operators, namely,
the isotropic and anisotropic graph~$p$-Laplacian operators, $p\geq1$.
Based on the choice of $p$ we are in particular able to solve optimization
problems on manifold-valued functions involving
total variation ($p=1$) and Tikhonov ($p=2$) regularization.
Finally, we present numerical results from processing both synthetic as well
as real-world manifold-valued data, e.g., from diffusion tensor imaging (DTI) and light detection and ranging (LiDAR) data.
\end{abstract}

\section{Introduction}
Variational methods and partial differential equations (PDEs) play a key role
for modeling and solving image processing tasks. Typically, in the setting of
vector-valued functions the respective continuous formulations are discretized,
e.g., by using finite differences or finite elements. These discretization
schemes are so popular because they are well-investigated and in general
preserve important properties of the continuous models, e.g., conservation laws
or maximum principles. Recently, the trend is to establish numerical
discretization schemes for functions not living on subsets of Euclidean spaces
but on surfaces, which might be given as manifolds in a continuous setting or as
finite point clouds in a discrete setting. The discretization of differential
operators becomes especially challenging for the latter case as there are
a-priori no explicit neighborhood relationships for the raw point cloud data.
Currently, there is high interest in translating variational methods and PDEs
to discrete data, which are modeled by finite weighted graphs and networks.
In fact, any discrete data can be represented by a finite graph in which the
vertices are associated to the data and its edges correspond to
relationships within the data. This modeling approach has led to interesting
applications in mathematical image processing and machine learning so far.
One major advantage of using graph-based methods is that they unify local and
nonlocal methods. This is due to the reason that one may model data
relationships not only based on geometric proximity but also based on
similarity of features. Furthermore, graphs can model
arbitrary spatial neighborhoods, e.g., when the domain itself is a submanifold
of the space the measurements are taken on.
In order to translate and solve PDEs on graphs, different discrete vector
calculi have been proposed in the literature, e.g., see~\cite{GP10} and
references therein.
One rather simple concept is to use discrete partial
differences~\cite{ELB08,GO08}. This mimetic approach allows to solve PDEs on
both regular as well as irregular data domains by replacing
continuous partial differential operators, e.g., gradient or divergence, by a
reasonable discrete analogue. By employing this approach many important tools and
results from the continuous setting can be transferred to finite weighted
graphs.

On the other hand, variational methods from image processing have recently been
generalized from real and vectorial data in Euclidean spaces to the case of
manifold-valued data, most prominently the total variation (TV)
regularization~\cite{SC11,CS13,LSKC13,WDS14} including efficient
algorithms~\cite{BPS16,BCHPS16} and second order differences based
models~\cite{BLSW14,BBSW16}. These manifold-valued images occur for example in
interferometric synthetic aperture radar (InSAR) imaging~\cite{MF98,BRF00,
DDT11}, where the measured phase-valued data may be noisy and/or incomplete.
Sphere-valued data appears, e.g., in directional analysis~\cite{KS02,MJ00}.
Data items with values in the special rotation group~\(\operatorname{SO}(3)\)
(or a quotient manifold thereof) are used to process orientations (including
certain symmetries) e.g., in electron backscattered diffraction imaging
(EBSD)~\cite{KWAD93,BHJPSW10,ASWK93}. Another application is diffusion tensor
imaging (DT-MRI)~\cite{BML94,FJ07,PFA06,VBK13}, where the diffusion tensors can
be represented as \(3\times 3\) symmetric positive definite matrices, which
also form a manifold. In general when dealing with covariance matrices the
given data is also a set of symmetric positive definite matrices of
size~\(n\times n\). Finally, multivariate Gaussian distributions can be
characterized as values on a hyperbolic space, namely the Poincaré
half-space~\cite{AV14}. Two examples are shown in Figure~\ref{fig:ExIntro}.
\begin{figure}\centering
  \begin{subfigure}[b]{.43\textwidth}\centering
    \includegraphics{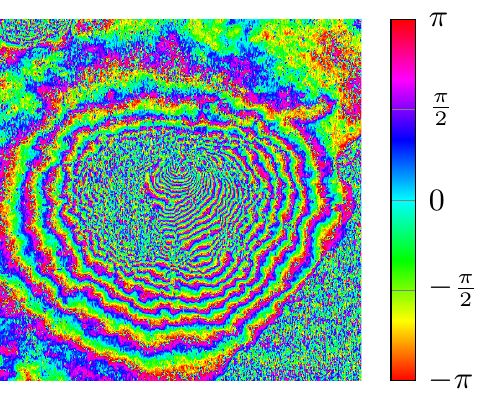}
    \caption{InSAR data of Mt.~Vesuvius from~\cite{RPG97}}
    \label{subfig:ExIntro:InSAR}
  \end{subfigure}
  \begin{subfigure}[b]{.54\textwidth}\centering
    \includegraphics[width=.75\textwidth]{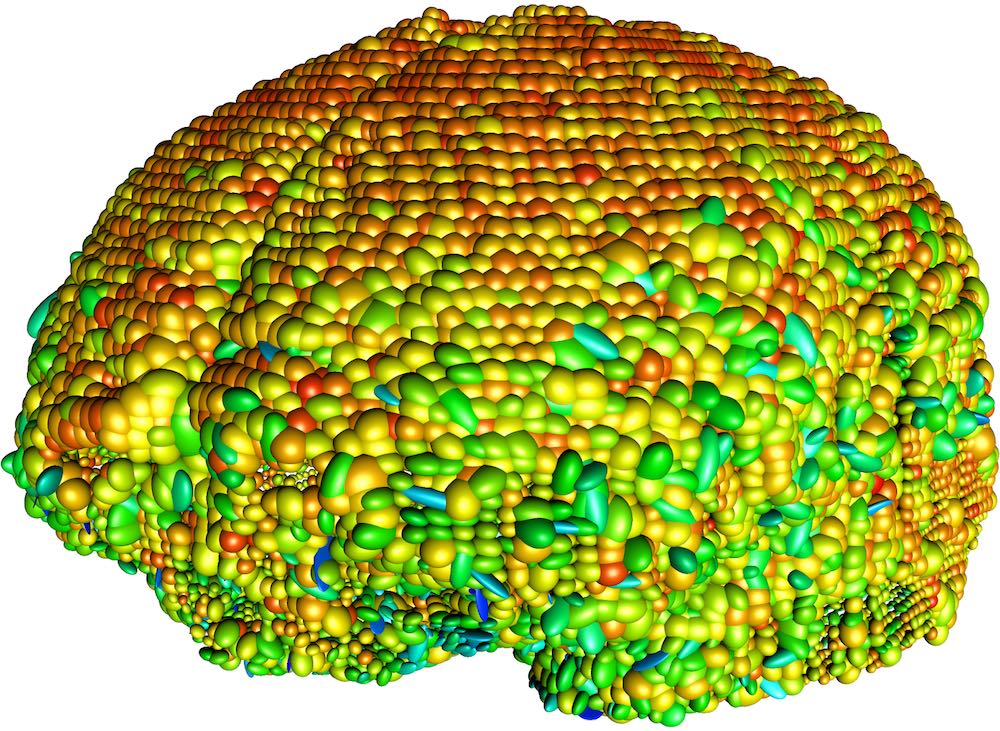}
    \caption{DT-MRI data on the surface of a human brain, extracted from~\cite{Camino}}
    \label{subfig:ExIntro:DT-MRI-Surface}
  \end{subfigure}
  \caption{Two examples of manifold-valued data:\ %
  (\subref{subfig:ExIntro:InSAR}) an Image of phase valued data in InSAR imaging
  (\subref{subfig:ExIntro:DT-MRI-Surface}) diffusion tensors from DT-MRI given on an implicit surface.}
  \label{fig:ExIntro}
\end{figure}
  The data of Mt.~Vesuvius shown in Fig.~\ref{subfig:ExIntro:InSAR})
  obtained in InSAR measurements is available from~\cite{RPG97}\footnote{%
see~\url{https://earth.esa.int/workshops/ers97/program-details/speeches/rocca-et-al/}%
}.
The data of a phase value in each pixel is given on the regular rectangular
image grid. Each pixel represents the “height modulo wavelength” at the
measured position, i.e.~a phase. Note that the distance measure for phase
values is different from the classical images with values on the real line.
This is indicated by the hue color map. Values near both the minimal and
maximal value are again close with respect to the distance on the sphere
\(\mathbb S^1\). In Figure~\ref{subfig:ExIntro:DT-MRI-Surface}) data measured
on the implicitly given surface of the human brain is shown. The data is
available within the Camino toolbox~\cite{Camino}\footnote{see
\url{http://camino.cs.ucl.ac.uk}}. Within the extracted surface data set, each
pixel is given on the implicit surface, see Section~\ref{s:DT-MRI} for details
on the extraction, while each pixel is a value representing the diffusion
tensor at this point, i.e.~a symmetric positive definite matrix. These are
visualized using their eigenvectors and eigenvalues as main axes and main axes
lengths, respectively. Since the data is given on an implicit surface, the
vicinity of pixel or the modeling of pixel being neighbors, can be done using a
graph. This graph has the pixel as their nodes and edges to neighboring pixel,
e.g.~whenever two pixel are below a certain threshold in the \(\mathbb R^3\)
the data is measured in.

All the mentioned applications including the two examples from Figure~\ref{fig:ExIntro} have in common that they
resemble manifold-valued images or more general manifold-valued data. This
yields discrete energy functionals that map from a tensor product of a manifold
with itself (the size of the number of pixels) into the (extended) real line.
Starting from the early '90s~\cite{Udr94} several groups worked on optimization
methods on manifolds, e.g.~Riemannian gradient, subgradient and proximal point
methods~\cite{FO98,FO02,AF05,GH14}, Newton iterations, e.g., see the
book~\cite{AMS08} for algorithms on matrix manifolds. In all these methods
there are three main challenges that emerge for the case of Riemannian
manifolds: (i) there is (in general) no additive group on a Riemannian
manifold, (ii) straight lines are replaced by geodesics, i.e.,~the space itself
is curved and this curvature has to be taken care of, and most importantly for
variational methods, (iii) there is a notion of global convexity only for the
case of Hadamard manifolds. These manifolds have nonpositive sectional
curvature and many tools from optimization have been transferred to these
spaces, see the monograph~\cite{Bac14}.
In this paper we combine for the first time differential operators on finite
weighted graphs and optimization methods on Riemannian manifolds
to construct local and nonlocal methods for manifold-valued data based on
variational models and partial differential equations.

\subsection{Related work}
\label{ss:related_work}
Due to the growing interest in graph-based modeling in the recent years there
have been many sophisticated methods based on nonlocal models and finite
weighted graphs. While one part of the community has concentrated on the theory
of nonlocal methods~\cite{AGP15,BO04,GO08,LPSS15} and the translation of finite
weighted graphs to continuous models~\cite{GS16,GSBLB16,LBB08}, other groups
investigated graph operators and their applications in the discrete setting.
As we are proposing a discrete graph framework for manifold-valued data in this
work, we will focus our discussion of related work on the latter research
domain.
Several groups have put their efforts in the translation of differential
operators and related PDEs from the continuous setting to finite weighted
graphs and the study of their discrete properties, such as the Ginzburg-Landau
functional \cite{GB12,CGSRF17}, the Allen-Cahn equation \cite{GGOB14},
variational~$p$-Laplacians \cite{ELB08,KSS12,ETT15} and its application to
spectral clustering~\cite{BH09}, and Hamilton-Jacobi
equations~\cite{M14,TEL16}. In addition to discrete data defined on regular
grids, recent works have investigated the feasibility of modeling data defined
on discrete surfaces and raw point clouds using finite weighted
graphs~\cite{BM16,GS16,LEL14,SOZ16,TEL16}. Especially the latter scenario of
point cloud processing is a difficult task, since there is a-priori no
connectivity given for the acquired points.
In particular, in a previous work with Elmoataz et al.~in \cite{ETT15} we have
studied variants of the graph~$p$-Laplacian and the graph $\infty$-Laplacian
and discussed their relationship to continuous operators such as the
variational $p$-Laplacians, the game~$p$-Laplacian or the
nonlocal~$p$-Laplacian, which occur as models in many domains, e.g., in physics, game theory, biology,
or economy. With respect to image processing and machine learning applications
the graph $p$-Laplacian and the graph $\infty$-Laplacian have been successfully
used for denoising, segmentation, and inpainting of images but also for general
data processing and clustering \cite{ELB08,ETT15,LEL14,SOZ16}.

All the above listed graph-based methods share in common that they can only be
applied to real- or vector-valued data. To the best of our knowledge, there is
no generalization of finite weighted graphs to manifold-valued data so far.
However, for manifold-valued images and 3D data sets, the same tasks arise as
in usual image processing, e.g.,~denoising, inpainting or segmentation. Recently
several works tackled these tasks such as~\cite{BW16,CS13,SC11} for inpainting,
or~\cite{APSS16,BFPS16,BT15} for segmentation of such data.
For denoising the TV approach or Rudin-Osher-Fatemi (ROF)~\cite{ROF92} model was introduced
by~\cite{WDS14,LSKC13} and generalized
to second order methods in~\cite{BLSW14,BBSW16,BW15,BW16}.
Furthermore, for the ROF model
half-quadratic minimization~\cite{BCHPS16} and the Douglas-Rachford
algorithm~\cite{BPS16} have led to a significant increase in computational efficiency. Another approach in~\cite{LPS16} uses second order statistics and employs a nonlocal denoising
method. All the discussed methods share in common that they work on discrete regular grids only, e.g., on pixel or voxel grids.

It is possible to embed every~\(d\)-dimensional manifold into an Euclidean space 
of at most dimension~\(2d\) due to the theorem of Whitney~\cite{Whi36},
or even use an isometric embedding following the theorem of Nash~\cite{Nas56}.
However, a major disadvantage is that this might increase the dimension of the data,
by a factor up to two following Whitney or even more when relying on an isometric
 embedding. While embedding is sometimes easy, e.g., for the sphere
 \(\mathbb S^d\subset\mathbb R^{d+1}\), it is often beneficial to use the 
intrinsic metric from an application point of view. For example for measurements 
on the earth the arc length is needed instead of the Euclidean norm in the
embedding space \(\mathbb R^3\). There are two further problems when
working in the embedding space. First, minimizing within the embedding space can be
 done by adding an enforcing constraint to a variational model, e.g., the indicator function~\(\iota_{\mathcal M}\) being \(0\) within the manifold and~\(+\infty\) otherwise.
These can easily be introduced into real-valued state-of-the-art methods like
alternating direction method of multipliers (ADMM)~\cite{GM76,GM75}. For example in
~\cite{RTKB14,RWTKB14} a matrix-valued TV functional was introduced using the ADMM for 
the classical TV functional extended by a such projection.
However, when considering the example of symmetric positive definite matrices,
the set is open. Hence the constraint can not be enforced by introducing a projection. The authors in the aforementioned papers project onto the \(\varepsilon\)-relaxed set or the closure and observe numerical convergence. However, from the mathematical viewpoint there remains the open question how to model such embeddings.
Furthermore, several methods like infimal convolution~\cite{CL97}, which was 
recently generalized in~\cite{BFPS17,BFPS17a} to Riemannian manifolds, split a 
signal or data set into several parts. When working in the embedding space directly these
parts also have to be regarded in the embedding space and lose their 
manifold-valued interpretation. For these reasons working intrinsically on manifolds is
often beneficial.

\subsection{Own contribution}
In this paper we introduce a novel graph framework for manifold-valued data.
On the one hand our approach generalizes manifold-valued image processing
models to arbitrary neighborhoods and discretizations which are modeled by the
topology of the graph.
Furthermore, this work also introduces a unified framework for both local and
nonlocal methods for manifold-valued data processing.
For local methods our framework introduces the possibility to process data not
only on a pixel grid, but also for the case that the measurements are taken on
surfaces. This surface might be explicitly given, e.g., measurements on the
surface of a sphere or just implicitly by a point cloud.
For the nonlocal case this framework unifies all methods for which vicinity is
defined via the similarity of features, e.g., adaptive filtering methods or
patch-based distances. On the other hand the data items are manifold-valued and
hence a huge variety of data measurement modalities are incorporated in this
framework.
The proposed framework is very flexible as it consists of three independent
parts, namely manifold specific operations, graph construction and operators,
and numerical solvers. Each of these parts can be exchanged and enhanced
without effecting the other modules. 
By introducing a comprehensive mathematical framework we also derive a notion
of an anisotropic as well as an isotropic manifold-valued
graph \(p\)-Laplacian.
For the special case of the manifold being an Euclidean space, the operators in
our framework simplify to the real- and vector-valued graph \(p\)-Laplacians. 
Thus, our approach can be interpreted as a generalization of well-known
discrete graph operators to the manifold-valued setting.
From the view of variational modeling in convex optimization we further derive
optimality conditions for a family of energy functionals corresponding to
denoising of manifold-valued data. Finally, by looking at the connection of
these optimality conditions to parabolic PDEs, we derive a simple algorithm to compute minimizers of the respective energy
functionals. We demonstrate the flexibility and performance of our approach on a variety of
synthetic as well as real world applications and solve all these with the
same universal numerical algorithm.

\subsection{Organization}
In Section~\ref{s:preliminaries} we first introduce the basic notation of
finite weighted graphs and the needed tools from differential geometry to perform
data processing on Riemannian manifolds.
Then we introduce our graph framework for manifold-valued functions in 
Section~\ref{s:our_framework} and define discrete differential operators to 
perform a huge variety of processing tasks on manifold-valued data. 
Subsequently, we discuss a special class of problems in
Section~\ref{s:numerics}, namely variational denoising tasks.
We formulate the respective energy functionals, derive necessary optimimality
conditions, and present two simple numerical schemes to compute solutions of
these problems.
In Section \ref{s:applications} we demonstrate the performance of our
approach on several synthetic as well as real-world applications.
Finally, we conclude our paper with a discussion and a short outlook to future
work in Section~\ref{s:conclusion}.

\section{Mathematical foundation}
\label{s:preliminaries}
In this section we introduce the mathematical concepts and notations needed to
define graph operators for manifold-valued data. We begin by giving the basic
definitions of finite weighted graphs in Section~\ref{ss:graphs_real}.
Subsequently, in Section~\ref{ss:prel_man} we introduce the necessary notation
from differential geometry to describe (complete) Riemannian manifolds,
tangent spaces, and functions defined on a manifold.
%
%
\subsection{Finite weighted graphs}
\label{ss:graphs_real}
Finite weighted graphs are present in many different fields of
research, e.g., image processing~\cite{ELB08,CGSRF17,ZB11},
machine learning~\cite{BM16,GSBLB16,ZB11},
or network analysis~\cite{LC12,M14,SNFOV13} as they allow to model and process
discrete data of arbitrary topology.
Furthermore, one is able to translate variational methods and partial
differential equations to finite weighted graphs and apply these to both local 
and nonlocal problems in the same unified framework.
Although their exact description is application dependent, there exists a widely
used consent of basic concepts and definitions for finite weighted graphs in the literature~\cite{ELB08,GGOB14,GO08}.
In the following we recall these basic concepts and the respective mathematical 
notation, which we will need to introduce a general graph framework for 
manifold-valued functions in Section~\ref{s:our_framework} below.

\begin{definition}[Finite weighted graph]
A \emph{finite weighted (directed) graph}~$G$ is defined as a
triple~$G = (V, E, w)$ for which
\begin{itemize}
  \item $V = \{1, \ldots, n\}, n\in\mathbb{N}$, is a finite set of indices 
  denoting the \emph{vertices},
  \item $E \subset V \times V$ is a finite set of (directed) \emph{edges}
  connecting a subset of vertices,
  \item $w \colon E \rightarrow \mathbb{R}^+$ is a nonnegative
  \emph{weight function} defined on the edges of the graph.
\end{itemize}
\end{definition}
Typically, for given application data each \emph{graph vertex}~$u \in V$
models an entity in the data structure, e.g., elements of a finite set, pixels 
in an image, or nodes in a network. Due to the abstract nature of the graph 
structure it is also possible to identify sets of entities by a vertex, e.g., 
when summarizing vertices in hierarchical graphs \cite{HLE13}.
It is important to distinguish between abstract data entities modeled by graph
vertices and their respective characteristics, which will be modeled by vertex
functions defined below.
A \emph{graph edge} \((u,v)\in E\) between a start node \(u\in V\) and an end
node \(v\in V\) models certain relationships between two entities, e.g., 
geometric neighborhoods, interactions, or similarity relationships, depending on the 
given application. In our case, we consider \emph{directed} edges,
i.e.,~\((u,v)\in E \nLeftrightarrow (v,u) \in E\) in general.

\begin{definition}[Neighborhood]
A node $v \in V$ is called a \emph{neighbor} of the
node~$u \in V$ if there exists an edge $(u,v) \in E$. We abbreviate this
as $v \sim u$, which reads as “\(v\) is a neighbor of \(u\)”.
If on the other hand \(v\) is not a neighbor of \(u\), we use~\(v\not\sim u\).
We further define the \emph{neighborhood}~\(\mathcal N(u)\) of a
vertex~$u \in V$ as $\mathcal{N}(u) \coloneqq \{ v\in V \colon v \sim u \}$.
The \emph{degree} of a vertex $u \in V$ is defined as the amount of its 
neighbors $\operatorname{deg}(u) = \lvert \mathcal{N}(u)\rvert$.
\end{definition}

Finally, the \emph{weight function} is the central element to model the
significance of a relationship between two connected vertices with respect to
an application dependent criterion. In many cases the weight function is chosen
as similarity function based on the characteristics of the modeled entities,
i.e., by the evaluation of vertex functions as defined below. Then the weight
function $w$ takes high values for important edges, i.e., high similarity of
the involved vertices, and low values for less important ones. In many
applications one normalizes the values of the weight function by $w \colon E
\rightarrow [0,1]$.

A natural extension of the weight function to the full set~$V \times V$ is
given by setting~$w(u,v) = 0$, if~$v \not\sim u$ or~$u = v$ for any~$u,v \in V$.
In this case the graph edge set can be characterized
as~$E = \{ (u,v) \in V \times V \colon w(u,v) > 0 \}$. 

In many cases it is preferable to use \emph{symmetric} weight
functions, i.e., \(w(u,v)=w(v,u)\). This also implicates
that~$v \sim u \Rightarrow u \sim v$ holds for all $u,v \in V$.
Hence, all directed graphs with symmetric weight function discussed in this
paper could be interpreted as~\emph{undirected} graphs, though in the following
it is important that each edge~\((u,v)\in E\) has a start node \(u\) which is
different from its end node~\(v\).
%
%
%
\subsection{Riemannian Manifolds}
\label{ss:prel_man}
We denote by~\(\mathcal M\) a \emph{complete, connected, \(m\)-dimensional
Riemannian manifold} and its Riemannian metric
by~\(\langle \cdot,\cdot\rangle_x \colon
  \T_x\mathcal M\times \T_x\mathcal M\to\mathbb R\),
where~\(\T_x\mathcal M\) is the tangent space at \(x \in \M\).
We furthermore denote by~\(\lVert\cdot\rVert_x\) the induced norm and
by~\(\T\mathcal M\) the disjoint union of all tangent spaces called
the \emph{tangent bundle}.
In the following we introduce the necessary notations and theory in order to
extend algorithms and methods for real-valued functions on finite weighted graphs to
the case of manifold-valued functions. Further details on this
subject can be found, e.g., in the textbooks on
manifolds~\cite{Jost11,dCar92,AMS08}.

The completeness of \(\M\) implies that any two points \(x,y\in\mathcal M\) can
be joined by a (not necessarily unique) shortest curve. We denote such a curve
by~\(\gamma_{\overset{\frown}{x,y}}\colon [0,L] \to \mathcal M\), where \(L\)
is the length of a shortest curve. We denote by~\(\dot\gamma(t) \coloneqq
\frac{d}{dt}\gamma(t)\in T_{\gamma(t)}\mathcal M\) the derivative vector field
of the curve. We further require, that the curve is parametrized with constant
speed, i.e.,~\(\lVert \dot\gamma_{\overset{\frown}{x,y}}(t)
\rVert_{\gamma_{\overset{\frown}{x,y}}(t)} = 1\), \(t\in[0,L]\). This
generalizes the idea of shortest paths from the Euclidean space~\(\mathcal
M=\mathbb R^m\), i.e., straight lines, to a manifold.
  
Let~\(D\colon \Gamma(\T\mathcal M)^2 \to \Gamma(\T\mathcal M)\) denote the
covariant derivative corresponding to the Levi-Citiva
connection~\cite[Theorem~3.6]{dCar92}, where \(\Gamma(\T\mathcal M)\) is the set
of all differentiable vector fields on~\(\mathcal M\).
A geodesic also fulfills that the covariant
derivative~\(D_{\dot\gamma(t)}\dot\gamma = \frac{D}{dt}\dot\gamma\)
of the (tangent) vector field~\(\dot\gamma(t)\in T_{\gamma(t)}\mathcal M\)
vanishes. However, (shortest) geodesics are not the only curves
possessing this property. For example let us fix on the two-dimensional
sphere~\(\mathcal M = \mathbb S^2 \coloneqq \{x\in\mathbb R^3
  : \lVert x\rVert = 1\}\) two points \(x,y\in\mathbb S^2\)
not being antipodal.
Then there exists a unique great circle containing the points \(x,y\).
One of the arcs is the shorter one and yields the
geodesic~\(\gamma_{\overset{\frown}{x,y}}\).
Still both great arcs (seen as curves on the manifold) have a vanishing
covariant derivative. In the literature, the non-shortest curves with vanishing
covariant derivative \(D\dot\gamma=0\) are often also called geodesics. In the
following we refer to geodesics as always being shortest ones. The shortest
geodesic might also not be unique, e.g.,~for two antipodal points~\(x,y\)
on~\(\mathbb S^2\) both connecting great half circles are geodesics and both are 
shortest ones.
In this case we will still write \emph{the geodesic} meaning that any of the shortest
geodesics is meant.

The length of the shortest geodesic induces
the \emph{geodesic distance}~\(d_{\mathcal M}
  \colon\mathcal M\times\mathcal M\to\mathbb R\).
From the Theorem of Hopf and Rinow, cf.~\cite[Theorem 1.7.1]{Jost11}, we obtain that for
some \(\varepsilon>0\) and \(\xi\in T_x\mathcal M\) there exists a unique
geodesic~\(\gamma_{x,\xi}\colon (-\varepsilon,\varepsilon)\to\mathcal M\)
fulfilling \(\gamma_{x,\xi}(0) = x\) and~\(\dot\gamma_{x,\xi}(0) = \xi\in
\T_x\mathcal M\). Furthermore, by the uniqueness the Hopf-Rinow Theorem states
that the uniqueness is given for any \(\varepsilon\) and hence
with~\(\varepsilon=1\) and~\(y \coloneqq \gamma_{x,\xi}(1)\) we see that the
geodesic~\(\gamma_{\overset{\frown}{x,y}}\) is just a reparametrization
of~\(\gamma_{x,\xi}\), namely \(\gamma_{x,\xi}(t) =
\gamma_{\overset{\frown}{x,y}}(\lVert \xi \rVert_x t)\).

\begin{definition}[Exponential and logarithmic maps]
\label{def:exp_log}
The \emph{exponential map}~\(\exp_x\colon \T_x\mathcal M\to\mathcal M\) is 
defined as~\(\exp_x(\xi) = \gamma_{x,\xi}(1)\).
Furthermore, let~\( r_x\in\mathbb R^+\) denote the injectivity radius,
i.e., the largest radius such that \(\exp_x\) is injective for all \(\xi\)
with~\(\lVert\xi\rVert_x < r_x\).
Furthermore, let \[
  \mathcal D_x \coloneqq
  \{ y\in\mathcal M :
    y = \exp_x\xi, \text{ for some }\xi\in T_x\mathcal M \text{ with } \lVert\xi\rVert_x<r_x\}\ .
\]
Then the inverse map~\(\log_x\colon \mathcal D_x\to \T_x\mathcal M\)
is called the~\emph{logarithmic map} and maps a
point~\(y=\gamma_{x,\xi}(1)\in\mathcal D_x\) to~\(\xi\). 
\end{definition}
By the properties of the exponential
map it holds that~\(d_{\mathcal M}(x,y) = d_{\mathcal M}(x,\gamma_{x,\log_xy}(1)) = \lVert\log_xy\rVert_x\).

For a differentiable vector field~\(V\in\Gamma(\T\mathcal M)\) and a smooth
curve~\(c(t)\) on the manifold we can further define the \emph{vector field
along~\(c\)} by~\(V(t) \coloneqq V(c(t))\).
The solution~\(W\in\Gamma(\T\mathcal M)\) of
\[
  D_{V(t)}W(t) = 0
\]
is called \emph{parallel transport} of~\(W(0) \in T_{c(0)}\mathcal M\) along
the curve~\(c\).
We introduce the special notation~\(\PT_{x\to y}\colon \T_x\mathcal M
  \to \T_y\mathcal M\) for the \emph{parallel transport} of a
vector~\(\nu\in \T_x\mathcal M\) to~\(y\) along a geodesic
connecting~\(x\) and~\(y\),
  i.e.,~\(\PT_{x\to y}(\nu)
  \coloneqq W(d_{\mathcal M}(x,y))\in T_y\mathcal M\) is the
parallel transport~\(W\) with~\(W(0)=\nu\)
and~\(c(t)=\gamma_{\overset{\frown}{x,y}}\) evaluated at the end
point~\(t=d_{\mathcal M}(x,y)\) of the geodesic, i.e., at~\(y\).
The above introduced definitions are collectively illustrated in
Figure~\ref{fig:manifold}.
\begin{figure}\centering
  \includegraphics{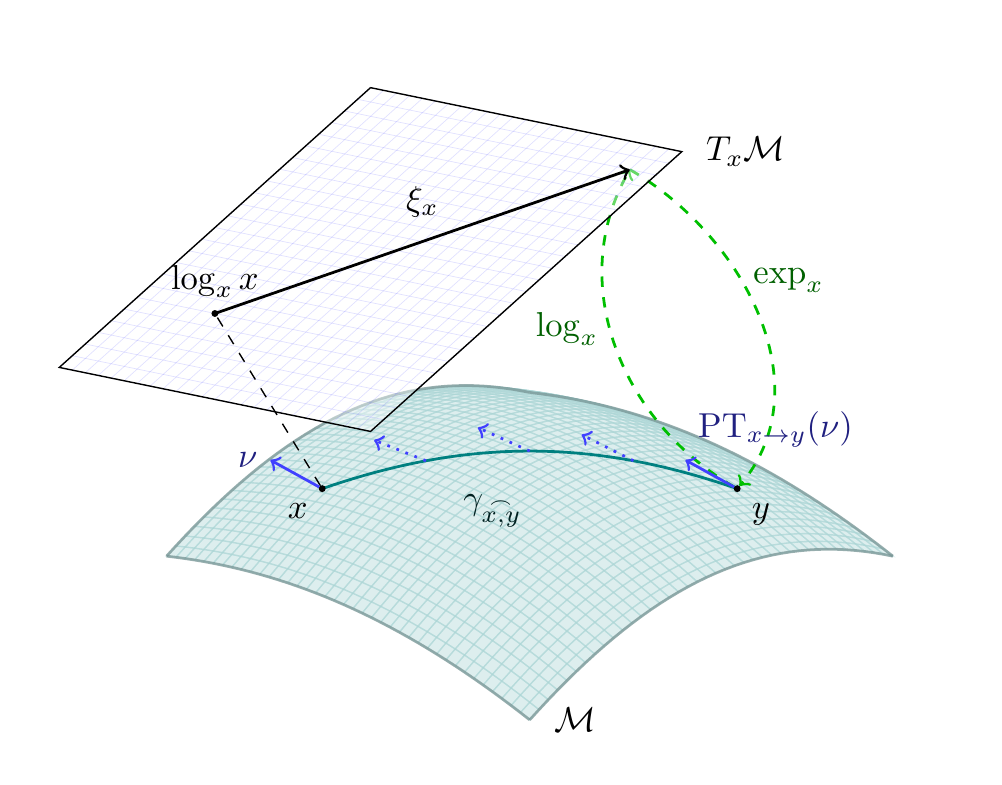}
  \caption{Illustration of the geodesic~\(\gamma_{\overset{\frown}{x,y}}\), the tangent plane \(\T_x\mathcal M\), the logarithmic, the exponential map and the parallel transport \(\PT_{x\to y}\) on a Riemannian
  manifold~\(\mathcal M\). Note that the tangent plane \(\T_x\mathcal M\) is in fact tangent to the manifold and only moved perpendicular to the surface (indicated by the dashed line) for presentation.}
  \label{fig:manifold}
\end{figure}
Whenever the manifold is not explicitly given, i.e.,~there is only an approximate
description available, the whole following framework can also be employed using
numerical approximations of the exponential and logarithmic maps, hence also using only an
approximation to the distance~\(d_{\mathcal M}\) and the parallel transport on
the manifold.

Finally, we have to generalize the notion of convexity to a Riemannian manifold. 
A set \(\mathcal C\subset\mathcal M\) is called \emph{convex} if for any two
points~\(x,y\in\mathcal M\) all minimizing geodesics~\(\gamma_{\overset{\frown}{x,y}}\) lie in \(\mathcal C\).
Such a set~\(\mathcal C\) is called \emph{weakly convex} if for all \(x,y\)
there exists a geodesic~\(\gamma_{\overset{\frown}{x,y}}\) lying completely
in~\(\mathcal C\).
Note that convexity might be a local phenomenon on certain manifolds: on the sphere \(\mathbb S^2\) there exists no convex set~\(\mathcal C\) larger than an
open half-sphere.
Finally, a function \(f\colon\mathcal M \to\mathbb R\) is called \emph{(weakly) convex} on a~(weakly) convex set~\(\mathcal C\) if for all \(x,y\in \mathcal C\)
the composition
\[
  f(\gamma_{\overset{\frown}{x,y}}(t)),\qquad t\in [0,d_{\mathcal M}(x,y)],
\]
is a convex function. This notion of convexity is sometimes called
\emph{geodesic} convexity.
\section{A graph framework for manifold-valued functions}
\label{s:our_framework}
In this section we propose a graph-based framework for processing manifold-valued data. This approach allows us to translate successful variational models and partial differential equations to manifolds. Furthermore, the graph framework allows us to unify local as well as nonlocal methods for manifold-valued data. This leads to a comprehensive discussion of related models and elegant numerical solutions.
We begin by introducing the concept of manifold-valued vertex functions and edge functions which take values in respective tangent spaces in Section \ref{ss:calculus_manifold}. In this setting we also discuss the characteristics of the corresponding function spaces. One fundamental definition will be the notion of the discrete directional derivative for manifold values, which we introduce in Section \ref{ss:difference_operators_manifold}. Based on this we introduce first order differential operators, i.e., the weighted local discrete gradient and the weighted local divergence operator. We use these definitions to derive higher order differential operators for manifold-valued vertex functions in Section \ref{ss:p-Laplacians_manifold}, namely a family of isotropic and anisotropic graph $p$-Laplacians. Finally, we discuss the special case in which the manifold $\mathcal M$ is simply a Euclidean space $\mathbb{R}^m$. In this setting we show that our proposed framework is a generalization of well-known graph operators from the literature.

\subsection{Discrete calculus for manifold-valued vertex functions
and tangent edge functions}
\label{ss:calculus_manifold}
In the following we introduce the basic definitions and observations for manifold-valued vertex functions and tangent edge functions on a finite weighted graph \(G = (V, E, w)\).
\begin{definition}[Manifold-valued vertex function]
\label{def:vertex_function}
Let $G = (V, E, w)$ be a finite weighted graph and let $\mathcal M$ be
a complete, connected,~$m$-dimensional Riemannian manifold. We define a \emph{manifold-valued
  vertex function}~$f$ as
\begin{equation*}
\begin{split}
  f \colon V \ &\rightarrow \ \mathcal M,\\
  u \ &\mapsto \ f(u).
\end{split}
\end{equation*}
For a given vertex function \(f\) we can define the \emph{finite, disjoint union
of the induced tangent spaces}~\(\T_f \mathcal M\) as
\begin{equation*}
\T_f \mathcal M
  \ \coloneqq \ \dot{\bigcup\limits_{u\in V}} \T_{f(u)} \mathcal M
  \, \subset \, \T\mathcal M.
\end{equation*}
\end{definition}
In the following we discuss properties of the function spaces of
manifold-valued vertex functions. We define a metric for vertex functions based
on mappings into respective tangent spaces employing the logarithmic map.

In the following we require that two adjacent data items \(f(u),f(v), (u,v)\in E\) posess the property that \(f(v)\in \mathcal D_{f(u)}\) and vice versa, such that the logarithmic map \(\log_{f(u)}f(v)\) is well defined.
\begin{definition}[Function spaces on vertex functions]
\label{def:function_spaces_manifold}
Let $G = (V, E, w)$ be a finite weighted graph.
We introduce a metric space of manifold-valued vertex functions denoted
\\
as~$(\mathcal H(V; \mathcal M), d_{\mathcal H(V; \mathcal M)}(\cdot,\cdot))$ on
a set of admissible functions
\begin{equation*}
\mathcal H(V;\mathcal M) \ \coloneqq \ \{f \colon V \rightarrow \mathcal M \bigl| f(u) \in \mathcal D_{f(v)}\text{ for all } (u,v)\in E\} \ ,
\end{equation*}
with the associated metric, which is for two vertex functions $f,g \in \mathcal
H(V;\mathcal M)$ given by
\begin{equation*}
 d_{\mathcal H(V;\mathcal M)}(f,g) \ \coloneqq \
  \biggl( \sum_{u\in V} d_{\mathcal M}^2(f(u),g(u)) \biggr)^{\frac{1}{2}}.
\end{equation*}
Here, $d_{\mathcal M}(\cdot, \cdot) \colon \mathcal M \times \mathcal M
\rightarrow \mathbb{R}_{\geq 0}$ denotes the geodesic distance between
two points on the manifold $\mathcal M$.
In fact this is the standard metric on the product manifold \(\mathcal M^n\).
\end{definition}
We will refer to functions \(f\in \mathcal H(V;\mathcal M)\) as data with neighboring items fulfilling the \emph{locality property}, i.e.~neighboring data items with respect to the set of edges are within their injectivity radii.
If this is not fulfilled for data given in practice, a denser sampling, i.e.~larger graph with smaller edge weights is required. 

Based on the notion of vertex functions mapping from the set of vertices~\(V\)
into the manifold \(\mathcal M\) we are able to define edge functions mapping
from the set of edges \(E\) into the respective tangent spaces. The edge
functions in the vector valued setting, see, e.g.~\cite{ETT15}, represent
finite differences. The analogue of a discrete difference as an approximation
of the derivative is given by the logarithmic map on a manifold. The
corresponding values are therefor given in the tangent bundle. Hence the
definition of edge functions is only meaningful with respect to an associated
vertex function as this induces the corresponding tangent spaces.
\begin{definition}[Tangential edge function]
\label{def:manifold_edge_function}
Let~$G = (V, E, w)$ be a finite weighted graph. We define a \emph{tangent
  edge function}~$H_f$ with respect to a manifold-valued vertex
function~$f\in\mathcal H(V;\mathcal M)$ as
\begin{equation*}
\begin{split}
  H_f \colon E \ &\rightarrow \ \T_f \mathcal M,\\
  (u,v) \ &\mapsto \ H_f(u,v) \in \T_{f(u)}\mathcal M.
\end{split}
\end{equation*}
\end{definition}
Note that these tangent edge functions require directed graphs due to the
tangent space \(\T_{f(u)}\mathcal M\) of the function value \(f(u)\) at the
start point \(u\in V\) of the edge involved. For undirected graphs, we consider
their directed analogon with the symmetry property \((u,v)\in E \Leftrightarrow
(v,u)\in E\) and can still investigate their tangent edge functions thereon.

We can introduce a family of measures on the values of an edge function~$H$
at an edge~$(u,v) \in E$ by using the respective vector norms which can be
associated to the tangent space $\T_{f(u)}\mathcal M$.

Let~\(\mathcal H(E; \T_f \mathcal M)\) denote the space of functions for
which~\(H(u,v) \in T_{f(u)}\mathcal M\) holds true for any edge \((u,v)\in E\).
Then we define a family of norms, namely for \(p,q\geq 1\), by
  \begin{align}\label{eq:HETfMnorm}
    \lVert H_f\rVert_{\ell_{p,q}(E; \T_f \mathcal M)} \ \coloneqq \
	\Biggl(
    \frac{2}{p}
    \sum_{u\in V}
      \biggl(
        \sum_{v\sim u}
          \lVert H_f(u,v) \rVert_{f(u)}^q
       \biggr)^{\frac{p}{q}}
    \biggr)^{\frac{1}{p}} \ .    
  \end{align}
  We further introduce the short hand notation~\(\lVert\cdot\rVert_{p,q}
  \coloneqq \lVert \cdot\rVert_{\ell_{p,q}(E; \T_f \mathcal M)}\)
  if~\(\mathcal M\), \(E\), and \(f\) are clear from the context.
  Note that this includes the special cases of the~\emph{isotropic} norm
  \(\lVert\cdot\rVert_{p,2}\) and the \emph{anisotropic} norm~\(\lVert\cdot\rVert_{p,p}\).
  For the special case $p=q=2$ we obtain a
  Hilbert space, which is equipped with the inner product
\begin{equation*}
\langle H_f,G_f\rangle_{\mathcal H(E; \T_f \mathcal M)}\ \coloneqq \ 
\sum_{(u,v)\in E} \langle H_f(u,v), G_f(u,v) \rangle_{f(u)} \ .
\end{equation*}
In order to discuss edge functions independently of an associated vertex function $f \in \mathcal H(V; \M)$ we define a general Hilbert space as union of all Hilbert spaces of admissible edge functions by
\begin{equation*}
\mathcal H(E) \coloneqq \bigcup_{f\in \mathcal H(V;\mathcal M)} \mathcal H(E; \T_f \mathcal M) \ .
\end{equation*}
Then one can interpret an edge function $H \in \mathcal H(E)$ as a function~$H \colon E \rightarrow \T \mathcal M$ with the restriction, that all~\(H(u,v)\), \(v \sim u\), are values in the same tangent space.

\begin{definition}[Local variation of an edge function]\label{def:locVarMeas}
Let $G = (V, E, w)$ be a finite weighted graph. We define the~\emph{local variation} of a tangent edge function \(H_f\in \mathcal H(E;\T_f\mathcal M)\)
at a vertex \(u\in V\) as
\begin{align}\label{eq:LocVarMeas}
	\lVert H_f \rVert^q_{q,f(u)}
	\ &\coloneqq \
	\sum_{v\sim u}
	\lVert H_f(u,v)\rVert_{f(u)}^q
  ,\quad q \geq 1.
\end{align}
\end{definition}
Note that this norm represents the summand for a fixed \(u\in V\) within the \(\ell_{p,q}\)-norm~~\eqref{eq:HETfMnorm}.

If we further employ the parallel transport,
i.e.,~for~\(F,G\in\mathcal H(E)\) with~\(F\in\mathcal H_f(E)\),
\(G\in\mathcal H_g(E)\) we have for~\(\tilde F(u,v) \coloneqq
  \PT_{f(u)\to g(u)}F(u,v)\)
that~\(\tilde F\in \mathcal H_g(E)\). This way, we may also
obtain a Hilbert space.

We define a symmetric mapping on the space of vertex
functions~\(\mathcal H(V;\mathcal M)\) by
\begin{align*}
	\langle f,g \rangle
	\coloneqq
	\sum_{u\in V}
	\biggl(
		\sum_{v\sim u}
			\bigl\langle \log_{f(u)}f(v), \PT_{g(u)\to f(u)}\log_{g(u)}g(v)\bigr\rangle_{f(u)}
	\biggr).
\end{align*}
The symmetry can be seen by noting that both sum run over the same set of nodes for \(\langle g,f\rangle\) and a single summand reads \(\langle \log_{g(u)}g(v), \PT_{f(u)\to g(u)}\log_{f(u)}f(v)\rangle_{g(u)}\).
This can be be parallel transported to \(\T_{f(u)}\mathcal M\) which does not
change its value. The symmetry follows from the symmetry of the inner product in \(\T_{f(u)}\mathcal M\).

We further equip the space \(\mathcal H(V;\mathcal M)\) with a norm, which is for \(p=2\) induced by the symmetric map
\[
	\lVert f\rVert_p^p
	\coloneqq
		\sum_{u\in V}
		\biggl(
			\sum_{v\sim u}
				\bigl\langle \log_{f(u)}f(v), \log_{f(u)}f(v)\bigr\rangle_{f(u)}
		\biggr)^{\frac{p}{2}},\qquad p\geq 1.
\]
Note that this distance can also be employed for arbitrary functions \(f\), i.e., also for data not fulfilling the locality property by employing the Riemannian distance \(d_{\mathcal M}(f(u),f(v))^p\) instead of the inner products of the logarithmic maps. Hence indeed the sum of all distances \(d_{\mathcal M}(f(u),f(v))\) between all nodes for \(p=2\), and thus induces a proper norm on the set of vertex functions.
However, due to the missing linearity, this norm is not induced by an inner product. We still employ the symbol \(\langle\cdot,\cdot\rangle\) for this symmetric map for the sake of simplicity.

Finally, we can also look at tangent vertex functions \(H\in \mathcal H(V;T_f\mathcal M)\) associated with \(f\in\mathcal H(V;\mathcal M)\), i.e.,
\[
  \mathcal H(V;T_f\mathcal M)
  \coloneqq
  \{
    H\colon V\to T_f\mathcal M,\ H(u)\in T_{f(u)}\mathcal M
  \},
\]
which can be equipped with the metric and its induced norm from the tensor
product~\(\T_f\mathcal M\) of the tangent spaces~\(T_{f(u)}\mathcal M\), \(u\in V\).
Similarly we use \(\mathcal H(V;\T\mathcal M)\).

\subsection{First-order difference graph operators}
\label{ss:difference_operators_manifold}
After defining all necessary functions, function spaces, and measures
in the last section,
we can now proceed to introduce novel first-order graph operators for
manifold-valued data.

\begin{definition}[Weighted directional derivative]
\label{def:weighted_difference_manifold}
Let $G = (V, E, w)$ be a finite weighted graph.
We define the \emph{weighted directional
derivative}~$\partial_v f \colon V \rightarrow \T_{f(u)} \mathcal M$
of a manifold-valued vertex function~$f \in \mathcal H(V;\mathcal M)$
at~$u\in V$ in direction of another vertex $v\in V$ as
\begin{equation*}
\partial_vf(u) \coloneqq \sqrt{w(u,v)}\log_{f(u)}f(v) \ .
\end{equation*}
Note that we set \(\partial_vf(u)=0\), whenever \(v\not\sim u\).
From this definition we can directly deduce the following lemma.
\end{definition}
\begin{lemma}[Properties of the weighted directional derivative]
The weighted directional
\\
derivative is
\begin{enumerate}[label=\normalfont\roman*)]
\item reflexive, i.e., \( \partial_vf(u)=0\) for $u=v$
\item anti-symmetric under parallel transport,
i.e., it holds
\begin{align*}
  \partial_vf(u) &= \sqrt{w(u,v)}\log_{f(u)}f(v)
    = -\PT_{f(v)\to f(u)}\sqrt{w(u,v)}\log_{f(v)}f(u)\\
    &= -\PT_{f(v)\to f(u)}(\partial_uf(v)).
\end{align*}
\end{enumerate}
\end{lemma}
We also define another variant of the weighted directional derivative 
as an edge function.
\begin{definition}[Weighted local gradient]
\label{def:weighted_gradient_manifold}
Let $G = (V, E, w)$ be a finite weighted graph.
We define the \emph{weighted local gradient}~$\nabla \colon
  \mathcal H(V;\mathcal M) \to \mathcal H(E)$ of
a manifold-valued vertex function~$f\in \mathcal H(V;\mathcal M)$ as
\begin{equation*}
\nabla f(u,v) \ \coloneqq \ \partial_v f(u) \ = \ \sqrt{w(u,v)}\log_{f(u)}f(v) \ ,\quad (u,v)\in E.
\end{equation*}
\end{definition}
Clearly, we see that $\nabla f \in \mathcal H(E;\T_f \mathcal M)$. Furthermore, the
following Theorem derives a relationship between the weighted local gradient \(\nabla f\) and a corresponding edge function \(H_f\).

\begin{theorem}\label{th:relationship_divergence}
Let $G = (V, E, w)$ be a finite weighted graph with a symmetric edge set \(E\),
i.e.~\((u,v)\in E\Leftrightarrow(v,u)\in E\) and an arbitrary, not necessarily
symmetric weight function \(w\). Let further $f \in \mathcal H(V;\mathcal M)$
be a vertex function and $H_f \in \mathcal H(E; \T_f \mathcal M)$ a tangent
edge function. Then we have the following relationship
\begin{equation*}
\begin{split}
\langle \nabla f, H_f &\rangle_{\mathcal H(E; \T_f \mathcal M)} \ = \\
&\sum_{u\in V} \sum_{v\sim u} \langle \log_{f(u)}f(v), \frac{\sqrt{w(u,v)}}{2} H_f(u,v) - \frac{\sqrt{w(v,u)}}{2} \PT_{f(v)\rightarrow f(u)}(H_f(v,u)) \rangle_{f(u)}.
\end{split}
\end{equation*}
\end{theorem}
\begin{proof}
  Starting with the left hand side we compute
\begin{align*}
&\left\langle \nabla f, H_f \right\rangle_{\mathcal H(E; \T_f \mathcal M)} \ = \ \sum_{(u,v)\in E} \left\langle \nabla f(u,v), H_f(u,v) \right\rangle_{f(u)} \\
 &= \ \frac{1}{2}\sum_{(u,v)\in E} \bigl\langle \sqrt{w(u,v)} \log_{f(u)} f(v), H_f(u,v)\bigr\rangle_{f(u)}\\
 &\hspace{2cm} + \frac{1}{2}\sum_{(u,v)\in E} \bigl\langle \sqrt{w(u,v)} \log_{f(u)} f(v), H_f(u,v)\bigr\rangle_{f(u)} \\
 &= \ \sum_{(u,v)\in E} \Bigl\langle \log_{f(u)} f(v), \frac{\sqrt{w(u,v)}}{2}H_f(u,v)\Bigr\rangle_{f(u)} \\ 
 & \hspace{2cm} + \sum_{(v,u)\in E} \Bigl\langle \PT_{f(v)\rightarrow f(u)}(\log_{f(v)} f(u)), \PT_{f(v)\rightarrow f(u)}\left(\frac{\sqrt{w(v,u)}}{2} H_f(v,u)\right)\Bigr\rangle_{f(u)} \\
 &= \ \sum_{(u,v)\in E} \Bigl\langle \log_{f(u)} f(v), \frac{\sqrt{w(u,v)}}{2}H_f(u,v)\Bigr\rangle_{f(u)} \\ 
 & \hspace{2cm} + \sum_{(v,u)\in E} \Bigl\langle -\log_{f(u)} f(v)), \frac{\sqrt{w(v,u)}}{2} \PT_{f(v)\rightarrow f(u)} \left( H_f(v,u) \right)\Bigr\rangle_{f(u)} \\
  &=
  \ \sum_{u\in V} \sum_{v\sim u} \langle \log_{f(u)}f(v), \frac{\sqrt{w(u,v)}}{2} H_f(u,v) - \frac{\sqrt{w(v,u)}}{2} \PT_{f(v)\rightarrow f(u)}(H_f(v,u)) \rangle_{f(u)}\ .
\end{align*}
\end{proof}

We use the relationship deduced in Theorem~\ref{th:relationship_divergence} to motivate the definition of a local divergence operator for the respective tangent spaces. Due to the fact that the discussed manifold-valued vertex functions are not mapping into a vector space, we are not able to deduce the divergence as an adjoint operator of the gradient operator as it is the case, e.g., for Euclidean spaces.

\begin{definition}[Weighted local divergence]
\label{def:divergence_manifold}
We define a \emph{weighted local divergence operator} \(\div \colon \mathcal H(E)\to \mathcal H(V;T\mathcal M)\) for a tangent edge function \(H_f \in \mathcal H(E; \T_f \mathcal M)\) associated to~\(f\in \mathcal H(V;\mathcal M)\) at a vertex \(u \in V\) as
\[
  \div H_f(u) \ \coloneqq \
    \frac{1}{2}\,
    \sum_{v\sim u}
      \sqrt{w(v,u)}
      \PT_{f(v)\to f(u)}H_f(v,u) - \sqrt{w(u,v)}H_f(u,v)\ .
\]
\end{definition}

\begin{remark}
Let~$G = (V,E,w)$ be a finite weighted graph with symmetric weight function,
i.e.,~$w(u,v) = w(v,u)$ for all $(u,v) \in E$.
Assuming an edge function~$H_f \in \mathcal H(E; \T_f \mathcal M)$ which is
anti-symmetric under parallel transport,
i.e.,~\(H_f(u,v) = -\PT_{f(v)-f(u)}H_f(v,u)\) for all \((u,v) \in E\),
we obtain a concise representation of the weighted local divergence operator as
\begin{equation*}
  \div H_f(u) \ = \ - \sum_{v\sim u} \sqrt{w(u,v)} H_f(u,v) \ .
\end{equation*}
\end{remark}

\subsection{A family of graph \(p\)-Laplace operators for manifold-valued functions}
\label{ss:p-Laplacians_manifold}
Based on the definitions of the weighted local gradient and the weighted local
divergence in Section~\ref{ss:difference_operators_manifold}, we are able to
introduce a family of graph $p$-Laplace operators for manifold-valued functions.
Note that for~\(p,q<1\) in~\eqref{eq:HETfMnorm} and \(p<1\) in~\eqref{eq:LocVarMeas} we have only quasi-norms; however, we still can
define the following operators for this case. For the sake of simplicity we will discuss all operators in the case of finite weighted graphs with symmetric weight function in the following. However, these operators can easily be derived for arbitrary weight functions using the tools introduced in Section \ref{ss:difference_operators_manifold} above.
\begin{definition}[Graph $p$-Laplacian operators for manifold-valued functions]
Let $G = (V,E,w)$ be a finite weighted graph with symmetric weight function $w$ and let $ 0 < p < \infty$. For a manifold-valued vertex function \(f\in \mathcal H(V; \M)\) we define the \emph{anisotropic graph \(p\)-Laplacian} $\Delta^\mathrm{a}_p \colon \mathcal H(V;\mathcal M) \rightarrow \mathcal H(V;T\mathcal M)$ at a vertex $u \in V$ as
\begin{align*}
	\Delta^\mathrm{a}_{p}f(u)
	\  \coloneqq& \
    \div \bigl(
      \lVert \nabla f\rVert_{f(\cdot)}^{p-2}\, \nabla f
    \bigr)(u)
  \\
  \ =& \
  - \sum_{v\sim u}
	\sqrt{w(u,v)}^p d^{\,p-2}_{\mathcal M}(f(u),f(v))\log_{f(u)}f(v)
\end{align*}
and the \emph{isotropic graph \(p\)-Laplacian} $\Delta^\mathrm{i}_p \colon \mathcal H(V;\mathcal M) \rightarrow \mathcal H(V;T\mathcal M)$ at a vertex $u \in V$ as
\begin{align*}
	\Delta^\mathrm{i}_{p}f(u)
  \ \coloneqq& \
	\div\bigl( 
    \lVert \nabla f \rVert^{p-2}_{2,f(\cdot)} \nabla f
  \bigr)(u)
  \\
  =& \
    - \,
  \Bigl(
    \sum_{v\sim u}  w(u,v) \, d^2_{\mathcal M}(f(u),f(v))
  \Bigr)^{\frac{p-2}{2}}
  \sum_{v\sim u}
  w(u,v)
  \log_{f(u)}f(v) \ .
\end{align*}
For the special case of \(p=2\) we obtain in both above definitions an operator $\Delta \colon \mathcal H(V;\mathcal M) \rightarrow H(V;T_f\mathcal M)$, which we denote as graph Laplacian for manifold-valued functions, by
\begin{equation*}
	\Delta f(u) \ \coloneqq \ \Delta^\mathrm{i}_2 f(u) \ = \ \div(\nabla f) (u) \ = \ - \sum_{v\sim u} w(u,v)\log_{f(u)}f(v) \ .
\end{equation*}
\end{definition}

\subsection{Special case of $\mathcal M = \mathbb{R}^m$}
\label{ss:flat_manifold}
In the following we discuss the special case in which the Riemannian manifold $\mathcal M$ is simply the Euclidean space $\mathbb{R}^m, m\in\mathbb{N}$. This setting is ubiquitous in many applications such as image and point cloud processing. As $\mathcal M$ has curvature zero everywhere we can demonstrate that the introduced calculus from Subsection \ref{ss:p-Laplacians_manifold} leads to equivalent definitions of well-known graph operators from the literature, e.g., as used in \cite{ELB08,ETT15,GO08}. Thus, the here proposed framework can be interpreted as a generalization of the common calculus for vector-valued vertex functions $f \colon V \rightarrow \mathbb{R}^m$.

For $\M = \mathbb{R}^m$ it gets clear that the respective tangent spaces are equal, i.e., $T_{f(u)}\mathcal M = \mathbb{R}^m$ for all $u \in V$, and hence all mathematical operations are defined globally in this case. In particular, the maps $\log \colon \mathcal M \rightarrow \T\mathcal M$ and $\exp \colon \T\mathcal M \rightarrow \mathcal M$ introduced in Section \ref{ss:prel_man} become globally defined linear operators which correspond to vector subtraction and addition in $\mathbb{R}^m$. Furthermore, the parallel transport reduces to the identity operator because the tangent spaces coincide with the tangent space at the origin, which is itself isomorphic to $\mathbb{R}^m$. These observations lead to a simplification of all introduced operators above, which are be summarized in the following Corollary.
\begin{corollary}
Let $\mathcal M =\mathbb{R}^m, f \in \mathcal H(V; \mathcal M)$ a vertex function, and $\xi \in \T_{f(u)}\mathcal M = \mathbb{R}^m$ a (tangent) vector. For vertices $u,v \in V$ we obtain\\
\begin{enumerate}[label=\normalfont\roman*)]
\item Basic mathematical operations
\begin{equation*}
\label{eq:log_exp_resolved}
\begin{split}
\log_{f(u)}f(v) \ &= \ f(v) - f(u) \in \mathbb{R}^m,\\
\exp_{f(u)}\xi \ &= \ f(u) + \, \xi \in \mathbb{R}^m.\\
\end{split}
\end{equation*}
\item First-order differential operators
\begin{equation*}
\begin{split}
\partial_v f(u) \ = \ \nabla f(u,v) \ = \ \sqrt{w(u,v)}(f(v) - f(u)),\\
\div H(u) \ = \ \frac{1}{2} \sum_{v\sim u} \sqrt{w(u,v)} (H(v,u) - H(u,v)), 
\end{split}
\end{equation*}
\item Graph \(p\)-Laplacian operators
\begin{align*}
\Delta_{p}^\mathrm{a} f (u) \ &= \ - \sum_{v\sim u}
  (w(u,v))^{\frac{p}{2}} \lVert f(v)-f(u)\rVert^{p-2}(f(v)-f(u)), \\
  \Delta_{p}^\mathrm{i} f(u) \ &= \
  -   \Bigl(
    \sum_{v\sim u} w(u,v)
      \lVert f(v)-f(u)\rVert^2
  \Bigr)^\frac{p-2}{2}
  \sum_{v\sim u} ~
 w(u,v) (f(v)-f(u)).
\end{align*}
\end{enumerate}
\end{corollary}
Note that for $\M = \mathbb{R}^m$ the weighted local divergence operator becomes a linear operator and thus the relationship in Theorem \ref{th:relationship_divergence} can be further elaborated leading to the well-known fact that the negative divergence operator is the adjoint graph operator of the weighted gradient, see e.g., \cite{ELB08}. 
\begin{corollary}
Let $G = (V, E, w)$ be a finite weighted graph with symmetric weight function $w$ and let \(\mathcal M = \mathbb{R}^m\). Then for any vertex function $f \in \mathcal H(V)$ and any edge function \(H \in \mathcal H(E)\) the following relationship holds:
\begin{equation}
\langle \nabla f, H \rangle_{\mathcal H(E; \mathbb{R}^m)} \ = \ \langle f, -\div H \rangle_{\mathcal H(V; \mathbb{R}^m)} \ ,
\end{equation}
for which the inner products above are the standard inner products of \(\mathbb{R}^{m\times n\times n}\) and \(\mathbb{R}^{m\times n}\), respectively.
\end{corollary}
%
%
%
%
\section{Formulation of variational problems}
\label{s:numerics}
In this section we develop different algorithms to solve mathematical
problems for manifold-valued vertex functions on graphs. Based on the
introduced graph operators in Section~\ref{s:our_framework} we
translate important PDEs and variational models from continuous mathematics to
graphs in the tradition of~\cite{ELB08,GO08} and for the first time formulate
image processing problems for manifold-valued functions uniformly in a local and
nonlocal setting. These image processing problems include segmentation,
inpainting, and denoising of manifold-valued data. Since in this section we are
interested in demonstrating the feasibility of our approach, we will restrict
ourselves to the latter task. Thus, in the following we will formulate a class
of denoising tasks as variational minimization problems on graphs in
Section~\ref{ss:problem_formulation}. Subsequently, we will discuss necessary
optimality conditions for minimizers of these problems in
Section~\ref{ss:optimality_conditions}, which yield the introduced graph $p$-Laplace operators for manifold-valued functions, $p\geq 1$.
To compute respective minimizers we derive two different numerical schemes in
Section~\ref{ss:numeric_schemes}.
\subsection{Problem formulation}
\label{ss:problem_formulation}
Let \(G = (V,E,w)\) be finite weighted graph with symmetric weight function $w \colon E \rightarrow [0,1]$, i.e., $w(u,v) = w(v,u)$ for all $u,v \in V$. Furthermore, let \(f_0\colon V \to \mathcal M\) be a manifold-valued vertex function which models the given (perturbed) data.
One possibility is to assume that $f_0$ is an observation of the following data formation process, see e.g., \cite{LPS16}:
\begin{equation*}\label{eq:noise-model}
f_0 \ = \ \exp_{\hat{f}}(\varepsilon).
\end{equation*}
Here, $f_0$ is modeled as a noisy variant of the unknown data $\hat{f} \in \mathcal H(V; \mathcal M)$, which is altered by some perturbation~\(\varepsilon \in T_{\hat{f}}\mathcal M\). Recovering the original noise-free data $\hat{f}$ from $f_0$ in~\eqref{eq:noise-model} is a common example of an inverse problem. In order to guarantee the well-posedness of this problem one may incorporate a-priori knowledge about an unknown solution $f$, e.g., smoothness, and pose the problem from a statistical view as a maximum a-posteriori (MAP) estimation method. 
Typically, one aims to find minimizers of a convex energy functional $\E\colon \mathcal H(V; \mathcal M) \rightarrow \mathbb{R}$ consisting of data fidelity and regularization terms, i.e.,
\begin{equation*}
\E(f) \ = \ \mathcal D(f;f_0) \: + \: \mathcal R(f).
\end{equation*}
For further details on deriving convex variational problems from a statistical
modeling perspective we refer to~\cite{STJB13}.

In the following we discuss variational models for denoising of manifold-valued
vertex functions. We begin introducing a family of anisotropic energy
functionals for $p \geq 1$ to be optimized:
\begin{equation}
\label{eq:anisomodel}
\begin{split}
\E_\mathrm{a}(f)
\ \coloneqq& \ \frac{\lambda}{2}\,d_{\mathcal H(V; \mathcal M)}^2(f_0,f) \: + \: \lVert \nabla f \rVert_{\ell_{p,p}(E; \T_f \mathcal M)}^p\\
 =& \ \frac{\lambda}{2} \sum_{u\in V} d_{\mathcal M}^2(f_0(u),f(u))
  \: + \: \frac{1}{p}\sum_{(u,v)\in E} \lVert \nabla f(u,v)\rVert_{f(u)}^p.
\end{split}
\end{equation}
Within the data fidelity term~$d_{\mathcal H(V; \mathcal M)}$
denotes the metric of the space of vertex functions $\mathcal H(V; \mathcal M)$
measuring the distance of a function $f$ to the given data $f_0$,
$\lambda > 0$ is a fixed regularization parameter controlling the smoothness
of the solution, and the regularization term on the right side of~\eqref{eq:anisomodel} denotes the discrete anisotropic Dirichlet energy as a measure of local variance.
The denoising task is now to find a solution \(f^{\star} \in \mathcal H(V; \mathcal M)\) to the following optimization problem:
\begin{equation}
\label{eq:variational_problem_aniso}
f^{\star}\in\argmin_{f \in \mathcal H(V; \mathcal M)} \E_\mathrm{a}(f).
\end{equation}

Furthermore, we are interested in a family of energy functionals of the form
\begin{equation}
\label{eq:isomodel}
\begin{split}
\E_\mathrm{i}(f)
\ &\coloneqq \ \frac{\lambda}{2}\,d_{\mathcal H(V; \mathcal M)}^2(f_0,f) \: + \: \lVert \nabla f \rVert_{\ell_{p,2}(E; \T_f \mathcal M)}^p\\
 &= \ \frac{\lambda}{2} \sum_{u\in V} d_{\mathcal M}^2(f_0(u),f(u))
  \: + \: \frac{1}{p}\sum_{u\in V} \Bigl( \sum_{v\sim u} \lVert \nabla f(u,v)\rVert_{f(u)}^2 \Bigr)^\frac{p}{2},
\end{split}
\end{equation}
The difference to the previous
model in~\eqref{eq:anisomodel} is the exchange of the
regularization term by an isotropic norm. Analogously as before, this
results in optimization problems of the form
\begin{equation}
\label{eq:variational_problem_iso}
f^{\star}\in\argmin_{f \in \mathcal H(V; \mathcal M)} \E_\mathrm{i}(f).
\end{equation}
Note that the optimization problems~\eqref{eq:variational_problem_aniso} and \eqref{eq:variational_problem_iso} cover two interesting special cases. For $p=2$ both formulations are equivalent and can be interpreted as Tikhonov-regularized denoising problem, which aims to reconstruct smooth solutions $f^{\star}$. For $p=1$ the problem \eqref{eq:variational_problem_aniso} corresponds to anisotropic total variation-regularized denoising, which has been investigated recently in the context of manifold-valued functions in~\cite{WDS14,LSKC13,BPS16} using lifting techniques and the cyclic proximal
point algorithm. On the other hand \eqref{eq:variational_problem_iso} yields for \(p=1\) an isotropic total variation-regularized denoising formulation, which has up to now only been tackled by half-quadratic minimization in~\cite{BCHPS16}.

\subsection{Optimality conditions}
\label{ss:optimality_conditions}
To perform denoising we need to find minimizers of the discrete energy functionals
introduced in Section~\ref{ss:problem_formulation}. Note that while
these energies are convex for $p \geq 1$ on the Euclidean
space~\(\mathcal M=\mathbb R^m\), in general this only holds locally on manifolds. As mentioned in Section~\ref{ss:prel_man} even sets might not have a
global notion of convexity on manifolds. However, when restricting to the
case of Hadamard manifolds, i.e., manifolds of nonpositive sectional curvature,
all important properties carry over, i.e., convexity and even lower
semi-continuity and coerciveness.
For details on optimization on Hadamard manifolds we refer to the
monograph~\cite{Bac14b}.
In the case of Hadamard manifolds we may thus conclude the existence of respective minimizers~$f^{\star} \in
  \mathcal H(V; \mathcal M)$ of \eqref{eq:variational_problem_aniso} and \eqref{eq:variational_problem_iso}.
Note that for general manifolds the minimizers discussed in the following might only be local minimizers, since there is no general notion of global convexity.
Locally, we introduced (geodesic) convexity in Section 2, which we employ in the following.

To compute a minimizer of an energy
functional~$\E \colon \mathcal H(V; \mathcal M) \rightarrow \mathbb{R}$ we briefly introduce the notion of a derivative and a subdifferential of~$\E$, i.e., which fall back
to subdifferentials on manifolds. Since one is not able to perform basic
arithmetic operations on $\mathcal H(V; \mathcal M)$, e.g., addition of
functions or scalar multiplication, a gradient of the energy functional $\E$
has to be defined on the set of tangent spaces $\T \mathcal M$, see, e.g.,~the text book~\cite{Lee97}.

\begin{definition}[Differential and gradient]
Let us denote by \(\mathcal C^{\infty}(\mathcal M,\mathbb R)\) all smooth maps from \(\mathcal M\) to \(\mathbb R\).  Furthermore, let \(F\in \mathcal C^{\infty}(\mathcal M,R)\) and let \(\gamma\colon(-\varepsilon,\varepsilon) \to \mathcal M\) be a curve with
  \(\gamma(0) = x\in\mathcal M\) and \(\dot\gamma(0)=\xi\in \T_x\mathcal M\).
  Then the \emph{differential} \(D_xF\colon \T_x\mathcal M\to\mathbb R\) of $F$ is given by
  \[
    D_xF[\xi] \ \coloneqq \ (F\circ\gamma)'(0)\ .
  \]
  The gradient \(\nabla_{\mathcal M} F(x)\) can be characterized by
  \[
    \langle \nabla_{\mathcal M}F(x),\xi\rangle_x \, = \, D_xF[\xi], \quad \text{ for all } \, \xi\in T_x\mathcal M.
  \]
\end{definition}
However, for the interesting case of \(p=1\) in \eqref{eq:variational_problem_aniso} and \eqref{eq:variational_problem_iso} classical differentiability is too restrictive for obtaining suitable minimizers. For this reason we adapt the notion of a subdifferential from~\cite{FO98} to our setting.
\begin{definition}[Subdifferential]
  Let \(F\colon\mathcal M \to \mathbb R\) be a (locally) convex function
  and \(x\in\mathcal M\).
  Then \emph{the subdifferential} of \(F\) at \(x\) is defined by
  \begin{align*}
    \partial F \, \coloneqq \,
    \Bigl\{
      \xi\in \T_x\mathcal M
      &:
      F(\gamma(t)) \geq F(x) + t\langle\xi,\dot\gamma(0)\rangle,\\
      &\text{ for any curve }\gamma\colon(-\varepsilon,\varepsilon)\to\mathcal M\text{ with }\gamma(0)=x\text{ and any }t\geq 0
    \Bigr\}.
  \end{align*}
  An element~\(\xi\in\partial F(x)\) of the subdifferential of \(F\) at \(x\) is called \emph{subgradient}. 
\end{definition}
If the subdifferential is a singleton the subgradient is unique and equals the gradient \(\nabla_{\!\mathcal M}\). 
Before explicitly deriving the necessary optimality conditions for the introduced variational denoising models 
on manifold-valued functions in Section \ref{ss:numeric_schemes}, we would like to recall the following useful result from~\cite{ATV13}.
\begin{lemma}
Let $y\in\mathcal M$ denote a fixed point on the manifold and let \(F_p\colon\mathcal M \to\mathbb R\) be a function with \(F_p(x) = d^p(x,y)\) for \(p\geq 1\). Then we have
\[
  \nabla_{\!\mathcal M}F_p \ = \
  \begin{cases}
     \, -\frac{p\log_xy}{d^{2-p}(x,y)} & \text{ for } x\neq y \text{ or } p\geq 2,\\
     \, \hfill 0 & \text{ for } x=y \text{ and } 1<p<2.
  \end{cases}
\]
Furthermore, for \(p=1\) and \(x=y\) we have \(\partial F_1(x) = B_{x,1}\), where \(B_{x,r}\coloneqq\{\xi\in \T_x\mathcal M\ :\ \lVert\xi\rVert_x\leq r\} \subset \T_x\mathcal M\).
\end{lemma}
The second case of the subgradient can be derived by looking at the
subdifferential and observing that it only contains \(0\).
The last subdifferential follows from the definition of the subdifferential, cf.~\cite{FO98},
\[
  \partial F_1(x) \, \coloneqq \,
    \bigl\{
      \xi \in T_x\mathcal M\ : \ d(x,z) \geq \langle\xi,\log_xz\rangle_x
      \text{ for all }z\in\mathcal M_x\},
\]
in a certain ball~\(\mathcal M_x\subset\mathcal D_x\subset\mathcal M\)
around \(x\) where the exponential map is injective and \(\xi = \frac{1}{d_{\mathcal M}(x,z)}\log_xz\) for \(x\neq z\).
Note that especially \(0\in\partial F_p(x)\) for \(x=y\) and \(p\geq 2\).

Based on this result we are able to derive the conditions for a minimizer of
the anisotropic energy functional~\eqref{eq:variational_problem_aniso}.
Note that in this discrete setting the energy functional is in fact a function
defined on the (product) manifold~\(\mathcal M^{\lvert V\rvert}\).
We derive
\begin{align*}
  0 &\in\ \partial\E_{\mathrm{a}}(f)\\
    \ &\quad=\ \partial
        \left( \frac{\lambda}{2}\,\sum_{u \in V} d_{\mathcal M}^2(f_0(u),f(u))
      \: + \: \frac{1}{p}\sum_{(u,v)\in E}
        \lVert \nabla f\rVert_{f(u)}^p \right)(f) \\
    \ &\quad=\ \frac{\lambda}{2}
        \sum_{u\in V} \nabla_{\mathcal M} d_{\mathcal M}^2(f_0(u),f(u)) \: + \: \frac{1}{p}\sum_{u\in V} \sum_{v\sim u} (w(u,v))^{\frac{p}{2}} \partial d^p_{\mathcal M}(f(u),f(v)).
\end{align*}
Note that only in the case \(p=1\) and \(f(u)=f(v)\), \(u,v\in V\), the subdifferential in the second summand is not a singleton. As discussed above the subdifferential however  contains \(0\). Furthermore, in the case \(p<2\) and \(f(u)=f(v)\) we have defined the anisotropic graph \(p-\)Laplacian as \(0\).
Bearing these special cases in mind we can further simplify the optimality conditions as
\begin{align*}
    0 \ &\overset{!}{=}\ - \lambda \sum_{u\in V} \log_{f(u)}f_0(u)
      \:-\:\sum_{u\in V}\sum_{v\sim u} (w(u,v))^{\frac{p}{2}}
        \,d_{\mathcal M}^{p-2}(f(u),f(v)) \log_{f(u)} f(v)\\
      \ &=\ \sum_{u\in V} \left( -\lambda \log_{f(u)} f_0(u)
      \: + \: \Delta^\mathrm{a}_{p}f(u) \right)
\end{align*}
This leads to the following PDE on a finite weighted graph as necessary condition for a minimizer of \eqref{eq:variational_problem_aniso}:
\begin{equation}\label{eq:PDE_aniso}
\Delta^\mathrm{a}_{p}f(u) \: - \: \lambda \log_{f(u)} f_0(u) \ = \ 0 \: \in T_{f(u)}\mathcal M \quad \text{ for all } u\in V.
\end{equation}

Following analogously the argumentation above 
we are able to derive the necessary conditions for a minimizer of the isotropic energy functional \eqref{eq:variational_problem_iso} as follows:
\begin{align*}
  0 &\in\ \partial\E_{\mathrm{i}}(f)\\
    &\quad=\ \partial\Bigl( \frac{\lambda}{2}\,%
      \sum_{u \in V} d_{\mathcal M}^2(f(u),f_0(u))
      \: + \: \frac{1}{p}\sum_{u\in V} \Bigl(\sum_{v\sim u}
        \lVert\nabla f(u,v)\rVert^2_{f(u)}\Bigr)^{\frac{p}{2}} \Bigr)(f)\\
    &\quad=\ \frac{\lambda}{2} \sum_{u\in V}\nabla_{\mathcal M} d_{\mathcal M}^2(f_0(u),f(u))\\
    &\quad\qquad + \:\frac{1}{p}\sum_{u\in V}
      \frac{p}{2} \Bigl(
        \sum_{v\sim u} w(u,v) d^2_{\mathcal M}
          (f(u),f(v))\Bigr)^{\frac{p-2}{2}}
      \sum_{v\sim u} w(u,v) \nabla_{\mathcal M} d^2_{\mathcal M}(f(u),f(v))\\
\end{align*}
Hence we obtain
\begin{align*}
    0 &\overset{!}{=}\ - \lambda \sum_{u\in V} \log_{f(u)}f_0(u)
    \: - \: \sum_{u\in V} \Bigl(
      \sum_{v\sim u} w(u,v) d^2_{\mathcal M} (f(u),f(v))\Bigr)^{\frac{p-2}{2}}
      \sum_{v\sim u} w(u,v) \log_{f(u)} f(v)\\
      &=\ \sum_{u\in V} \left( -\lambda \log_{f(u)} f_0(u) \: + \: \Delta^\mathrm{i}_{p}f(u) \right)
\end{align*}
This leads to the following PDE on a finite weighted graph as necessary condition for a minimizer of \eqref{eq:variational_problem_iso}:
\begin{equation}
\label{eq:PDE_iso}
\Delta^\mathrm{i}_{p}f(u) \: - \: \lambda \log_{f(u)} f_0(u) \ = \ 0 \: \in T_{f(u)}\mathcal M, \quad \text{ for all } u\in V.
\end{equation}

\subsection{Numerical optimization schemes}
\label{ss:numeric_schemes}
In the following we derive numerical minimization schemes to approximate solutions to the necessary optimality conditions of both the anisotropic and the isotropic denoising problems discussed in Section \ref{ss:optimality_conditions}. 
In particular, we will discuss two different iterative methods converging to solutions of the PDEs 
in~\eqref{eq:PDE_aniso} and~\eqref{eq:PDE_iso}. For the sake of simplicity we discuss both approaches collectively by introducing an operator $\Delta_p \colon \mathcal H(V; \mathcal M) \rightarrow \mathcal H(V; T\mathcal M)$ as place holder for the anisotropic and isotropic $p$-Laplacian operators. 
We only distinguish between the two formulations when giving explicit computation formulas at the end of each paragraph.
\paragraph{Explicit scheme by time-discretized evolution.}
The basic idea of our first numerical approach is to extend the domain of our vertex functions by an additional artificial time dimension and subsequently compute a steady-state solution of the resulting time-dependent PDEs, see e.g., \cite{S86}. Thus, the aim is to find a stationary solution of the parabolic PDEs with initial value conditions of the form
\begin{equation}
\label{eq:parabolic_formulation}
	\begin{cases}
	\ \frac{\partial f}{\partial t}(u,t) \, &= \ \Delta_pf(u,t) \: - \: \lambda \log_{f(u,t)} f_0(u,t) \qquad \text{ for all } u \in V, \, t \in (0, \infty),\\
	\ f(u,0) \, &= \ f_0(u) \hspace{4.91cm} \text{ for all } u\in V, \, t=0,
	\end{cases}
\end{equation}
i.e., we try to compute a solution $f \in \mathcal H(V\times(0,\infty); \mathcal M)$ such that
\begin{equation*}
\frac{\partial f}{\partial t}(u,t) \ = \ 0, \quad \text{ for all } u \in V.
\end{equation*}
Note that finding a stationary solution to~\eqref{eq:parabolic_formulation}
directly yields a minimizers of the original energy functionals
problems~\eqref{eq:anisomodel}
and~\eqref{eq:isomodel} as the left side becomes zero. Setting $\lambda = 0$
the problem \eqref{eq:parabolic_formulation} covers two well-known special
cases. For $p=2$ one gets a diffusion equation for manifold-valued vertex
functions which is equivalent to the discrete heat equation for a local neighborhood in the finite weighted graph $G$. For $p=1$ one gets the well-known total variation flow.

To solve the initial value problem \eqref{eq:parabolic_formulation} we discretize the time derivative on the left hand side of the equation using an forward difference scheme and assuming the right hand side of the equation to be known from the last time step, i.e.,
\begin{equation*}
	\frac{\log_{f_n(u)} f_{n+1}(u)}{\Delta t} \ = \ \Delta_p f_n(u) \: - \: \lambda \log_{f_n(u)} f_0(u)\ .
\end{equation*}
Here, $\Delta t$ denotes the step size of the evolution process induced by the time discretization. This gives us the following explicit formulas for the iterative computation of solutions to the anisotropic and isotropic denoising problems,
\begin{equation}\label{eq:aniso-explicit}
	\begin{split}
	&f_{n+1}(u) \ = \ \exp_{f_n(u)} \left( \Delta t \, ( \Delta_p^\mathrm{a} f_n(u) - \lambda \log_{f_n(u)} f_0(u) ) \right) \\
	&\!\! = \ \exp_{f_n(u)} \Bigl( \Delta t \, \Bigl( -\sum_{v\sim u}\sqrt{w(u,v)}^p d^{p-2}_{\mathcal M}(f_n(u),f_n(v))\log_{f_n(u)}f_n(v) - \lambda \log_{f_n(u)} f_0(u) \Bigr)\Bigr)
	\end{split}
\end{equation}
and
\begin{equation}\label{eq:iso-explicit}
	\begin{split}
		&f_{n+1}(u) \ = \ \exp_{f_n(u)} \left( \Delta t \, ( \Delta_p^\mathrm{i} f_n(u) - \lambda \log_{f_n(u)} f_0(u) ) \right) \\ 
	&\!\! = \ \exp_{f_n(u)} \Bigl( \Delta t \, \Bigl( - \,
  \Bigl(
    \sum_{v\sim u} w(u,v) d^2_\M(f_n(u),f_n(v))
  \Bigr)^{\frac{p-2}{2}}
  \\&\hspace{4cm}\times
  \sum_{v\sim u}
  w(u,v)
  \log_{f_n(u)}f_n(v) - \lambda \log_{f_n(u)} f_0(u) \Bigr) \Bigr),
	\end{split}
\end{equation}
respectively.

Note that this Euler forward time discretization realizes a gradient descent minimization algorithm. 
Although this approach is very simple to implement, it leads for $p < 2$ to very strict Courant-Friedrich-Lewy (CFL) conditions 
in order to guarantee the stability of this numerical scheme. This means that 
updates $f^n \rightarrow f^{n+1}$ can only be performed with very small 
step size $\Delta t$ and hence convergence to possible solutions may need a huge amount of iterations. 
To alleviate this problem one may consider higher-order schemes, e.g., Runge-Kutta discretization schemes.
A general CFL condition for a broad class of graph operators can be found in \cite{ETT15} for the case $\mathcal M = \mathbb{R}^n$.

\paragraph{Semi-implicit scheme by linearization.}
Another approach, which we discuss in the following, is based on the idea of approximating the original problem \eqref{eq:parabolic_formulation} in every iteration step and solving a simple linear equation system. To perform this linearization we first reformulate \eqref{eq:parabolic_formulation} to an equivalent problem by inserting two zeros terms $\log_{f(u)} f(u) = 0$. 

For the anisotropic case we get
\begin{align*}
	 0 \ &=\  \Delta^\mathrm{a}_{w,p}f(u) \: - \: \lambda \log_{f(u)} f_0(u)\\
	 &=\  -\sum_{v\sim u}\sqrt{w(u,v)}^p d_{\mathcal M}(f(u),f(v))^{p-2} \, (\log_{f(u)}f(v) - \log_{f(u)}f(u))\\
   &\hspace{4cm} - \ \lambda \, (\log_{f(u)} f_0(u) - \log_{f(u)}f(u)).
\end{align*}
By introducing \[
b_{\mathrm{a}}(u,v) \coloneqq \sqrt{w(u,v)}^p d_{\mathcal M}(f(u),f(v))^{p-2}\]
we can write this more concisely as
\[
0	= -\sum_{v\sim u}b_{\mathrm{a}}(u,v)(\log_{f(u)}f(v) - \log_{f(u)}f(u))
\ - \ \lambda \, (\log_{f(u)} f_0(u) - \log_{f(u)}f(u)).
\]
Analogously for the isotropic case and writing
\[
  b_{\mathrm{i}}(u,v)\coloneqq \Bigl(\sum_{x\sim u}  w(u,x) \, d_{\mathcal M}^2(f(u),f(x))
  \Bigr)^{\frac{p-2}{2}} w(u,v)
\]
we obtain
\begin{align}\label{eq:iso-reformulated}
	 0 \ &= \  \Delta^\mathrm{i}_{w,p}f(u) \: - \: \lambda \log_{f(u)} f_0(u)\\
	&= - \,
  \sum_{v\sim u} b_{\mathrm{i}}(u,v)(\log_{f(u)}f(v) - \log_{f(u)}f(u))
	- \ \lambda \ (\log_{f(u)} f_0(u) - \log_{f(u)}f(u)).\nonumber
\end{align}
We can investigate both the isotropic and the anisotropic model in the same
manner by introducing \(b(u,v)\) which refers to exactly one of the above
introduced coefficients. The discussed problems can be linearized by assuming
that~\(b(u,v)\) is known (from the last iteration $f_n$) and approximating
the introduced zero terms as~$\log_{f(u)}f(u) \approx \log_{f_{n+1}(u)}f_n(u)$.
Hence, we need to solve a linear equation system $Bx = c$.
Note that this corresponds to a semi-implicit scheme for the original
optimality conditions~\eqref{eq:PDE_aniso} and~\eqref{eq:PDE_iso}, respectively.
For the case of graph constructions with relatively small neighborhoods, e.g.,
for local grid graphs or $k$-nearest neighbor graphs with small integers $k$,
the resulting matrix $B$ is very sparse. Hence, it is advisable to use iterative
methods to approximate solutions of the linear equation system. 

By applying Jacobi's method we get the following relationship
in~\(T_{f_n(u)}\mathcal M\):
\begin{equation*}
\left( \lambda + \sum_{v\sim u} b(u,v)\right) \log_{f_n(u)}f_{n+1}(u) \ = \ \sum_{v\sim u} b(u,v) \log_{f_n(u)}f_n(v) \, + \, \lambda \log_{f_n(u)} f_0(u).
\end{equation*}
This leads to the iterative update formula
\begin{equation}\label{eq:Jacobi-iterate}
	f_{n+1}(u) \ = \ \exp_{f_n(u)} \left( \frac{\sum_{v\sim u}b(u,v)\log_{f_n(u)}f_n(v) + \lambda \log_{f_n(u)} f_0(u) }{\lambda \: + \: \sum_{v\sim u}b(u,v)} \right) ,
	\end{equation}
for which the terms $b(u,v)$ are given as above.

We would like to give some additional notes on the proposed numerical minimization schemes discussed above.
\begin{enumerate}
\item In order to avoid division by zero we use a relaxation of the terms \(d_\M^{p-2}\) for \(0 < p < 2\) by adding a small term \(\varepsilon = 10^{-7}\). Thus, we smooth our regularization functionals at zero.
\item The prefactor~\(b_{\mathrm{i}}(u,v)\) of the isotropic graph $p$-Laplacian in \eqref{eq:iso-reformulated} is equal for all neighbors $v \sim u$ and hence can be precomputed once per vertex.
\item Although Jacobi's method is in general not as efficient as related methods, e.g., the Gauss-Seidel method or successive overrelaxation methods, it has the advantage that one is able to compute all \(\log\) terms and all geodesic distances~\(d_{\mathcal M}\) needed in~\eqref{eq:Jacobi-iterate} in a vectorial manner, which gives a higher overall computational performance
when exploiting vectorization techniques.
\item For small values of \(\lambda\) tending to zero we observed that the proposed semi-implicit scheme tends to be less robust than the proposed explicit scheme. This is likely due to the reason that the employed linearization for Jacobi's method is not valid in cases with almost no data fidelity. Hence, for small values of the regularization parameter \(\lambda\) we will use the explicit scheme in our experiments.
\item While the derivation in this Section is restricted to \(p\geq 1\), the numerical schemes can also be applied for the case~\(p\geq 0\).
The investigation of the quasi-norm cases is a point of future work.
\end{enumerate}
%
\section{Applications}
\label{s:applications}
In the following we present several examples illustrating the large variety of
problems that can be tackled using the proposed manifold-valued graph framework. 
We begin by discussing our general experiment setup. We first focus
on image domains in Section~\ref{ss:applications_images}, i.e.,~manifold-valued data defined on regular grids and investigate the
evolutionary flow for a local graph modeling spatial pixel neighborhoods.
Furthermore, we compare our framework for the special case of nonlocal denoising
of phase-valued data to a state-of-the-art method. Finally, we demonstrate a
real-world application from denoising surface normals in digital elevation maps
from LiDAR data.
Subsequently, we model manifold-data measured on samples of an explicitly given surface and in particular illustrate denoising of diffusion tensors measured on a sphere in Section~\ref{ss:applications_sphere}.
Finally, we investigate denoising of real DT-MRI data from medical applications in Section~\ref{s:DT-MRI}
both on a regular pixel grid as well as on an implicitly given surface.

All algorithm were implemented in \textsc{Mathworks Matlab} by extending the open source
software package Manifold-valued Image Restoration Toolbox (MVIRT)\footnote{open source, available at~%
\href{http://www.mathematik.uni-kl.de/imagepro/members/bergmann/mvirt/}{www.mathematik.uni-kl.de/imagepro/members/bergmann/mvirt/}}.

\paragraph{Patch distances on manifolds.}
Similar to the real- or vector-valued case we can employ a \emph{patch-based} distance function. Let \(f\in \mathcal M^{n,m}\) be a manifold-valued image of dimension \(n\times m\). Let us further denote by \(f_{i}^{s}\in \mathcal M^{2s+1,2s+1}\) the \(2s+1\times 2s+1\)-sized patch with center \(f_{i}^s(s+1,s+1)
= f(i_1,i_2), i=(i_1,i_2)\in \mathcal G\coloneqq \{1,\ldots,n\}\times\{1,\ldots,m\}\).
For constructing patches at the boundary of the domain we assume periodic boundary conditions if \(i_1-s<0, i_2<0, i_1+s>n,\) or \(i_2+s>m\).

Then the patch-based similarity measure for manifold-valued data is given by
\begin{equation}\label{eq:patchdistance}
  \operatorname{PSM}^2(i,j) \coloneqq
  \sum_{k,l=-s}^s d^2_{\mathcal M}(f_{i}^s(k,l),f_j^s(k,l)),
  \qquad i,j\in \mathcal G.
\end{equation}

Using this similarity measure we can introduce the (nonlocal) \(\varepsilon\)-ball graph
\(G_{\text{nl}}\) of an image \(f\), where \(V=\mathcal G\), i.e., every pixel is a vertex and \(E_\varepsilon = \{(i,j) \in V\times V : 1-\operatorname{PSM}(i,j) < \varepsilon\}\). Furthermore, we can also introduce the (nonlocal) \(k\)-neighbor graph, where \(i\in V\) is connected via an edge
to its \(k\) most similar patches.

\paragraph{Evaluation.}
For all our experiments we performed a parameter search with respect to the
regularization parameter \(\lambda\) in the models~\eqref{eq:anisomodel}
and~\eqref{eq:isomodel} in order to minimize the
mean squared error (MSE), which is given by
\[
  \operatorname{MSE}(f)
  \coloneqq
  \frac{1}{\lvert V \rvert}\sum_{v\in V}
    d_{\mathcal M}^2(f(v),\hat{f}(v)),
\]
for which \(\hat{f}\) denotes the original data without any perturbations.
%
%
\subsection{Manifold-valued images}
\label{ss:applications_images}
\paragraph{Evolution equations for the graph \(p\)-Laplacian.}
\begin{figure}
  \begin{subfigure}{.03\textwidth}\centering\ 
  \end{subfigure}
  \begin{subfigure}{.235\textwidth}\centering original image
  \end{subfigure}
  \begin{subfigure}{.235\textwidth}\centering
    \(i=1\,000\)
  \end{subfigure}
  \begin{subfigure}{.235\textwidth}\centering
    \(i=10\,000\)
  \end{subfigure}
  \begin{subfigure}{.235\textwidth}\centering
    \(i=100\,000\)
  \end{subfigure}
  \\
  \begin{subfigure}{.03\textwidth}\centering
    \rotatebox{90}{\small\(p=2\), (an)isotropic}
  \end{subfigure}
  \begin{subfigure}{.235\textwidth}\centering
    \includegraphics[width=.98\textwidth]{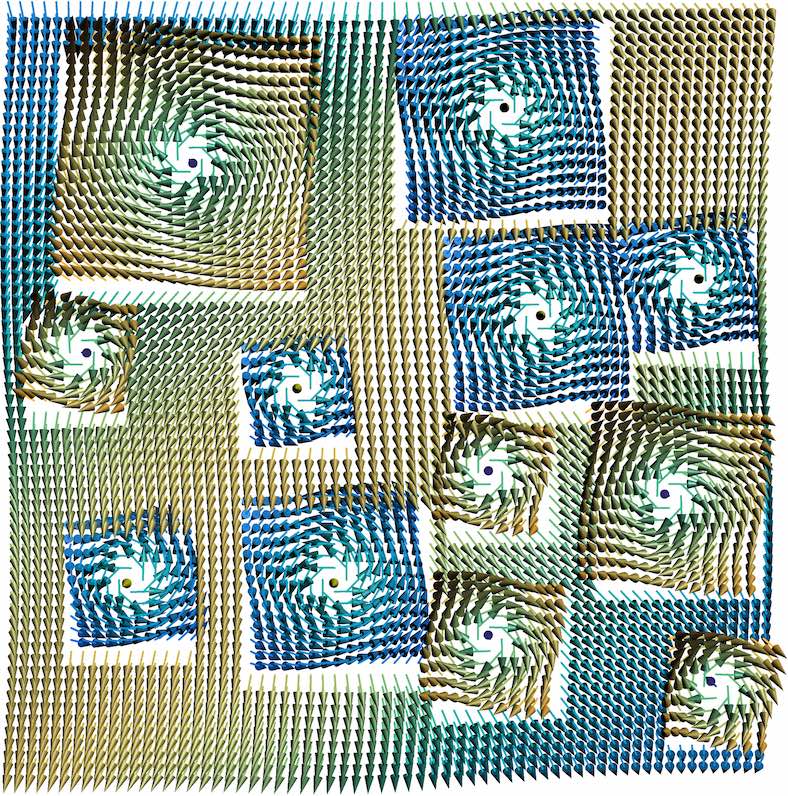}
  \end{subfigure}
  \begin{subfigure}{.235\textwidth}\centering
    \includegraphics[width=.98\textwidth]{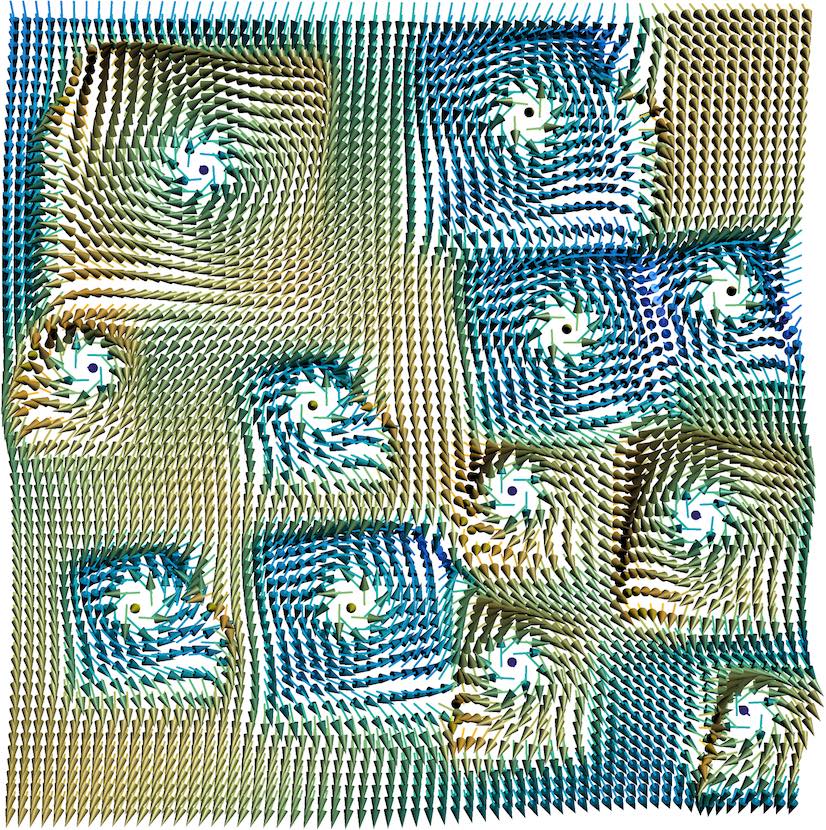}
  \end{subfigure}
    \begin{subfigure}{.235\textwidth}\centering
    \includegraphics[width=.98\textwidth]{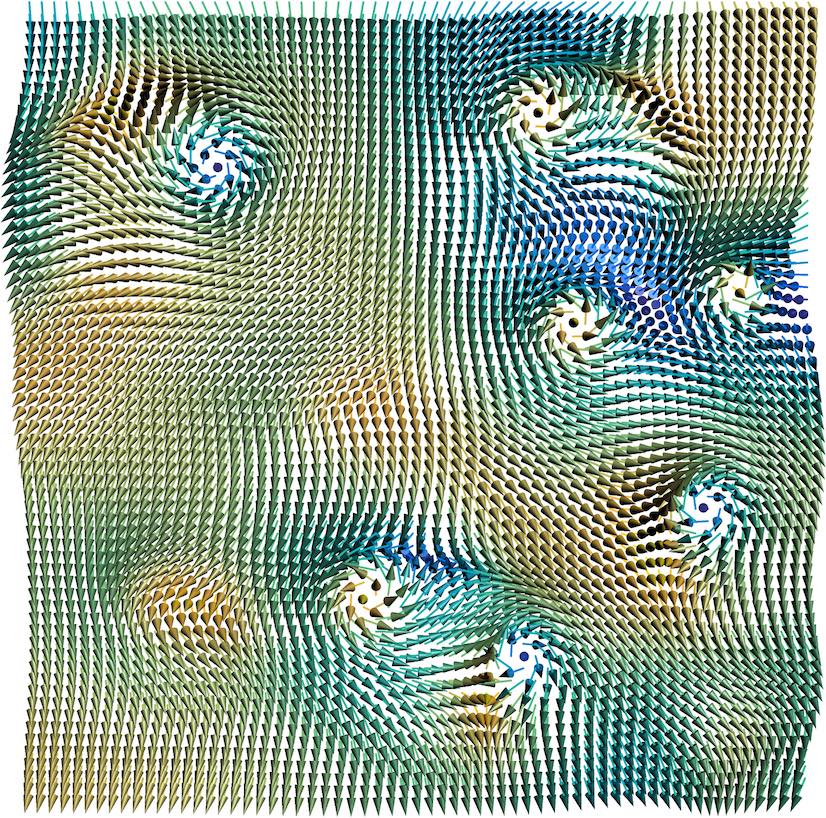}
  \end{subfigure}
  \begin{subfigure}{.235\textwidth}\centering
  \includegraphics[width=.98\textwidth]{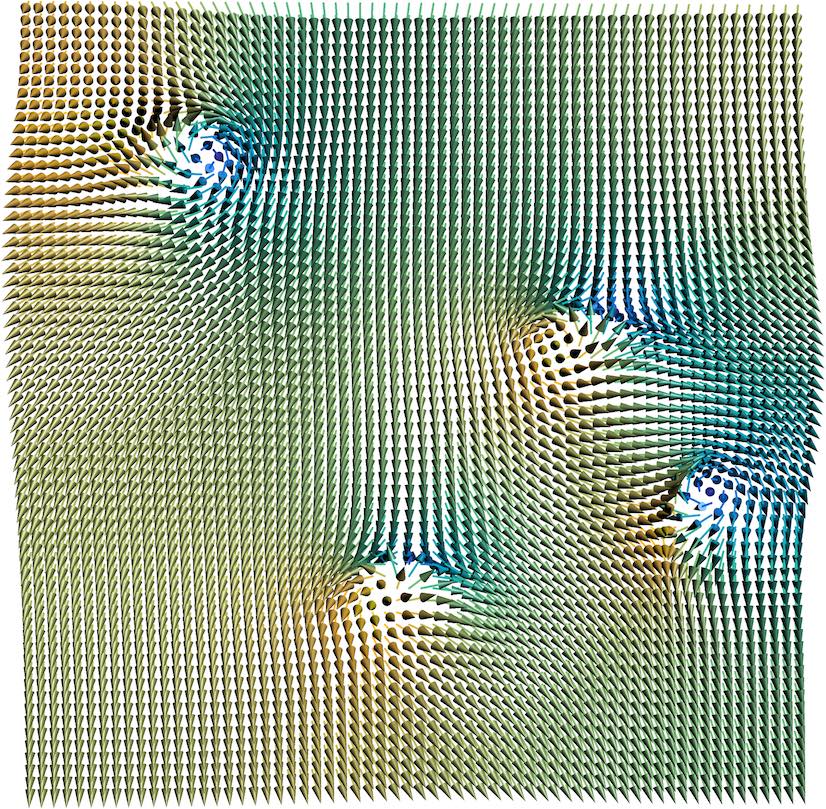}
  \end{subfigure}
  \\
  \begin{subfigure}{.03\textwidth}\centering
    \rotatebox{90}{\small\(p=1\), isotropic}
  \end{subfigure}
  \begin{subfigure}{.235\textwidth}\centering
    \includegraphics[width=.98\textwidth]{S2Whirl-original.jpg}
  \end{subfigure}
  \begin{subfigure}{.235\textwidth}\centering
    \includegraphics[width=.98\textwidth]{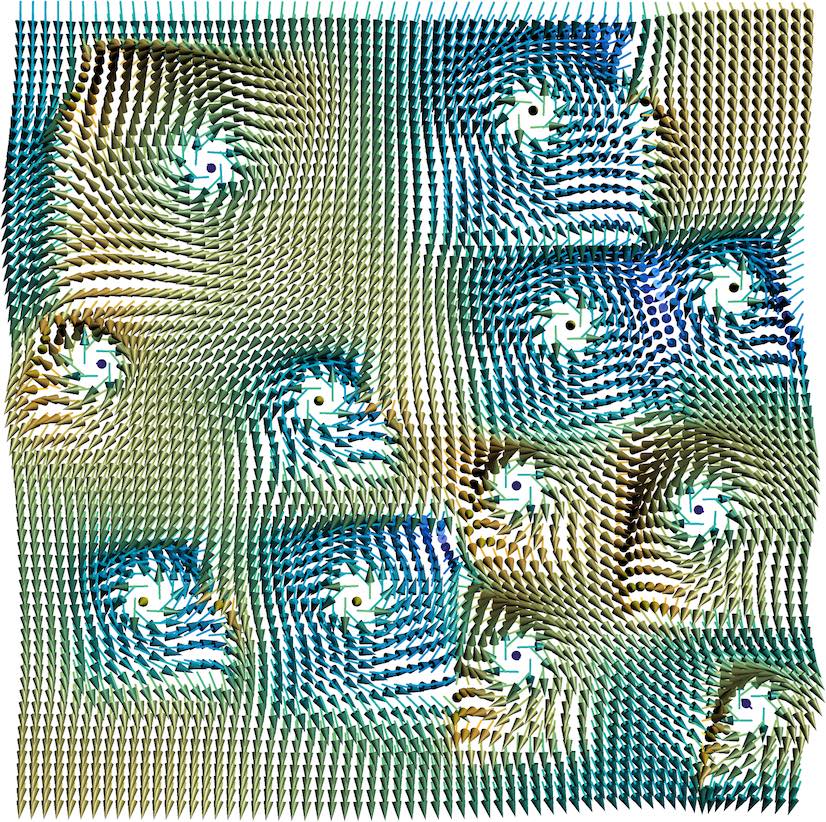}
  \end{subfigure}
  \begin{subfigure}{.235\textwidth}\centering
    \includegraphics[width=.98\textwidth]{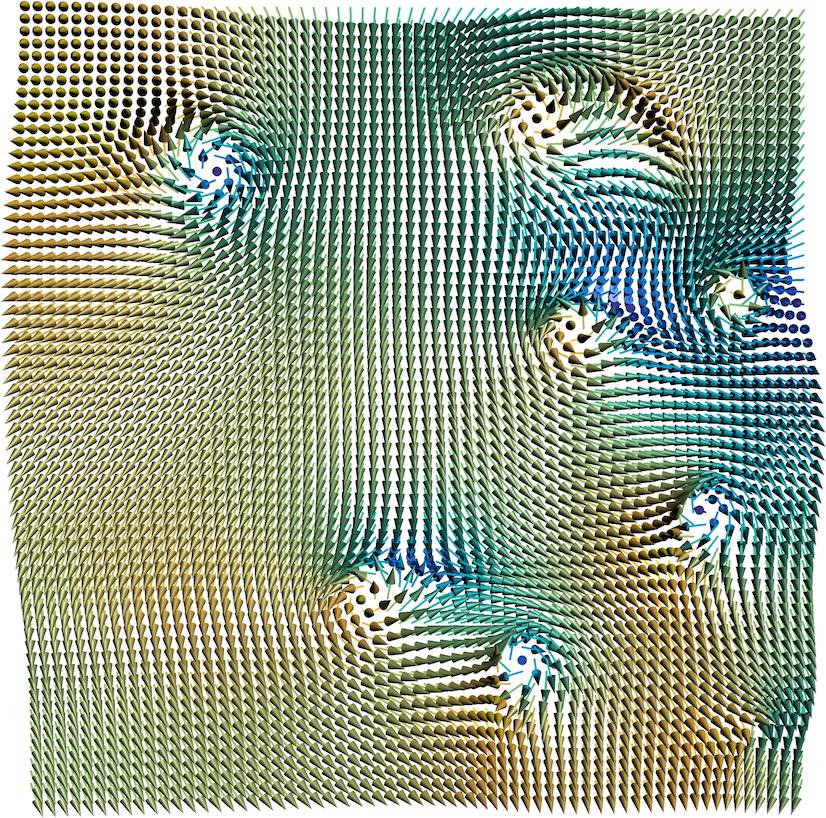}
  \end{subfigure}
  \begin{subfigure}{.235\textwidth}\centering
  \includegraphics[width=.98\textwidth]{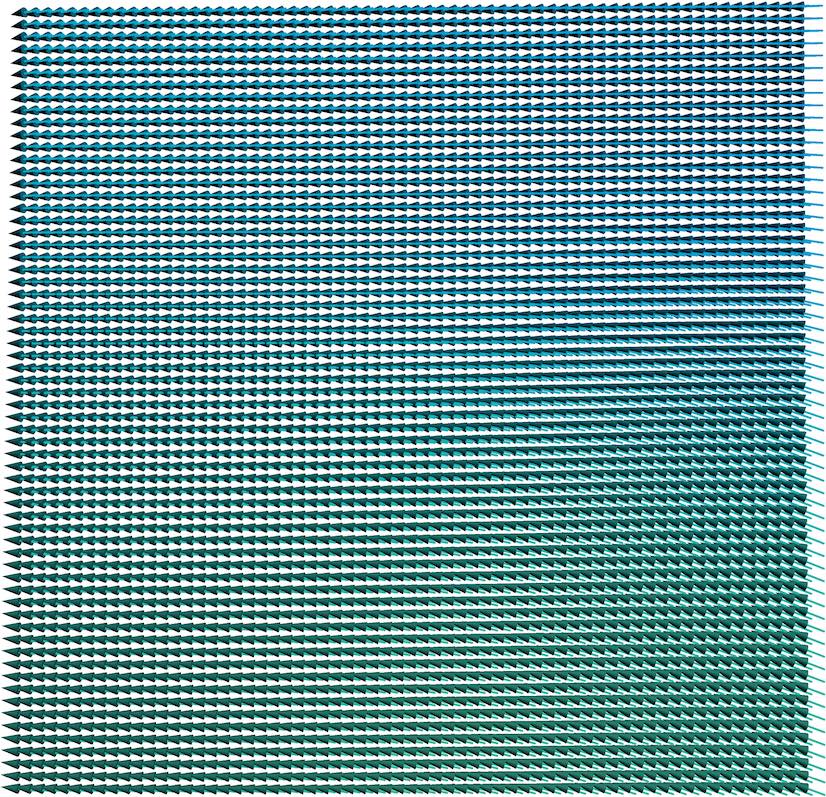}
  \end{subfigure}
  \\
  \begin{subfigure}{.03\textwidth}\centering
    \rotatebox{90}{\small\(p=1\), anisotropic}
  \end{subfigure}
  \begin{subfigure}{.235\textwidth}\centering
    \includegraphics[width=.98\textwidth]{S2Whirl-original.jpg}
  \end{subfigure}
  \begin{subfigure}{.235\textwidth}\centering
    \includegraphics[width=.98\textwidth]{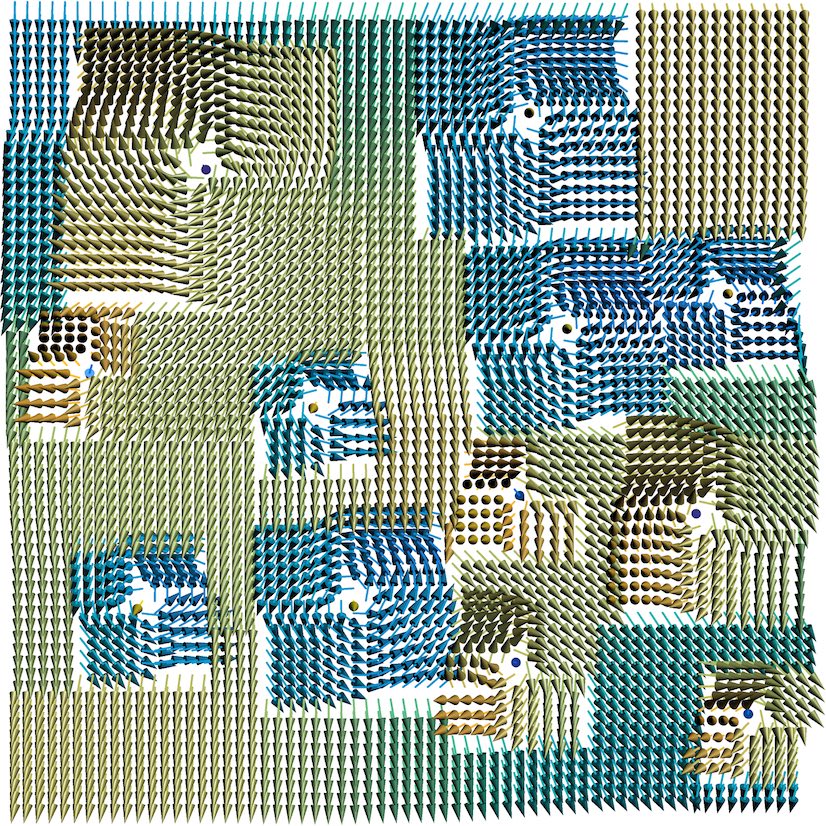}
  \end{subfigure}%
   \begin{subfigure}{.235\textwidth}\centering
    \includegraphics[width=.98\textwidth]{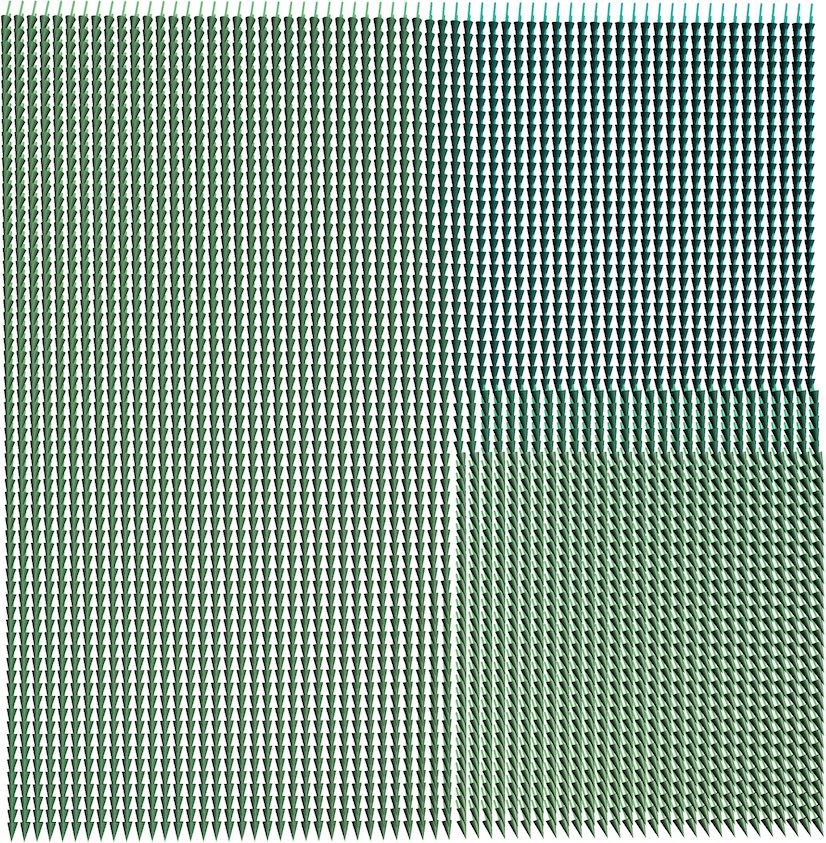}
  \end{subfigure}
  \begin{subfigure}{.235\textwidth}\centering
    \includegraphics[width=.98\textwidth]{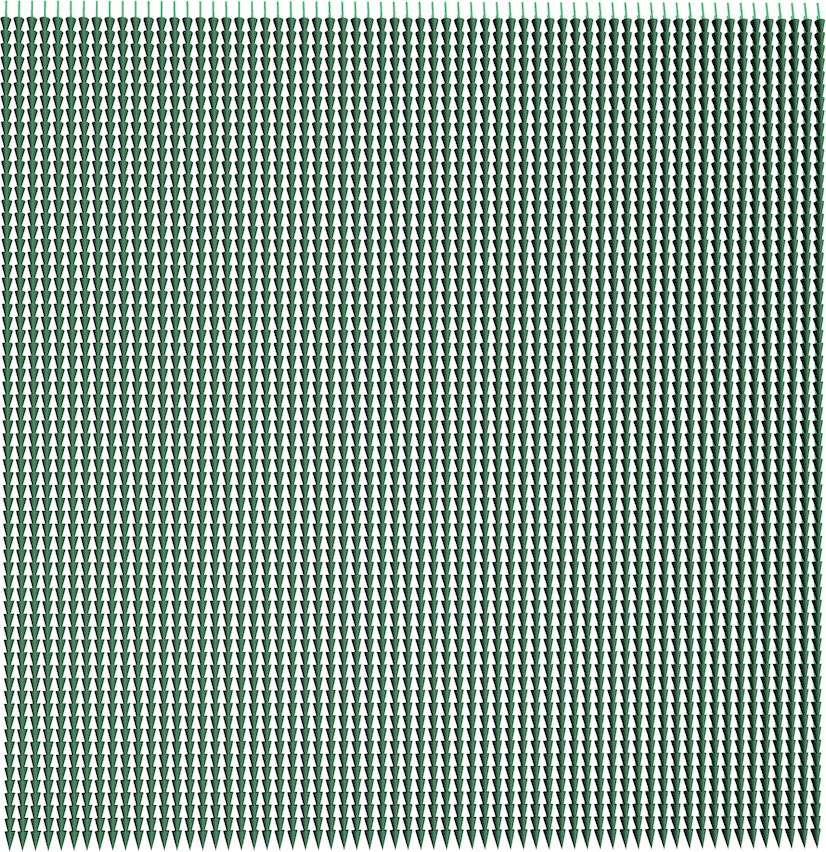}
  \end{subfigure}
  \\
  \begin{subfigure}{.03\textwidth}\centering
    \rotatebox{90}{\small\(p=0.1\), isotropic}
  \end{subfigure}
  \begin{subfigure}{.235\textwidth}\centering
    \includegraphics[width=.98\textwidth]{S2Whirl-original.jpg}
  \end{subfigure}
  \begin{subfigure}{.235\textwidth}\centering
    \includegraphics[width=.98\textwidth]{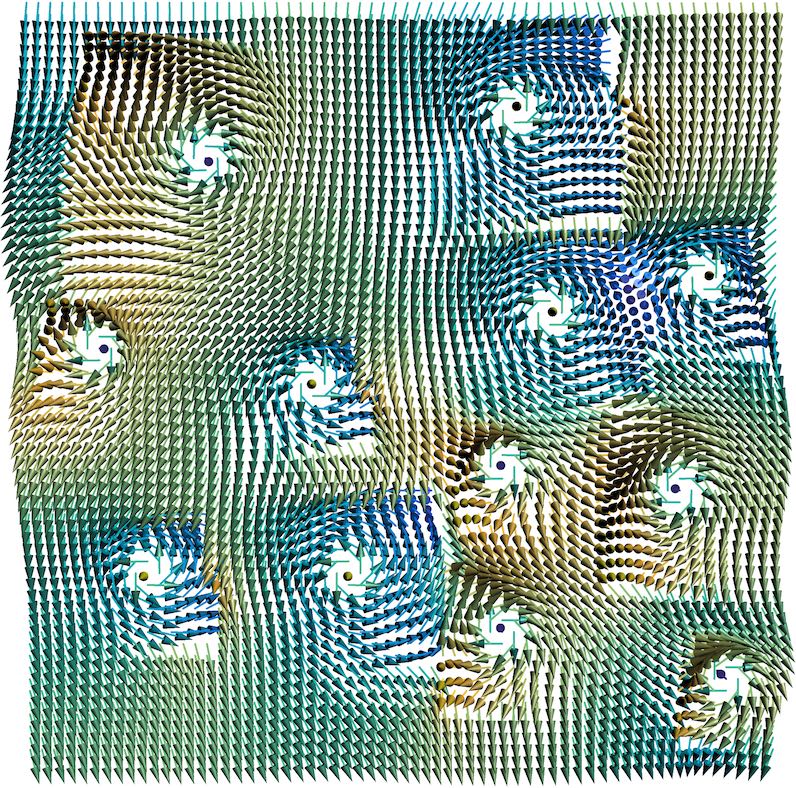}
  \end{subfigure}
  \begin{subfigure}{.235\textwidth}\centering
    \includegraphics[width=.98\textwidth]{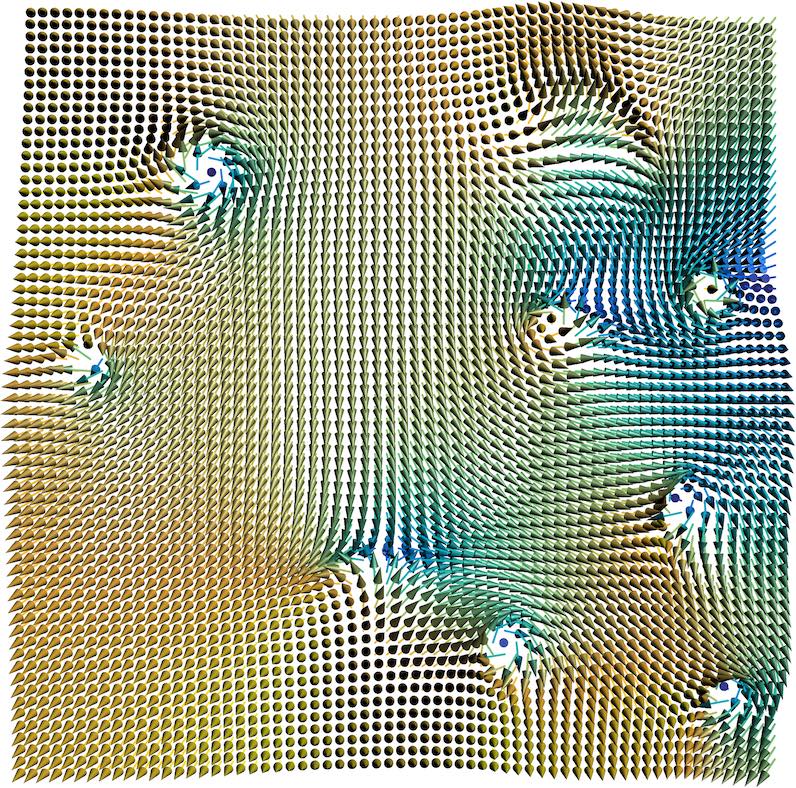}
  \end{subfigure}
  \begin{subfigure}{.235\textwidth}\centering
    \includegraphics[width=.98\textwidth]{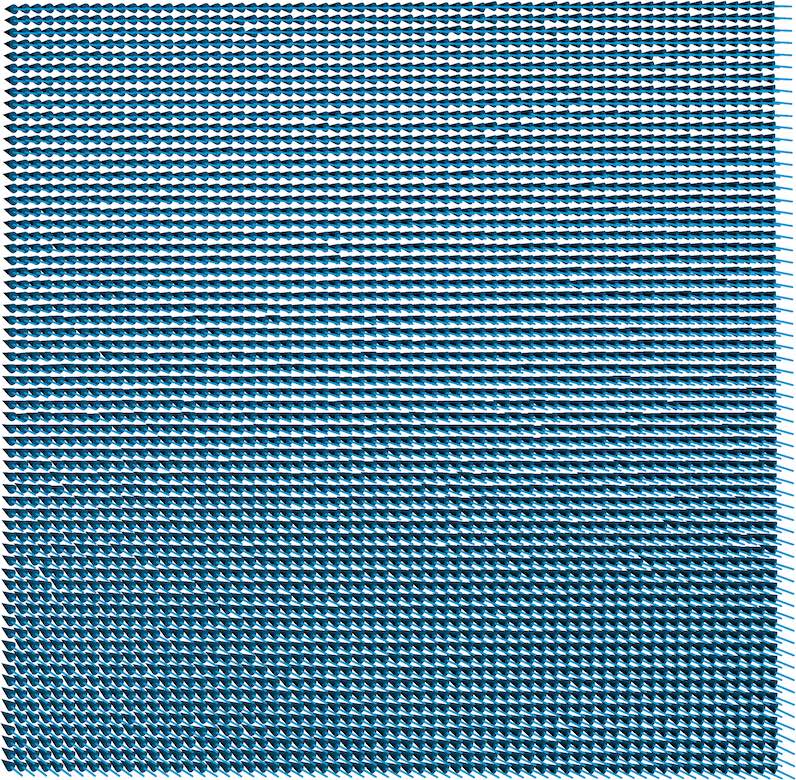}
  \end{subfigure}
  \begin{subfigure}{.03\textwidth}\centering
    \rotatebox{90}{\small\(p=0.1\), anisotropic}
  \end{subfigure}
  \begin{subfigure}{.235\textwidth}\centering
    \includegraphics[width=.98\textwidth]{S2Whirl-original.jpg}
  \end{subfigure}
  \begin{subfigure}{.235\textwidth}\centering
    \includegraphics[width=.98\textwidth]{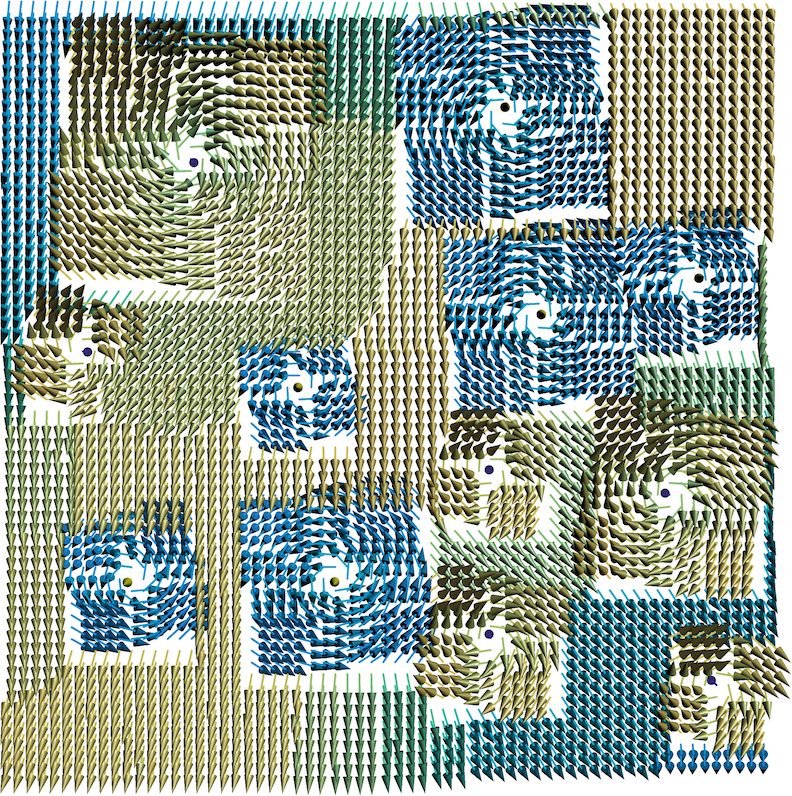}
  \end{subfigure}
    \begin{subfigure}{.235\textwidth}\centering
    \includegraphics[width=.98\textwidth]{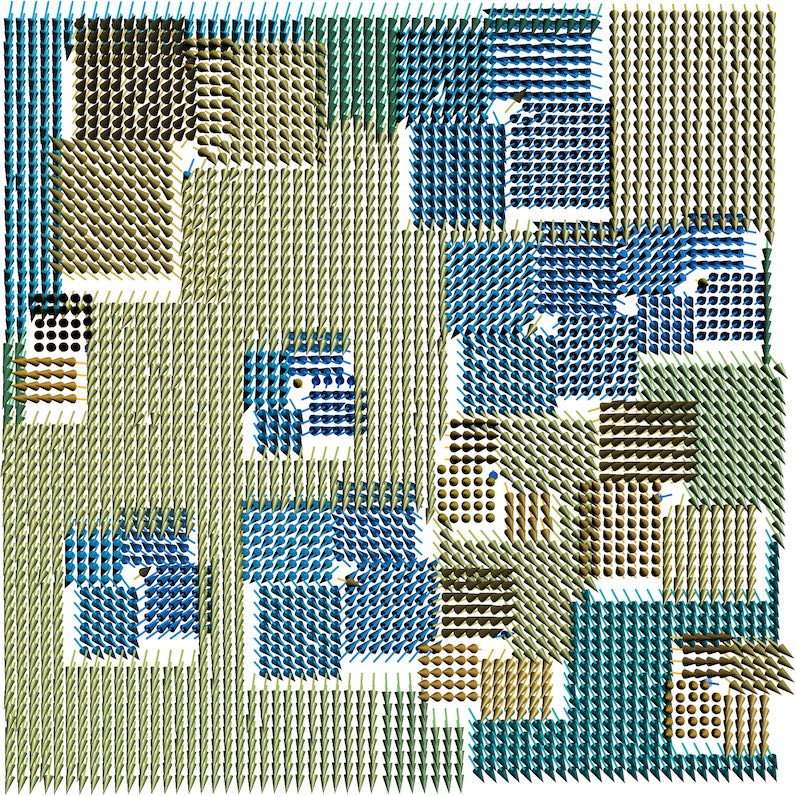}
  \end{subfigure}
  \begin{subfigure}{.235\textwidth}\centering
    \includegraphics[width=.98\textwidth]{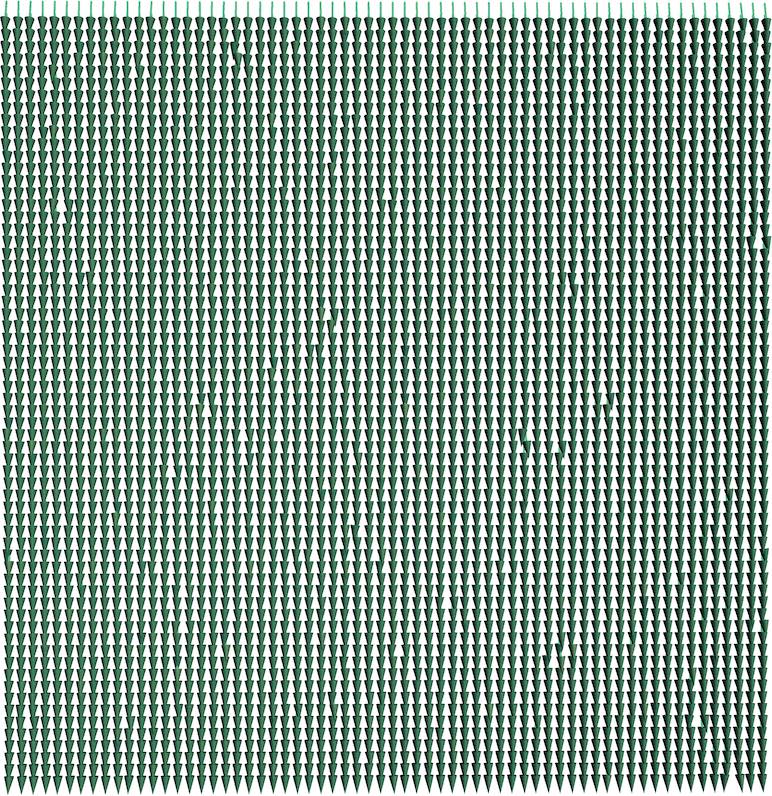}
  \end{subfigure}
  \caption{Evolutionary flow on an \(\mathbb S^2\)-valued image for
  different~\(p\in\{0.1,1,2\}\). For \(p=2\) both models are identical.}
  \label{fig:S2Whirl}
\end{figure}
We first take a look at a toy example of \(\mathbb S^2 \)-valued data by taking
the image~\lstinline!S2WhirlImage!\footnote{created by J.~Persch} shown in the
first column of Figure~\ref{fig:S2Whirl}.
In this picture a few squares consisting of \lq whirl\rq\ are depicted on a smoothly
varying background. For both kinds of whirls, one being on the northern
hemisphere clockwise (yellow) and the other on the southern hemisphere anti-clockwise (blue), the center point is the antipodal point, i.e., the south and north
pole, respectively. To look at the \(p\)-Laplace flow we set~\(\lambda=0\) in
the variational denoising models~\eqref{eq:anisomodel} and~\eqref{eq:isomodel}. 
We construct a local grid graph by connecting each pixel with his four direct neighbors.
We then employ the explicit schemes~\eqref{eq:iso-explicit}
and~\eqref{eq:aniso-explicit}, because the linearization introduced for the 
Jacobi iterations is not robust for small \(\lambda\) and is not valid
for \(\lambda=0\). We employ Neumann boundary conditions in the case of local neighborhoods.

In Figure~\ref{fig:S2Whirl} the evolution of solutions over time in the explicit scheme is shown for
the iterations \(i=0\) (starting conditions), \(i=1\,000\), \(i=10\,000\), and \(i=100\,000\). 
We set the time step size for all experiments of the explicit algorithm for \(p\in \{1,2\}\) 
to~\(\Delta t=10^{-3}\) and for \(p=0.1\) to~\(\Delta t=10^{-4}\). We further changed \(\varepsilon=10^{-4}\) for the last case to improve numerical stability. We observe
 that for all six experiments the algorithm numerically converges to a constant solution. 
For the Tikhonov case, \(p=2\), (first row) both the anisotropic and isotropic Laplace
flow coincide and the convergence to a constant image is slower compared to the
other rows, it reaches a constant solution after \(i=1\,000\,000\) iterations.
Due to the quadratic regularization term all edges are smoothed nicely, as can be best observed 
in iteration \(i=1\,000\) in the first row. In fact, at this point of the
iterations, the isotropic results are quite similar, comparing the first, second
and fourth row for \(p=2,1,0.1\), respectively. They are, however, not equal and tend to different constant solutions. 
The reason for this smooth appearance is due to the fact that we constructed a 
local graph consisting of four direct neighbors instead of 
using only two neighbors, as classically done in image processing.

On the other hand for the anisotropic TV case, \(p=1\), (third row) edges are
visibly preserved; see again iteration \(i=1\,000\). 
In this case the flow field yields a nearly piecewise constant regions in the
intermediate results.
Furthermore, the anisotropic TV flow first joins regions having only small
jumps (in geodesic distance) in between, see, e.g.,~the top left whirls at 
the bottom right corner after \(i=1\,000\) iterations.

Finally, for the anisotropic \(p=0.1\) case (last row) the intermediate results are even
 more piecewise constant compared to the anisotropic TV case, yielding a quite sparse
  gradient as one would expect for \(p<1\).

\paragraph{Nonlocal denoising of phase-valued data.}
\begin{figure}[tb]\centering
  \begin{subfigure}[t]{.33\textwidth}
    \includegraphics[width=.95\textwidth]{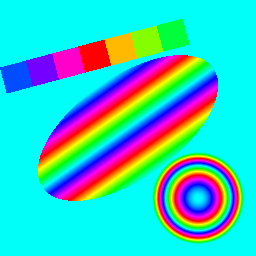}
    \caption[]{original phase image.}\label{subfig:S1:orig}
  \end{subfigure}
  \begin{subfigure}[t]{.33\textwidth}
    \includegraphics[width=.95\textwidth]{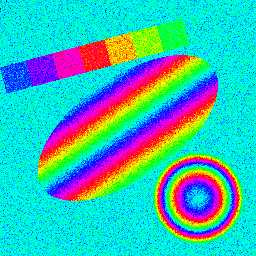}
    \caption[]{obstructed by noise,\\%
    $\sigma_{\text{n}}=0.3$, $\operatorname{MSE} = 0.0895$.}
    \label{subfig:S1:noisy}
  \end{subfigure}
  \begin{subfigure}[t]{.33\textwidth}
    \includegraphics[width=.95\textwidth]{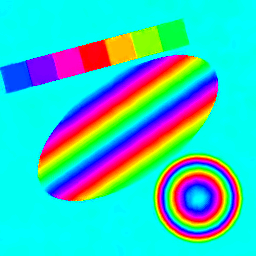}
    \caption[]{NL-MSSE approach~\cite{LPS16}
    \\$\operatorname{MSE}=0.00250$.}
    \label{subfig:S1:NLMSSE}
  \end{subfigure}
  \begin{subfigure}[t]{.33\textwidth}
    \includegraphics[width=.95\textwidth]{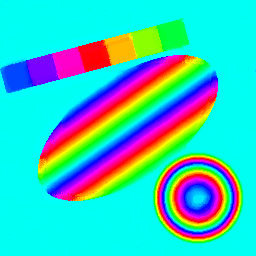}
    \caption[]{NL-TV approach,
    \(s=8\), \(\lambda = \tfrac{1}{256}\),\\ \(k=12\),
    $\operatorname{MSE} = 0.00267$.}
  \end{subfigure}
  \caption[]{Comparison of our nonlocal anisotropic TV (NL-TV) method to a state-of-the-art method using the
  NL-MSSE approach from~\cite{LPS16}.}
  \label{fig:S1}
\end{figure}
As a first data-driven denoising example let \(\mathcal M = \mathbb S^1\), i.e., we are
considering an image of size \(256\times 256\) pixels, where each pixel
of~\(f_0\in (\mathbb S^1)^{256\times 256}\) is a phase value living on the unit circle. This data occurs for
example in interferometric synthetic aperture radar (InSAR),
where height maps are generated from two parallel measurements with a laser of
a certain wavelength. Their inference often results in a noisy phase shift, i.e.,
a height map “modulo wavelength”. For an example of real-valued data see Figure~\ref{subfig:ExIntro:InSAR}. An artificial phase valued image was
introduced in~\cite{BLSW14}, see Figure~\ref{subfig:S1:orig})
and further used in~\cite{LPS16}. For our experiments we use
the same data, which has been perturbed by wrapped Gaussian noise with \(\sigma_{\text{n}}=0.3\) as shown in
Figure~\ref{subfig:S1:noisy}). The best result with respect to the MSE using second order statistics~\cite{LPS16}
is shown in Figure~\ref{subfig:S1:NLMSSE}).
Indeed, the \(k-\)nearest neighbor construction in~\cite{LPS16} is similar to ours as the authors
employ the same patch distance. However, our edge weight function is different. Based on the
\(k\)-nearest-neighbors for each vertex we set the weight function for the most
similar neighbor to~\(1\) and for the least similar neighbor to~\(0\) and
interpolate the weight function linearly in between for the other neighbors.
We optimized the involved parameters \(s,k\), and~\(\lambda\) with respect to
the MSE in the range from patch size \(s\in\{2,\ldots,7\}\),
\(\lambda\in\tfrac{1}{512}\{1,\ldots,8\}\), and~\(k\in\{5,\ldots,15\}\)
neighbors.
We employ our explicit scheme with step size~\(\Delta t=10^{-4}\).
As can be seen in Figure~\ref{fig:S1} our reconstruction result can compete with the best result from the state-of-the-art method
in~\cite{LPS16}. While the background and also small patches in the top
right corner in Figure~\ref{subfig:S1:NLMSSE}) suffer from non-constant reconstructions, our approach reconstructs constant regions
nearly perfectly due to the TV regularization. Only the the most left patch in blue suffers from fuzzy edges since there were not enough patches being similar
in the data.
Furthermore some small areas along the yellow part of the ellipsoidal
region seem also to suffer from too few similar patches, yielding a slightly worse MSE for our reconstruction.

\paragraph{Denoising surface normals.}
\begin{figure}[tbp]\centering
  \begin{subfigure}[t]{.31\textwidth}\centering
    \includegraphics[height=.98\textwidth]{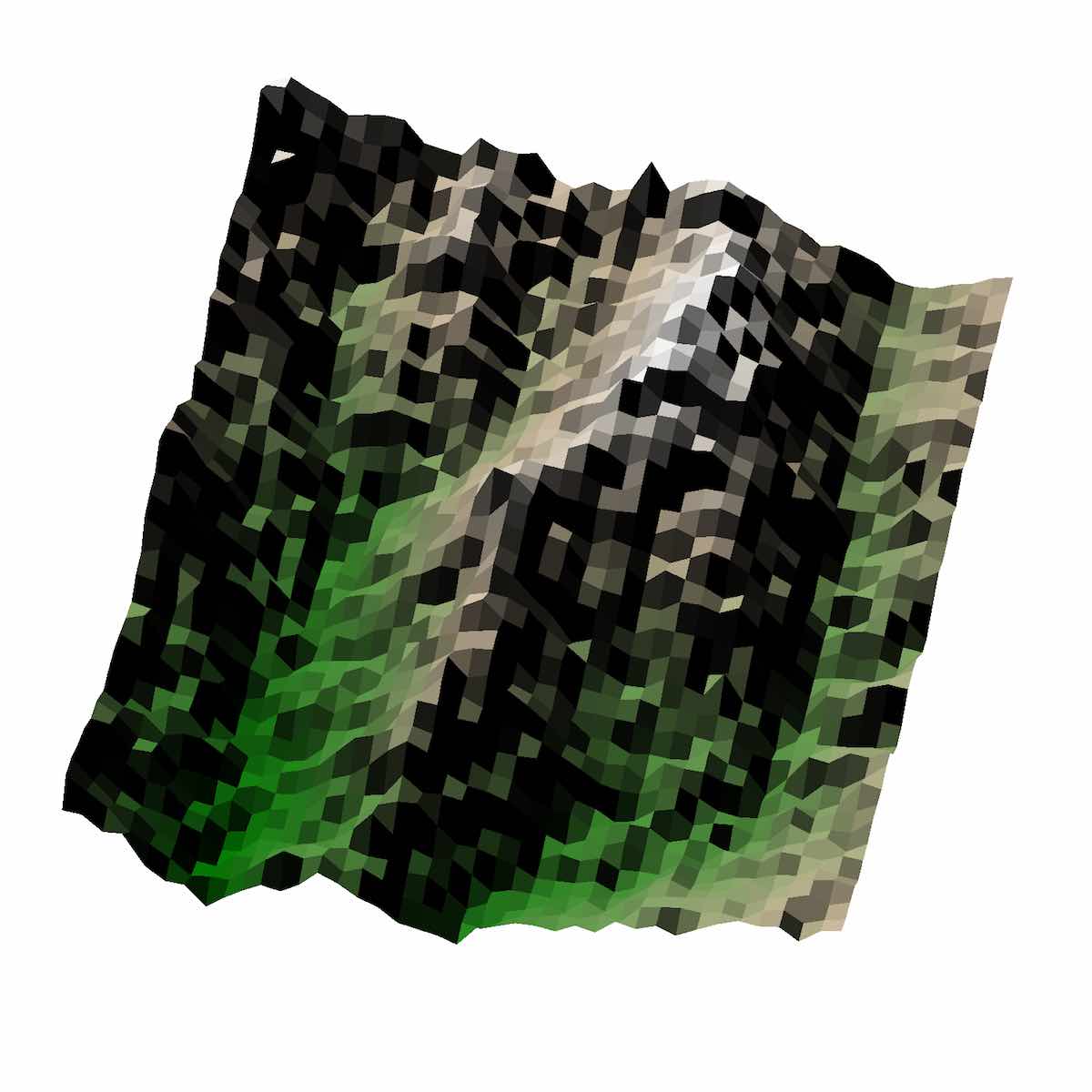}\\[-\baselineskip]
    \includegraphics[height=.825\textwidth]{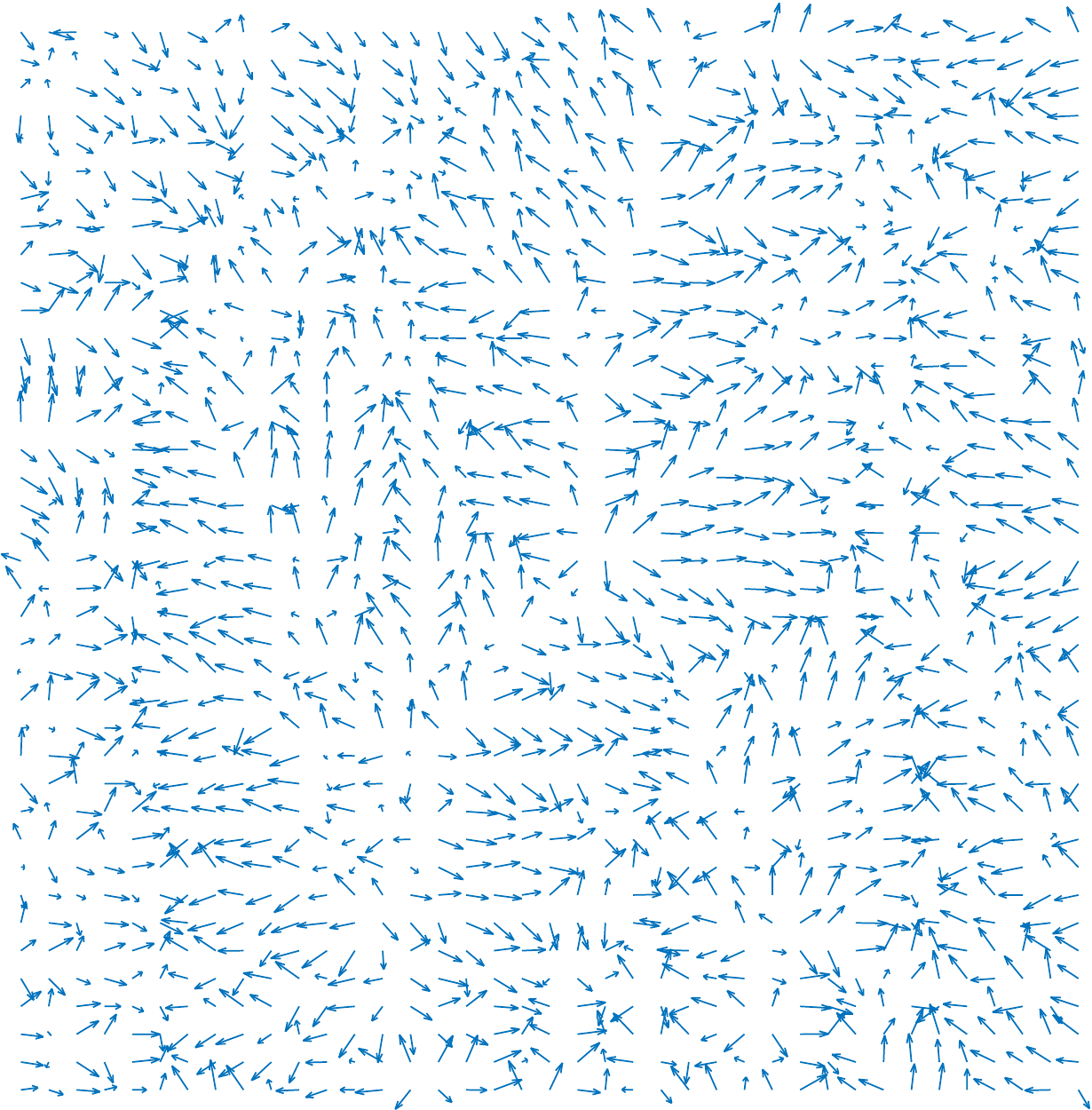}
    \caption{Original NED data.}
    \label{subfig:NED:orig}
  \end{subfigure}
    \begin{subfigure}[t]{.31\textwidth}\centering
    \includegraphics[height=.98\textwidth]{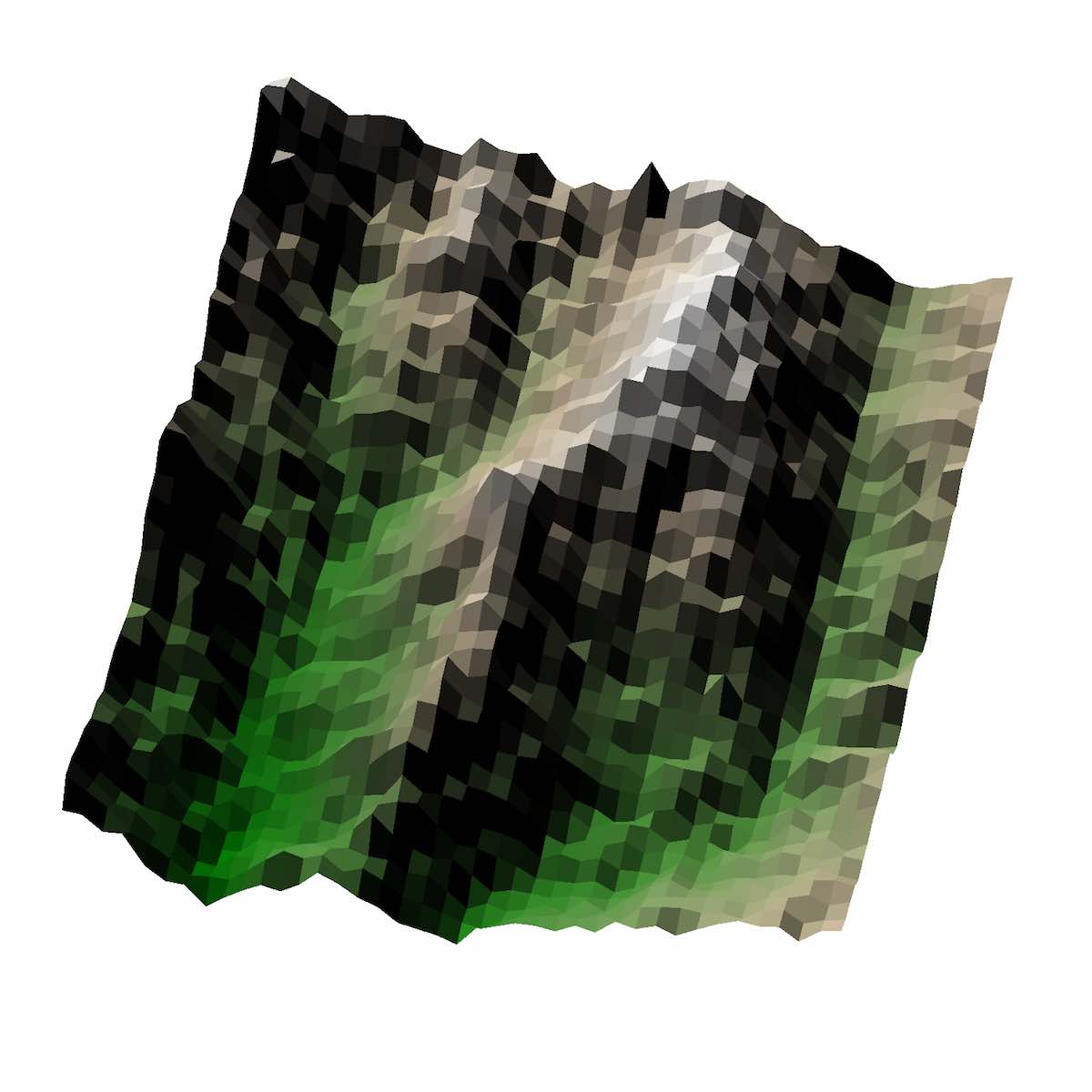}\\[-\baselineskip]
    \includegraphics[height=.825\textwidth]{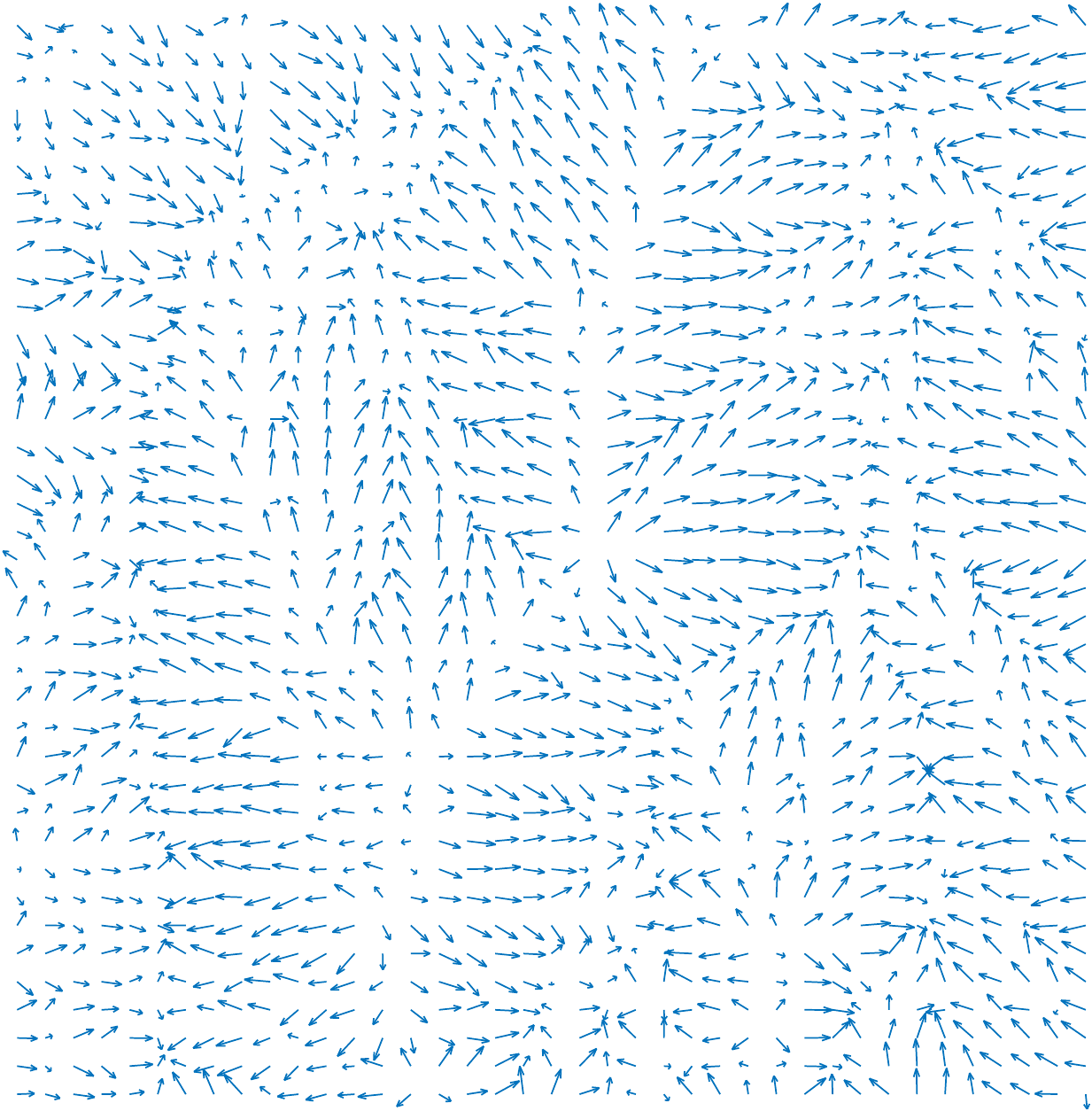}
    \caption{Reconstruction for $p=2$ and $\lambda=5$.}
    \label{subfig:NED:rec_p2_l5}
  \end{subfigure}
  \begin{subfigure}[t]{.31\textwidth}\centering
    \includegraphics[height=.98\textwidth]{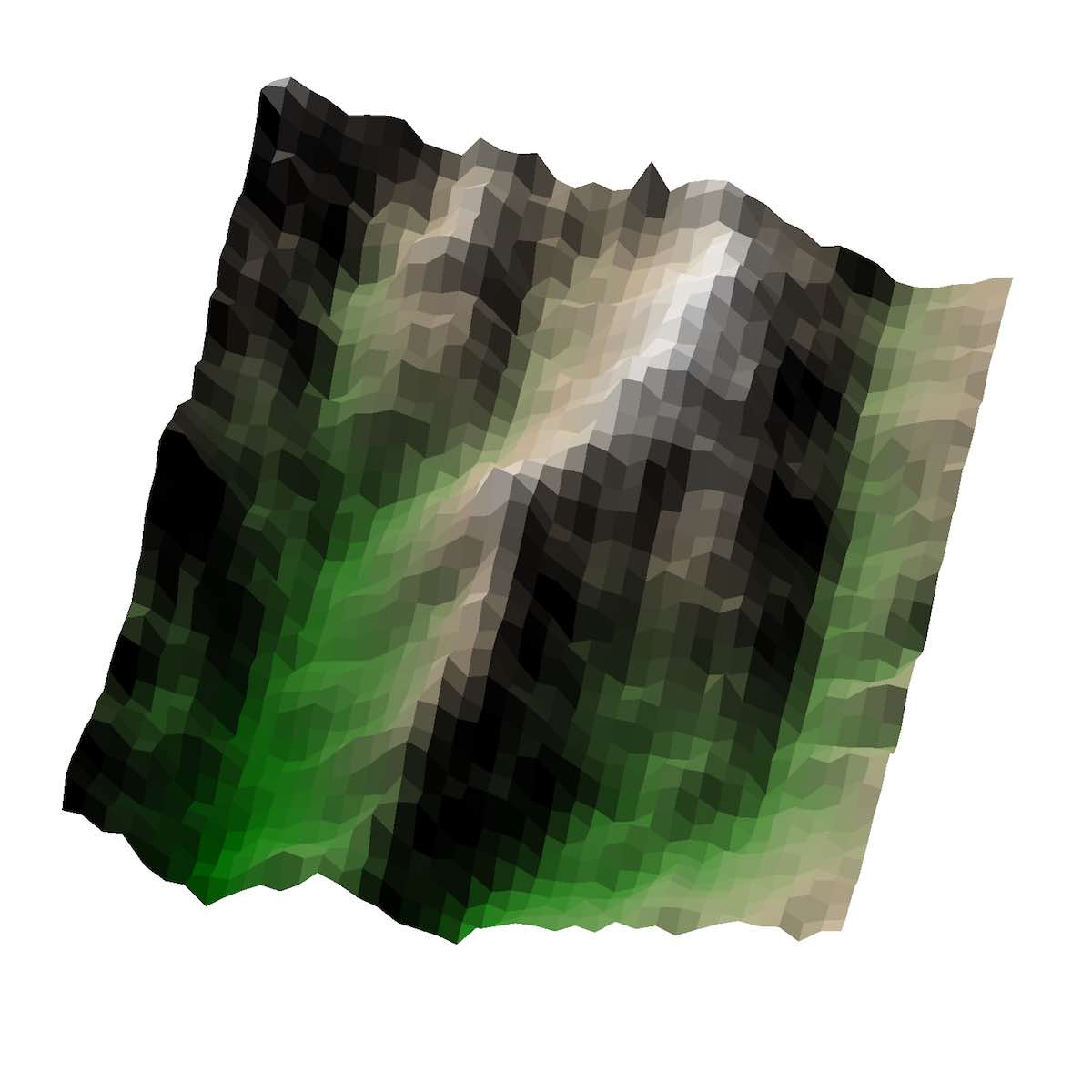}\\[-\baselineskip]
    \includegraphics[height=.825\textwidth]{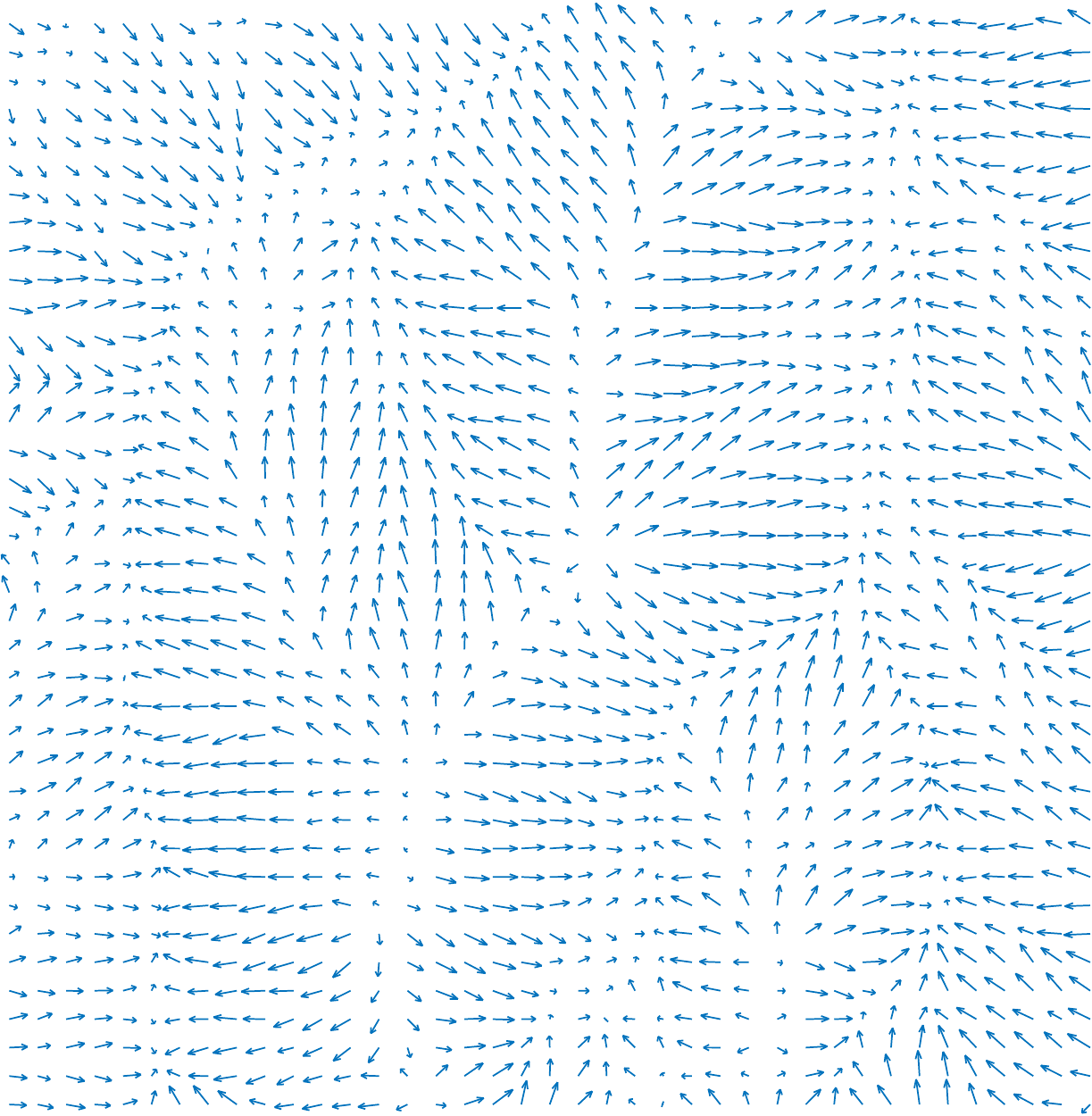}
    \caption{Reconstruction for $p=2$ and $\lambda=0.5$.}
    \label{subfig:NED:rec_p2_l05}
  \end{subfigure}\\
    \begin{subfigure}[t]{.31\textwidth}\centering
    \includegraphics[height=.98\textwidth]{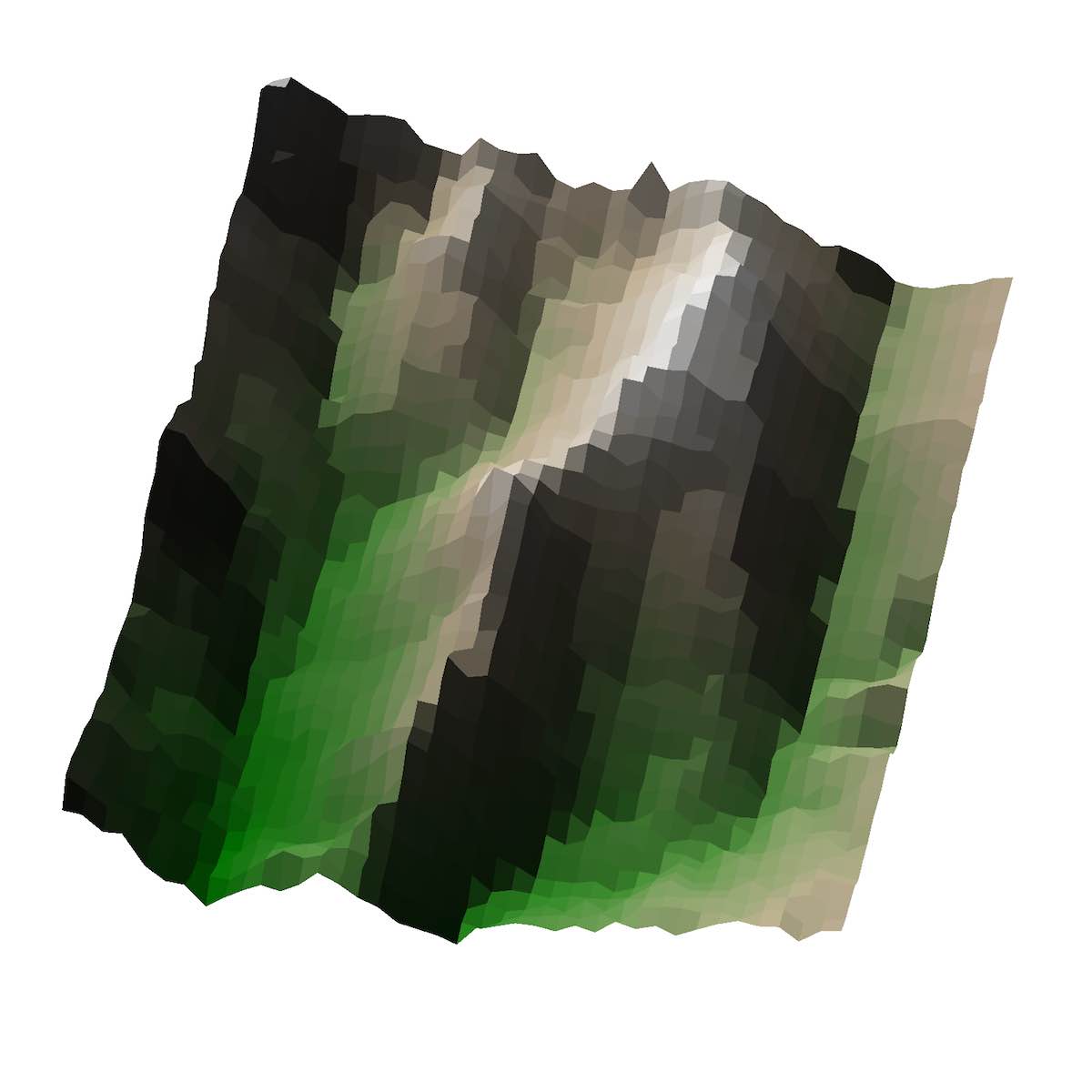}\\[-\baselineskip]
    \includegraphics[height=.825\textwidth]{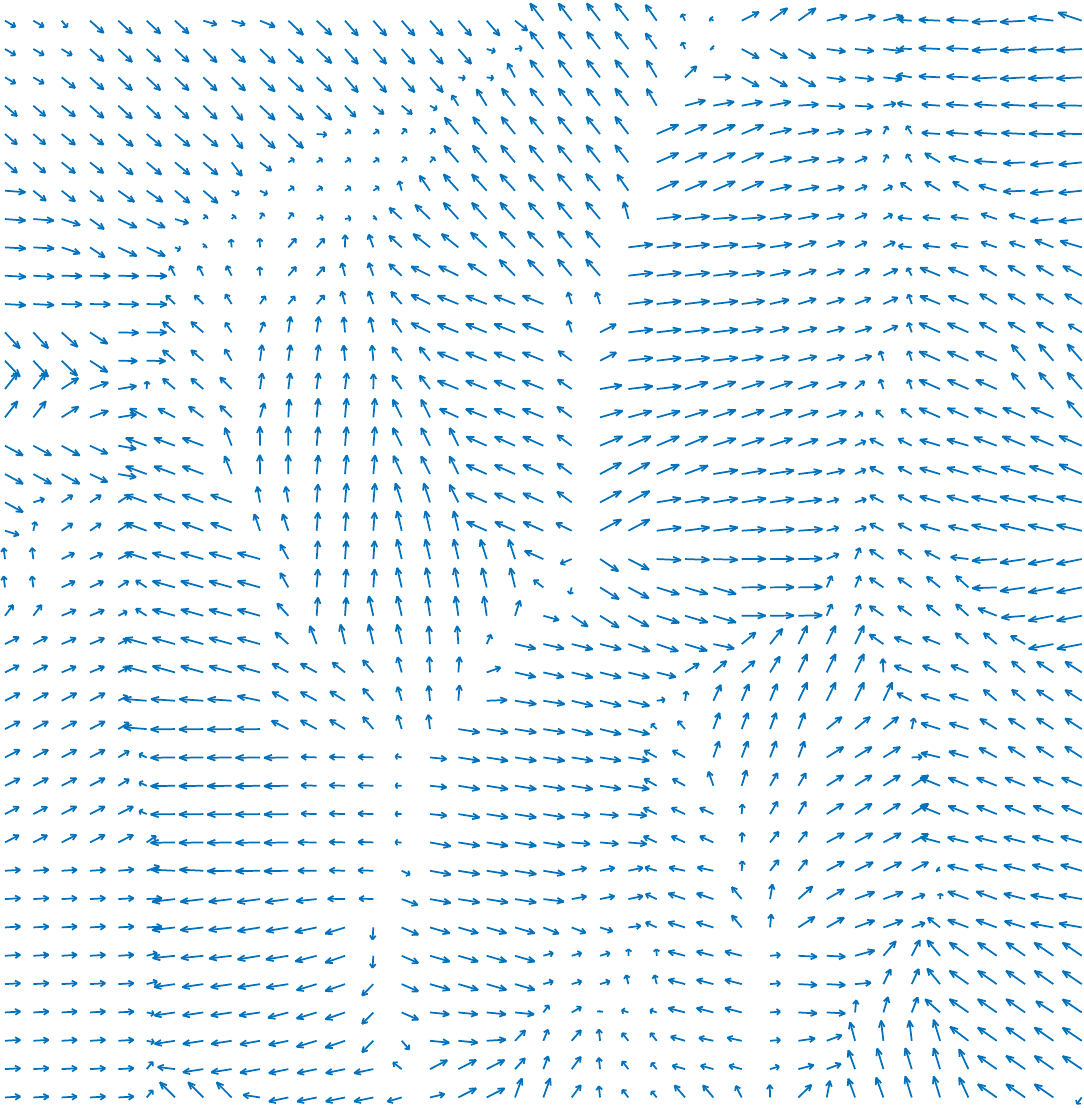}
    \caption{Reconstruction for $p=1$ (anisotropic) and $\lambda=2$.}
    \label{subfig:NED:rec_p1_l2_aniso}
  \end{subfigure}
  \begin{subfigure}[t]{.31\textwidth}\centering
    \includegraphics[height=.98\textwidth]{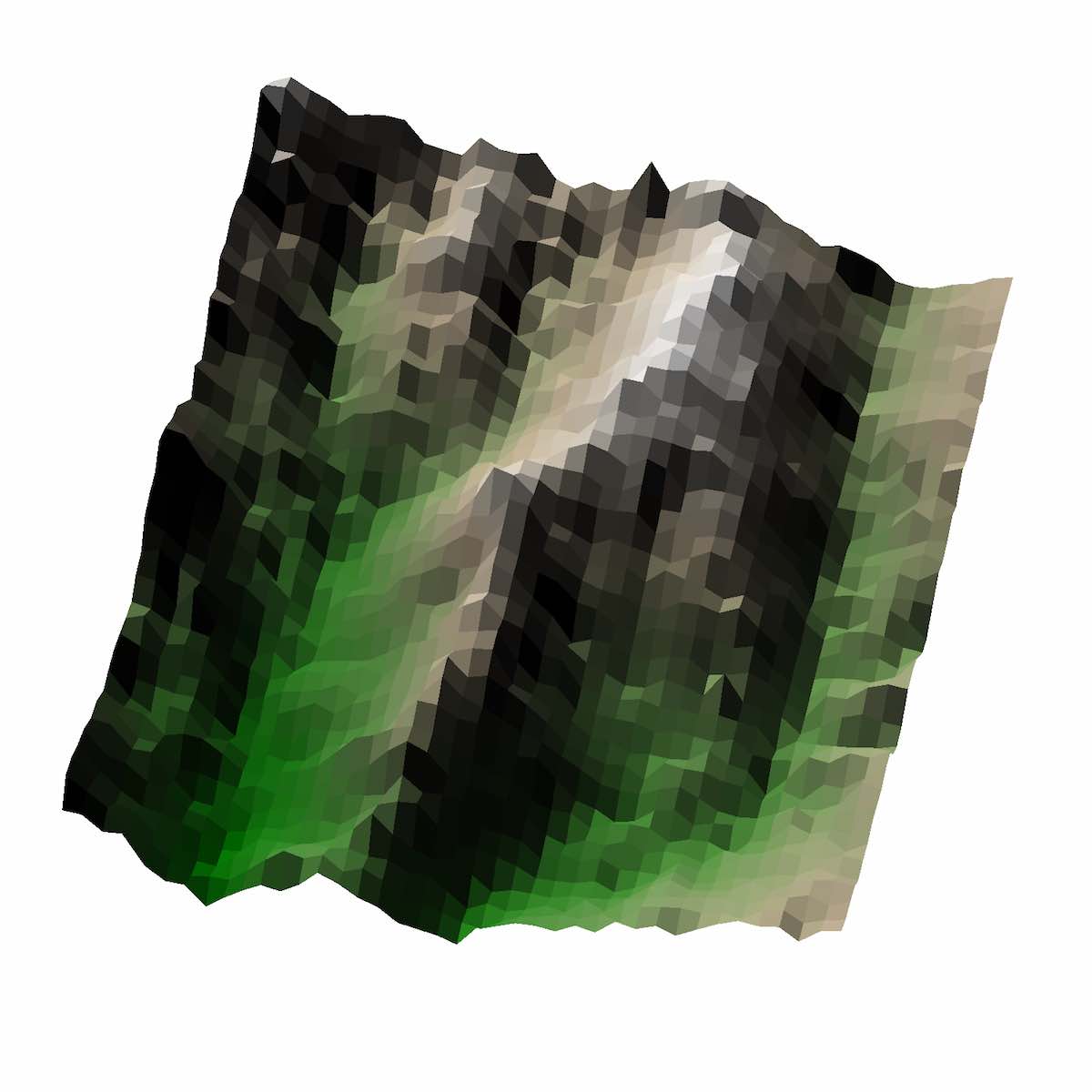}\\[-\baselineskip]
    \includegraphics[height=.825\textwidth]{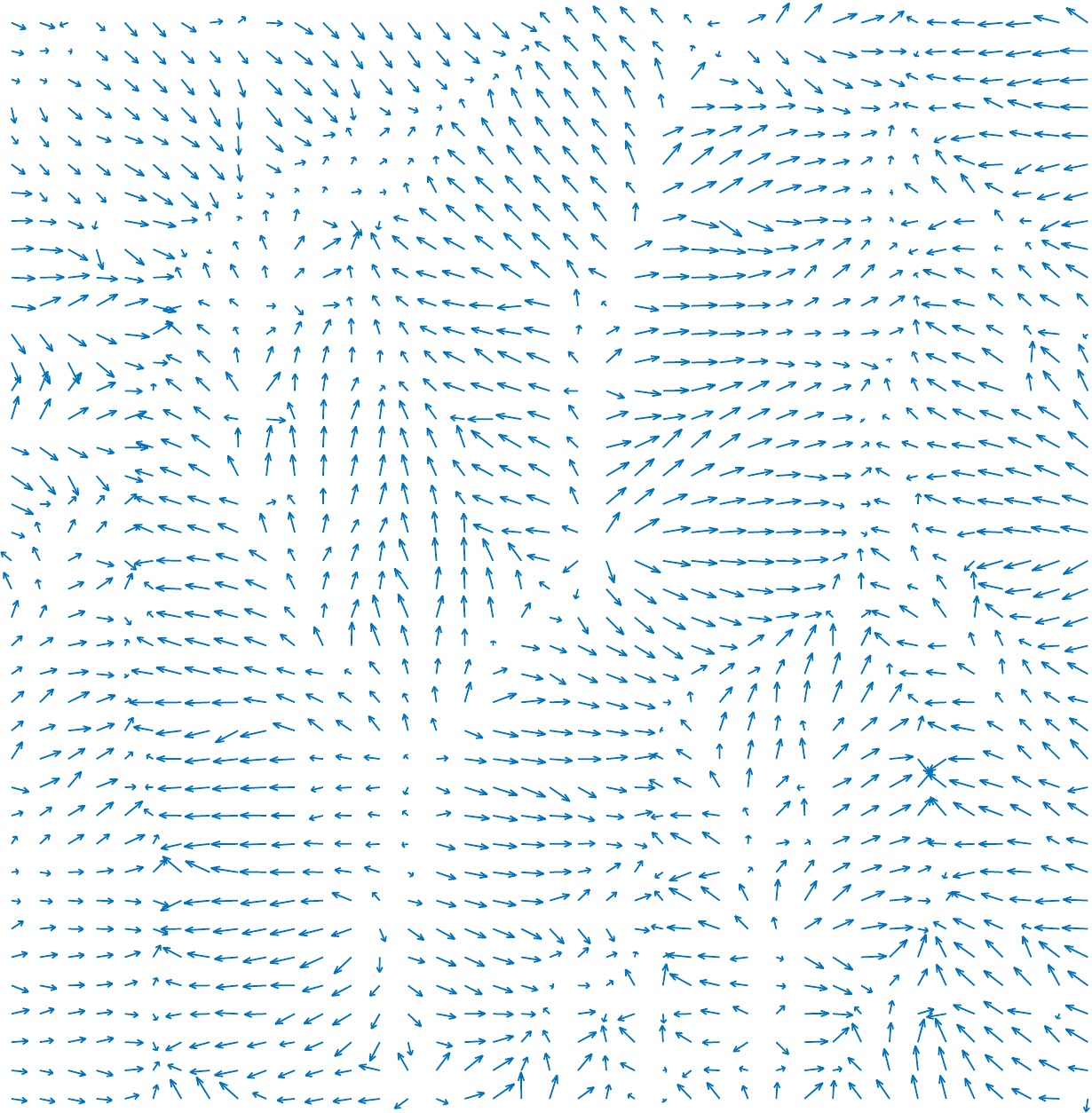}
    \caption{Reconstruction for $p=1$ (isotropic) and $\lambda=2$.}
    \label{subfig:NED:rec_p1_l2_iso}
  \end{subfigure}
  \begin{subfigure}[t]{.31\textwidth}\centering
    \includegraphics[height=.98\textwidth]{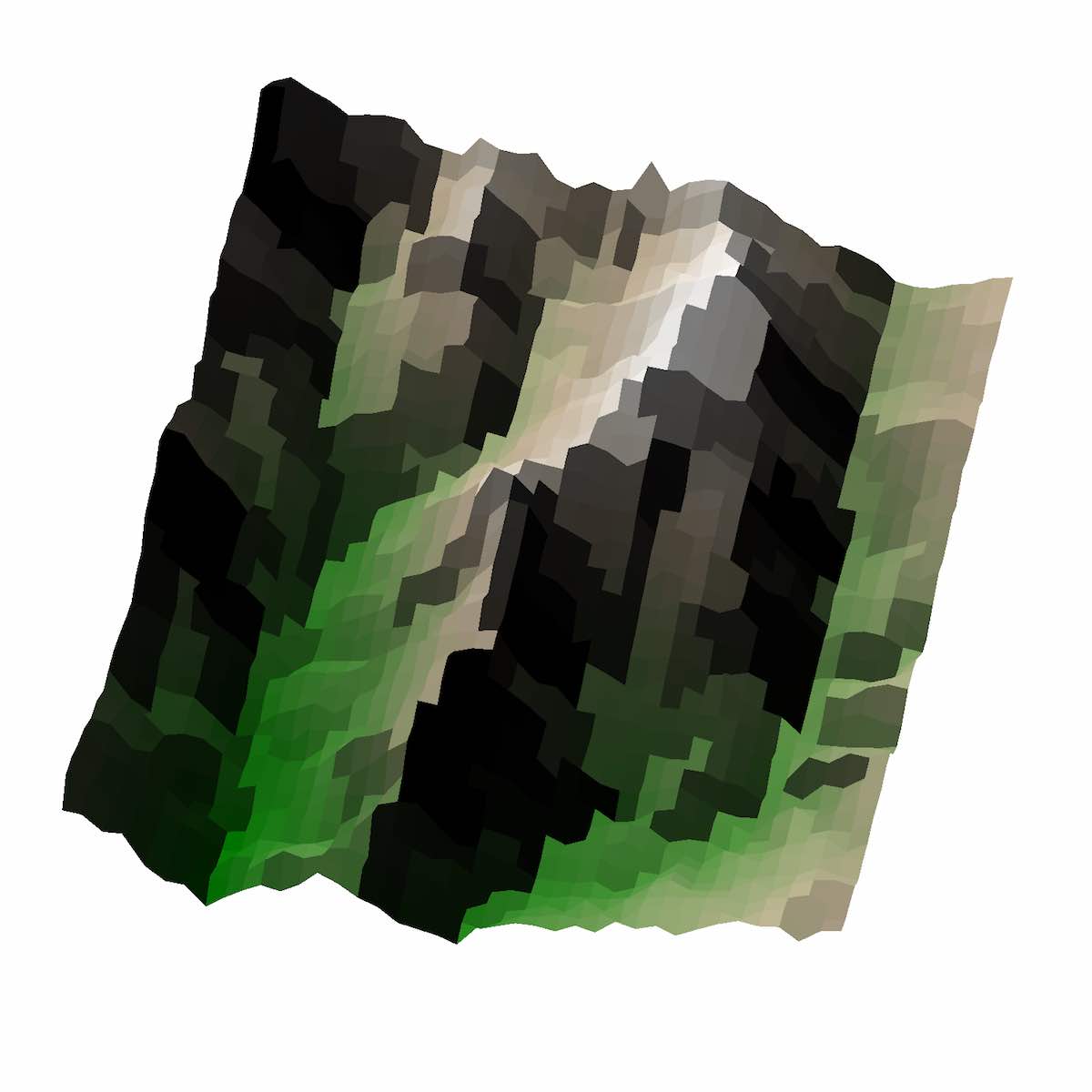}\\[-\baselineskip]
    \includegraphics[height=.825\textwidth]{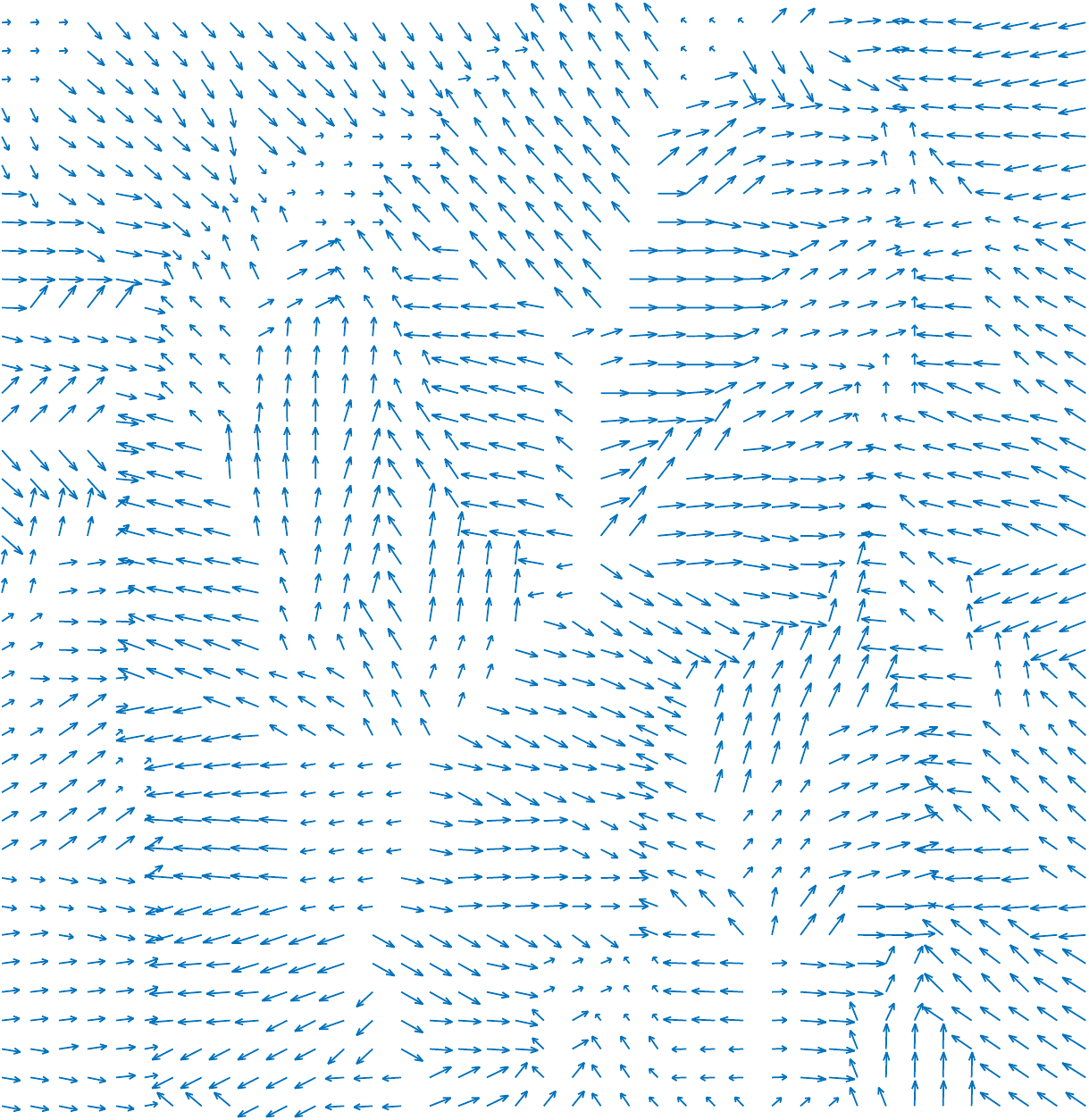}
    \caption{Reconstruction for $p=0.1$ (anisotropic) and $\lambda=1$.}
    \label{subfig:NED:rec_p05_l1_aniso}
  \end{subfigure}
  
  \caption{Reconstruction results of measured surface normals in a DEM and the resulting surface visualization for different denoising parameter settings.}
\end{figure}
Next, we discuss a real world application based on topological surface data from
the National Elevation Dataset (NED)~\cite{GEMHC09}, which was already used for
manifold-valued data processing in~\cite{LMS13,LSKC13} with different denoising
approaches, e.g., lifting for the TV-based regularization approach proposed
in~\cite{LSKC13}.
The provided digital elevation maps (DEM) are generated by light detection and ranging (LiDAR) measurements of earth's surface topology. 
The given DEMs consist of measured surface heights and surface patch normals
corresponding to the heights. The resolution of the image is quite low and hence
one height and surface normal pair is a measurement from a relatively large area of the surface. The normals are usually used to color
the data correctly and thus enhance the visual 3D perception of it.
The original data often gives a perturbed visual impression (see
Figure~\ref{subfig:NED:orig})) due to measurement noise and the complexity of
the underlying real surface, which is only sparsely sampled during data
acquisition. The provided data in the following has a size of
only~\(40 \times 40\) pixels, while covering a large area of a mountain
formation. The processing task is now to reconstruct the measured surface
normals, which live on the manifold \(\mathbb{S}^2\), such that the visual
impression of the DEM is enhanced, while preserving important topological
features, e.g., kinks, and non-smooth regions with high curvature.

As can be seen in Figures~\ref{subfig:NED:rec_p2_l5})
and~\ref{subfig:NED:rec_p2_l05}) using the regular graph Laplacian, i.\,e.~\(p=2\), and the explicit scheme with \(\Delta t= 10^{-4}\), effectively
reduces the noise of the surface normals and thus improves the impression of
the respective visualizations. However, for \(p=2\) one is not able to preserve
the predominant kink of the mountain ridge. For this reason, we demonstrate in
Figures~\ref{subfig:NED:rec_p1_l2_aniso}) and~\ref{subfig:NED:rec_p1_l2_iso})
the effect of reconstructing the DEMs with a total variation-based
regularization
functional for \(p=1\) in the cases of the anisotropic and isotropic graph
\(p\)-Laplacian, respectively. While the top of the mountain ridge is now
well-preserved the resulting normals of the anisotropic TV regularization are
nearly piecewise constant, which leads to a rather blocky visual impression in
Figure~\ref{subfig:NED:rec_p1_l2_aniso}). The isotropic TV regularization in
Figure~\ref{subfig:NED:rec_p1_l2_iso}) on the other hand succeeds in preserving
geological details of the surface, while smoothing the surface normals for an
improved visual impression. As a final comparison we show the results of
non-convex regularization for \(p=\tfrac{1}{2}\) in the anisotropic
graph~\(p\)-Laplacian. It can be observed that the resulting surface normals
field has only very few directions and hence is relatively sparse.
This leads to the visual impression of very sharp edged in
Figure~\ref{subfig:NED:rec_p05_l1_aniso}).
\subsection{Diffusion tensors on the sphere}
\label{ss:applications_sphere}
\begin{figure}
  \begin{subfigure}[t]{.24\textwidth}
    \includegraphics[width=\textwidth]%
      {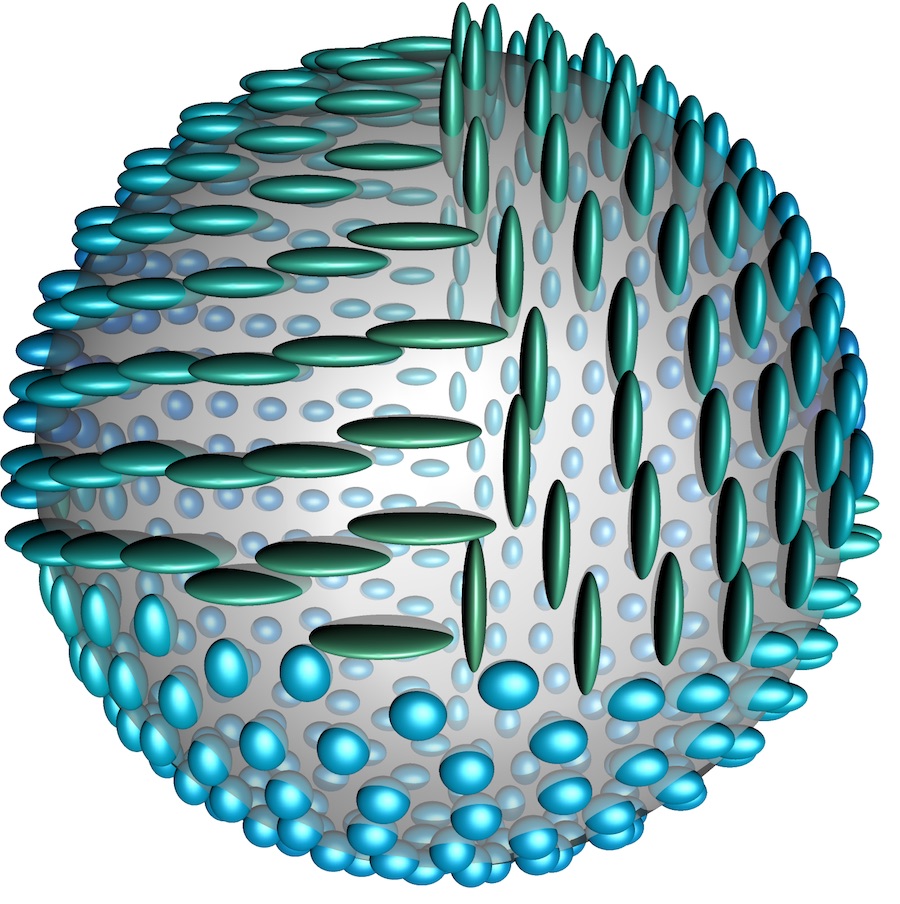}
      \caption{original data}
      \label{subfig:S2SPD:orig}
  \end{subfigure}
  \begin{subfigure}[t]{.24\textwidth}
    \includegraphics[width=\textwidth]%
      {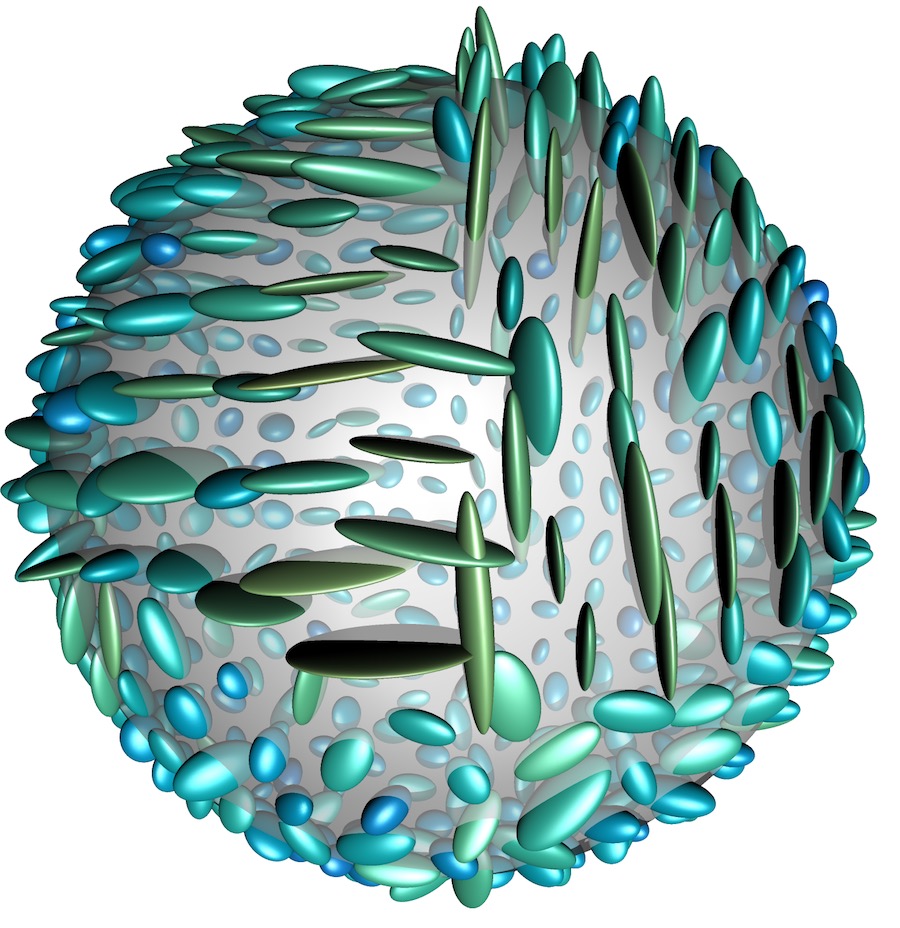}
      \caption{noisy data, Gaussian noise,
      \(\sigma=\frac{1}{4}\); \(\operatorname{MSE} =0.3755\).}
      \label{subfig:S2SPD:noisy}
  \end{subfigure}
  \begin{subfigure}[t]{.24\textwidth}
    \includegraphics[width=\textwidth]%
      {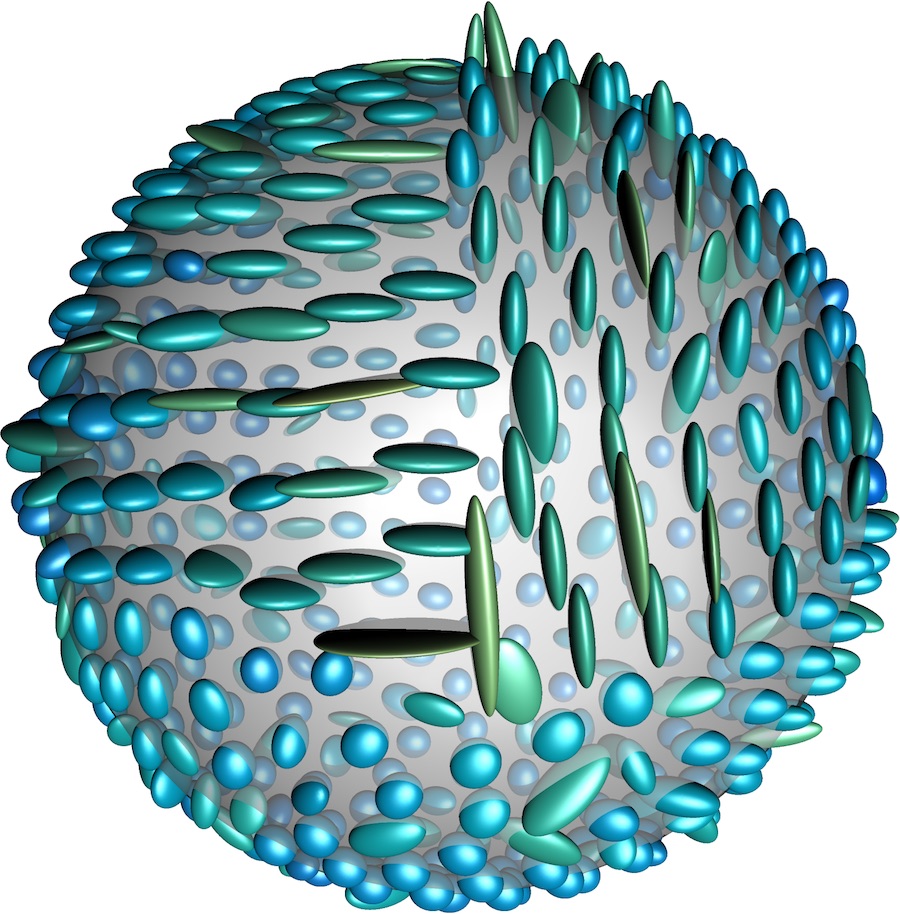}
      \caption{reconstruction,\\\(p=1\),
      \(\lambda=10\); \(\operatorname{MSE} =0.1776\).}
    \label{subfig:S2SPD:p1}
  \end{subfigure}
  \begin{subfigure}[t]{.24\textwidth}
    \includegraphics[width= \textwidth]%
      {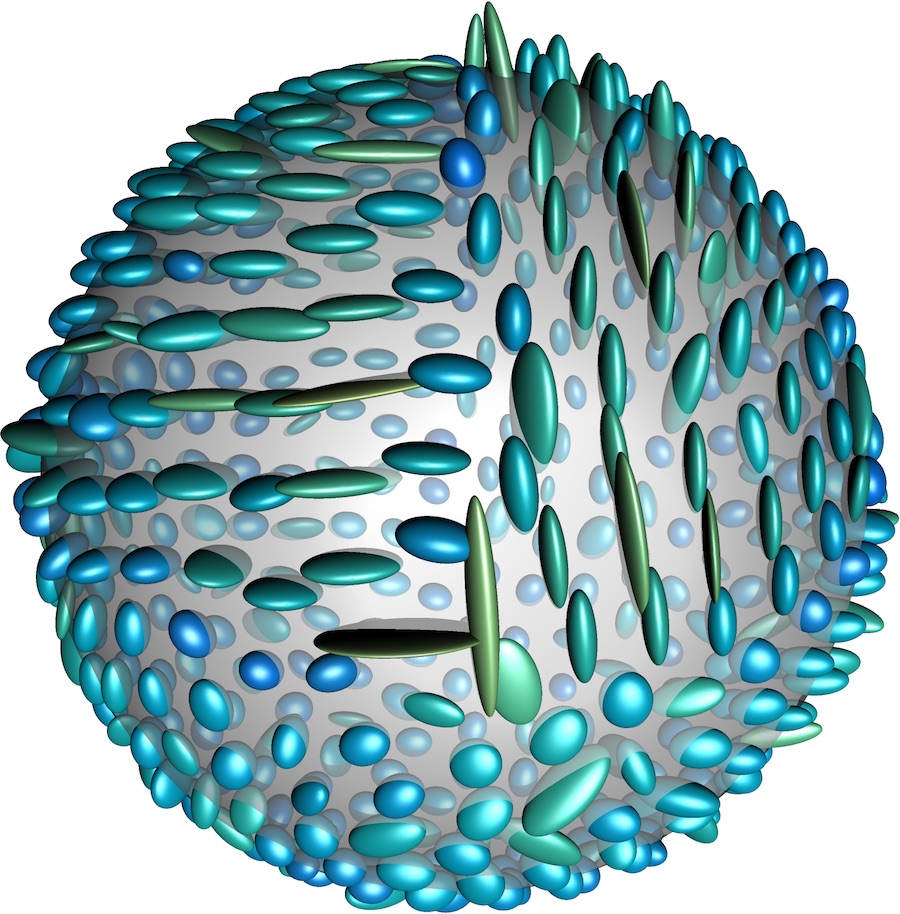}
      \caption{reconstruction,\\\(p=2\), \(\lambda=80\);
      \(\operatorname{MSE} =0.2073\).}
    \label{subfig:S2SPD:p2}
  \end{subfigure}
  \caption{The anisotropic graph \(p\)-Laplacian, \(p=1,2\) for diffusion
  tensors measured equally distributed on a sphere:
  The original~(\subref{subfig:S2SPD:orig}) is obstructed by
  Gaussian~(\subref{subfig:S2SPD:noisy}).
  The denoising using the \(p\)-Laplace for \(p=1\) keeps the edges,
  while the case \(p=2\) approach reduces the MSE, it also smoothes the edges.}
  \label{fig:S2SPD}
\end{figure}
To process manifold-valued data defined on a manifold itself, we investigate
diffusion tensors defined on the unit sphere \(\mathbb S^2\). For sampling we use 
the quadrature points of Chebyshev-type, which were computed in~\cite[Example 6.19]{Graef13}
corresponding to polynomial degree \(30\), i.e., on \(M=480\) points.
These are available on the homepage of the software package accompanying the
thesis\footnote{see~\href{http://homepage.univie.ac.at/manuel.graef/quadrature.php}%
{homepage.univie.ac.at/manuel.graef/quadrature.php}}~\cite{Graef13}.
We assign to each sampling point a symmetric positive definite matrix,
which yields the original (unperturbed) dataset shown in
Figure~\ref{subfig:S2SPD:orig}). The illustrated ellipsoids on the sphere are colored with
respect to the geodesic anisotropy, cf.~\cite[p.~290]{MB06}, employing the
colormap~\lstinline!viridis!.
This artificial data set is then perturbed
by i.i.d.~Gaussian noise, see e.g.~\cite{LPS16}, with standard
deviation~\(\sigma=\frac{1}{4}\), which yields the noisy data in
Fig.~\ref{subfig:S2SPD:noisy}).
Each of the sampling points on the sphere
\(p_i\in\mathbb S^2\), \(i=1,\ldots,480\), constitutes one vertex~\(v_i\in V\)
of the finite weighted graph \(G=(V,E,w)\). We construct an $\varepsilon-$ball graph by 
connecting to vertices \(v_i,v_j \in V\) by an edge~\((v_i,v_j)\in E\)
if \(d_{\mathbb S^2}(p_i,p_j) \leq \varepsilon\) with \( \varepsilon = \frac{\pi}{12}\) for this
experiment. This results in \(3648\) edges, i.e., each node has in average~\(7.6\)
neighbors.
The edges are weighted with the arc length distance on~\(\mathbb S^2\) of their
incident nodes, i.e., \(w(i,j) = d_{\mathbb S^2}(p_i,p_j)^{-2}\).

We then employ the aforementioned Jacobi iteration methods with~\(N=10,000\)
iterations to compute a denoised reconstruction of the noisy data using the
anisotropic graph \(p-\)Laplace for~\(p=1\) and~\(p=2\) with parameters 
\(\lambda=23\) and \(\lambda=80\), respectively.
These parameters where chosen among all
integers~\(\lambda\in\mathbb Z\cap\{1,\ldots,100\}\) such that the MSE is minimized.
The results of this approach are shown in Figure~\ref{subfig:S2SPD:p1})
and~\subref{subfig:S2SPD:p2}).
While the TV regularization for \(p=1\) preserves edges,
the Tikhonov regularization for \(p=2\) introduces smooth reconstructions,
which leads to higher values of the MSE.

\subsection{Diffusion tensor imaging}
\label{s:DT-MRI}
A special case of data taking values on the Riemannian
manifold~\(\mathcal M = \mathcal P(3)\) of symmetric positive
definite~\(3\times3\)-matrices are diffusion tensors as obtained in diffusion
tensor magnetic resonance imaging (DT-MRI).
For this experiment we use the open data set from the Camino
project\footnote{see \url{http://camino.cs.ucl.ac.uk}}~\cite{Camino}.
The provided data \(f_0\in (\mathcal P(3))^{112\times112\times50}\) of a human brain
consists of~\(50\) transversal slices, each contained in an array of
size~\(112\times112\). Values outside the brain are masked and take zero-matrices as
values.

\begin{figure}[tbp]\centering
  \begin{subfigure}{.470\textwidth}\centering
    \includegraphics[width=.49\textwidth]{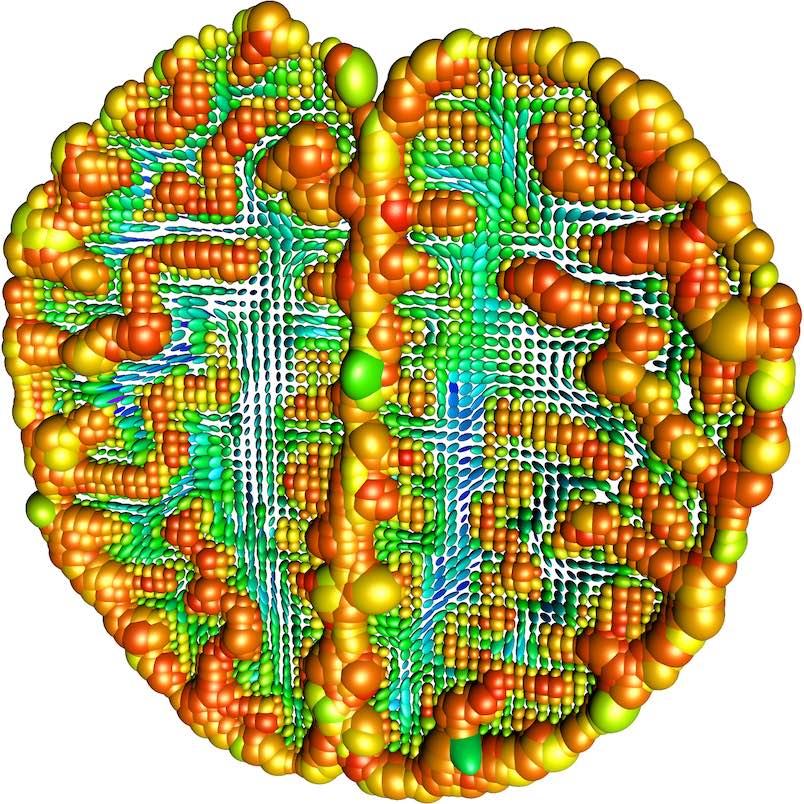}
    \caption{Camino slice \(z=35\).}
    \label{subfig:CaminoNL:orig}
  \end{subfigure}
  \\
  \begin{subfigure}{.03\textwidth}\centering
    \rotatebox{90}{\small\(k=10\), anisotropic}
  \end{subfigure}
  \begin{subfigure}{.235\textwidth}\centering
\includegraphics[width=.98\textwidth]{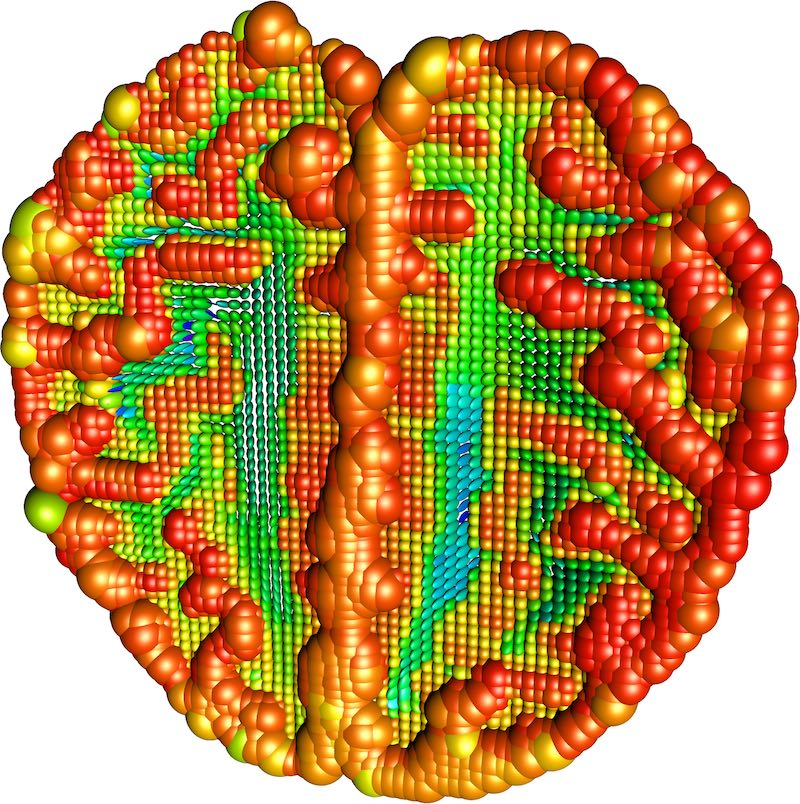}
    \caption{\(s=1\), \(\lambda=10\).}\label{subfig:CaminoNL:a-s1l10}
  \end{subfigure}
  \begin{subfigure}{.235\textwidth}\centering
\includegraphics[width=.98\textwidth]{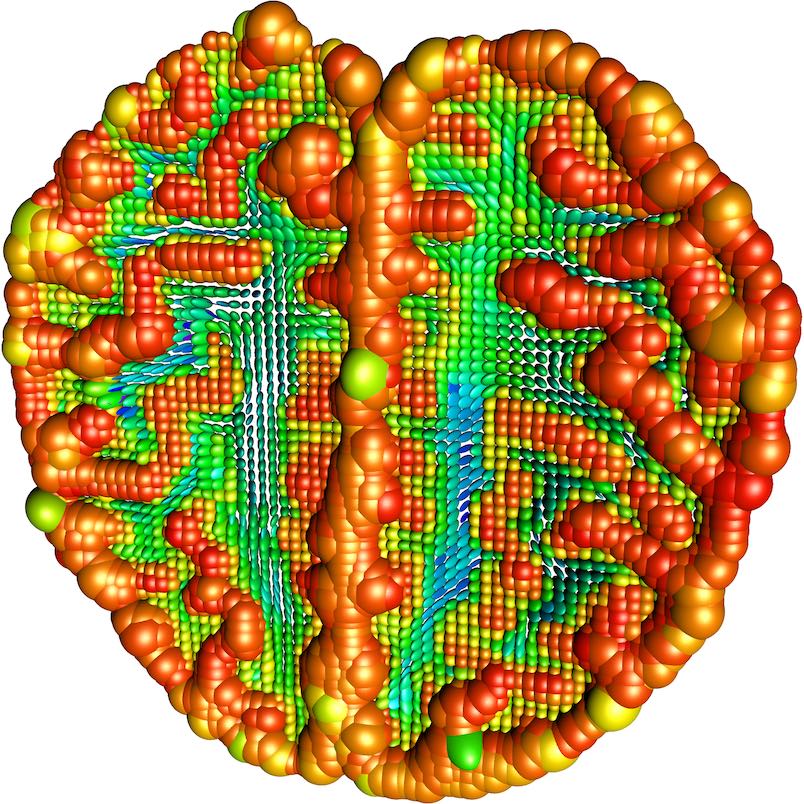}
    \caption{\(s=1\), \(\lambda=25\).}\label{subfig:CaminoNL:a-s1l25}
  \end{subfigure}
  \begin{subfigure}{.235\textwidth}\centering
\includegraphics[width=.98\textwidth]{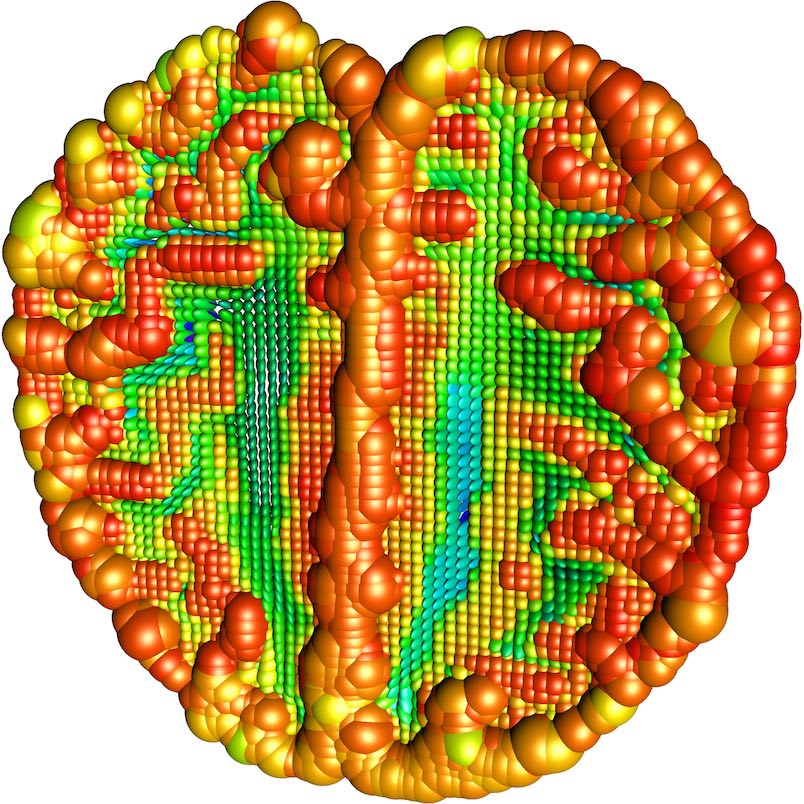}
    \caption{\(s=2\), \(\lambda=10\).}\label{subfig:CaminoNL:a-s2l10}
  \end{subfigure}
  \begin{subfigure}{.235\textwidth}\centering
\includegraphics[width=.98\textwidth]{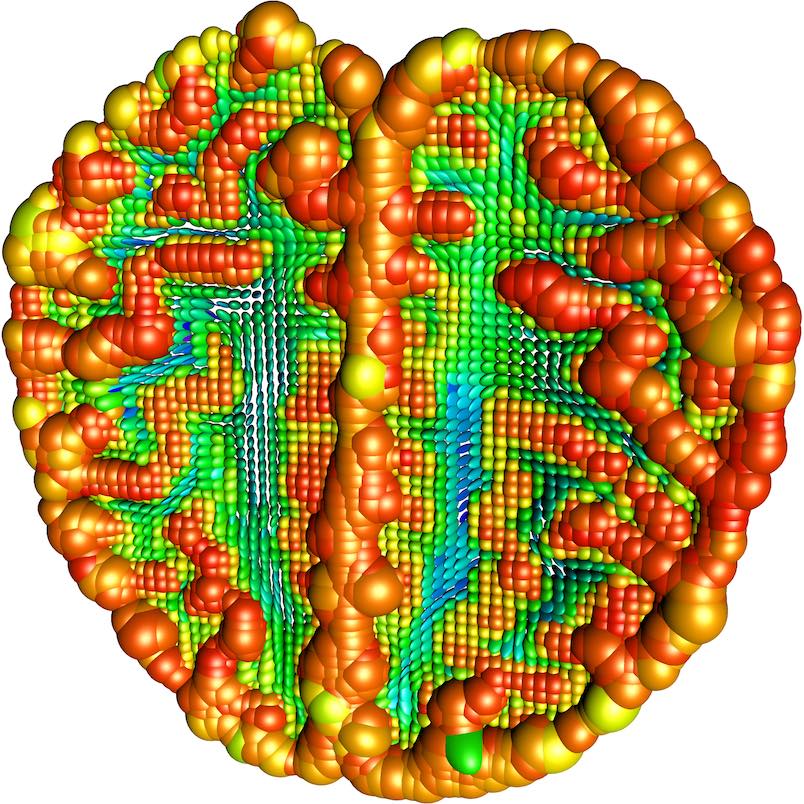}
    \caption{\(s=2\), \(\lambda=25\).}\label{subfig:CaminoNL:a-s2l25}
  \end{subfigure}
  \\
  \begin{subfigure}{.03\textwidth}\centering
    \rotatebox{90}{\small\(k=10\), isotropic}
  \end{subfigure}
  \begin{subfigure}{.235\textwidth}\centering
\includegraphics[width=.98\textwidth]{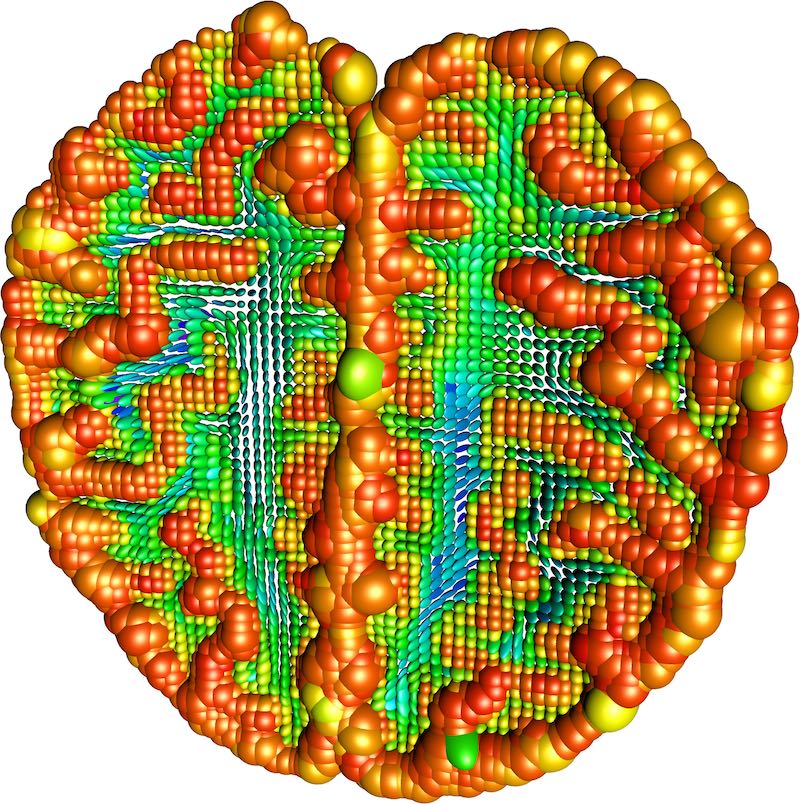}
    \caption{\(s=1\), \(\lambda=10\).}\label{subfig:CaminoNL:i-s1l10}
  \end{subfigure}
  \begin{subfigure}{.235\textwidth}\centering
\includegraphics[width=.98\textwidth]{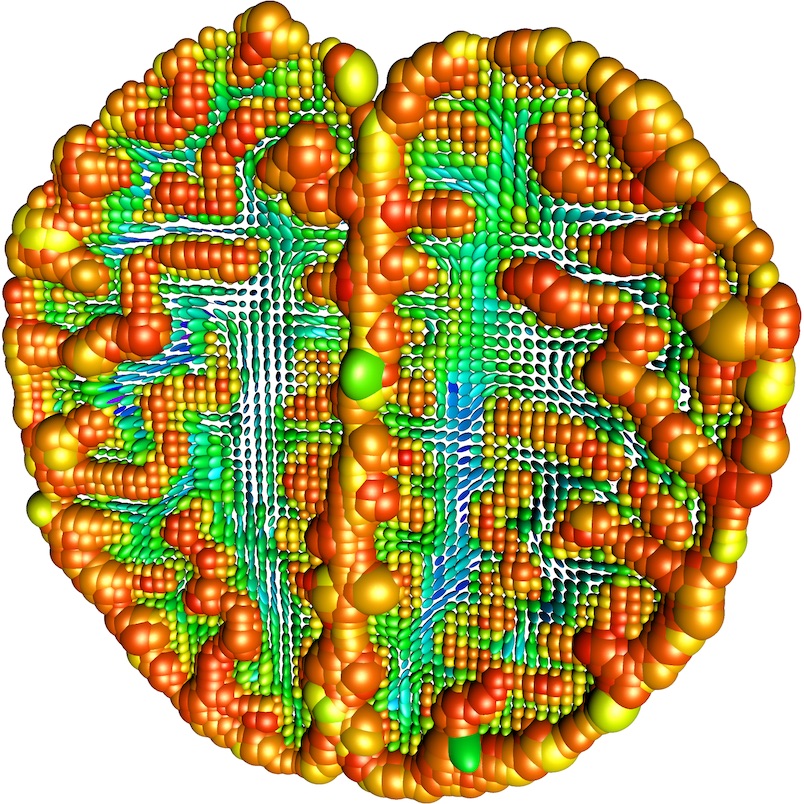}
    \caption{\(s=1\), \(\lambda=25\).}\label{subfig:CaminoNL:i-s1l25}
  \end{subfigure}
  \begin{subfigure}{.235\textwidth}\centering
\includegraphics[width=.98\textwidth]{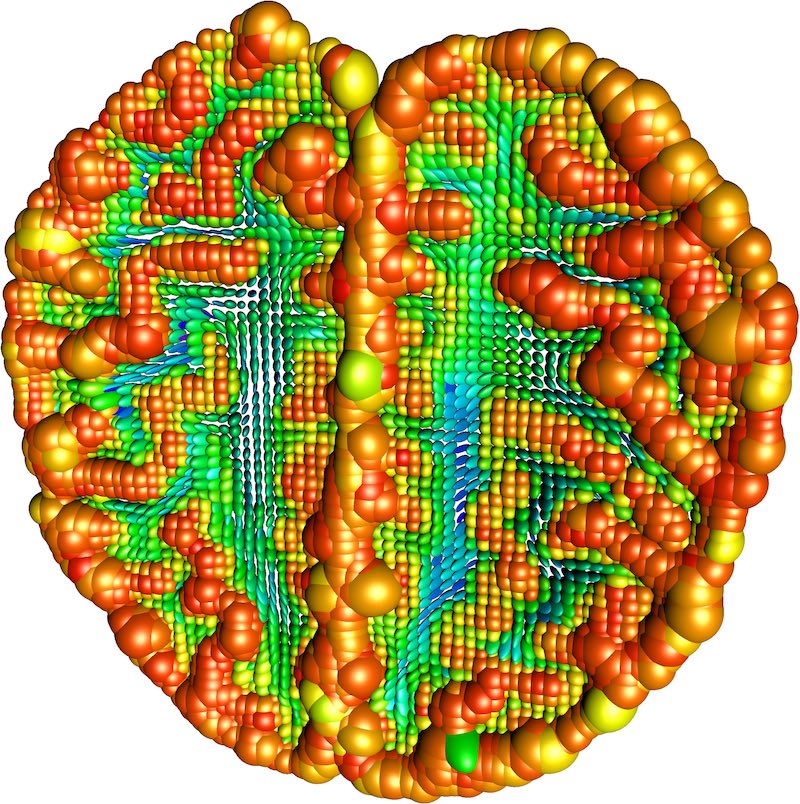}
    \caption{\(s=2\), \(\lambda=10\).}\label{subfig:CaminoNL:i-s2l10}
  \end{subfigure}
  \begin{subfigure}{.235\textwidth}\centering
\includegraphics[width=.98\textwidth]{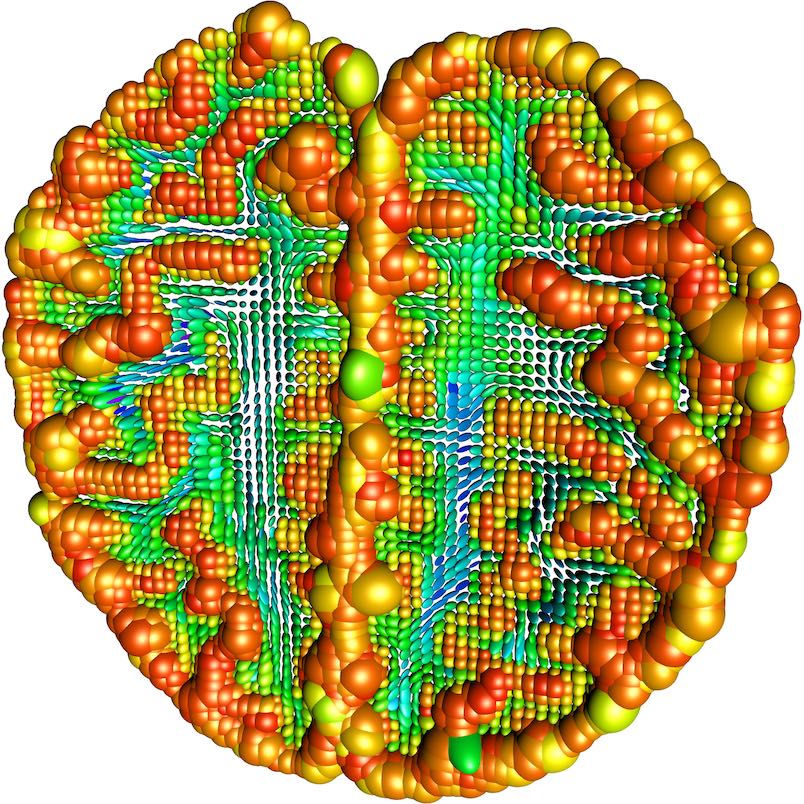}
    \caption{\(s=2\), \(\lambda=25\).}\label{subfig:CaminoNL:i-s2l25}
  \end{subfigure}
  \caption{Reconstruction of a transversal slice~\(z=35\) of the Camino DT-MRI data set
  with two patch-based kNN graphs, \(k=10\), of patch size \(2s+1, s\in\{1,2\}\),
  and different regularization parameters \(\lambda\) using the anisotropic model~\eqref{eq:anisomodel}
  (first row) and the isotropic variant~\eqref{eq:isomodel} (second row).}
  \label{fig:CaminoNL}
\end{figure}
\paragraph{Nonlocal denoising.} In our first experiment on the Camino data 
we extract the transversal slice \(z=35\), see Figure~\ref{subfig:CaminoNL:orig}).
Analogously to the artificial InSAR experiment with phase-valued data in Section \ref{ss:applications_images}, we 
perform nonlocal TV-based denoising based on a \(k-\)nearest neighbor graph \(G=(V,E,w)\) constructed
for \(k=10\) from patches of size \(2s+1\), \(s\in\{1,2\}\),
cf.~\eqref{eq:patchdistance}.
We use the same linearly interpolated weight function for nonlocal denoising
as already described for the artificial InSAR experiment discussed above.
This experiment is the first to apply nonlocal TV-based approaches for manifold-valued data
on a real medical dataset.
Since the image consists also of masked data values (the surrounding of the brain),
those pixels are ignored during computation of the patch-distance the \(k-\)nearest neighbors.
For minimization we employ the semi-implicit minimization scheme using Jacobi iterations from \eqref{eq:Jacobi-iterate}.
In Figure~\ref{subfig:CaminoNL:a-s1l10})--\subref{subfig:CaminoNL:a-s2l25}) the results 
of the anisotropic denoising approach are shown with a maximal number~\(N=100\) of iterations 
unless the average relative change~\(\frac{1}{\lvert V\rvert} \sum_{u\in V} d_{\mathcal M}(f_{n-1},(u),f_n(u)) < 10^{-7}\) is met earlier.
For the isotropic case the results are shown in
Figure~\ref{subfig:CaminoNL:i-s1l10})--\subref{subfig:CaminoNL:i-s2l25}).
For both experiments we chose \(\lambda \in\{10,25\}\), yielding eight results in
total. The strongest regularization effects can be observed for~\(s=1, \lambda=10\),
cf.~Figure~\ref{subfig:CaminoNL:a-s1l10}) and~\subref{subfig:CaminoNL:i-s1l10}).
As can be seen the anisotropic model introduces more smoothing within the nonlocal
neighborhood than the isotropic one. Increasing \(\lambda\) to \(25\),
cf.~Figure~\ref{subfig:CaminoNL:a-s1l25}) and~\subref{subfig:CaminoNL:i-s1l25}),
and hence increasing the importance of the data fidelity term reduces the smoothing but also
introduces a single outlier in the middle of the slice. Furthermore, the isotropic
case already seems to keep a little bit of noise in the left half of the brain.
Increasing the patch size to \(s=2\), see Figures~\ref{subfig:CaminoNL:a-s2l10}),
\subref{subfig:CaminoNL:i-s2l10}), \subref{subfig:CaminoNL:a-s2l25}),
and~\subref{subfig:CaminoNL:i-s2l25}), leads to results which are more smoothed within
anatomical structures, e.g.,~the longitudinal fissure in the anisotropic case, \(\lambda=10\),
in Figure~\ref{subfig:CaminoNL:a-s2l10}) is denoised quite nicely. On the other hand
larger patches also tend to oversmooth small details and fine structures within the data.

\paragraph{Denoising on an implicitly given surface.}
\begin{figure}[tb]\centering
  \begin{subfigure}{.32\textwidth}\centering
    \includegraphics[width=.98\textwidth]{CaminoSurface-orig-h37p5-r35}
    \caption{Surface data, front left view.}\label{subfig:CaminoS:orig}
  \end{subfigure}
   \\ 
   \begin{subfigure}{.32\textwidth}\centering
\includegraphics[width=.98\textwidth]{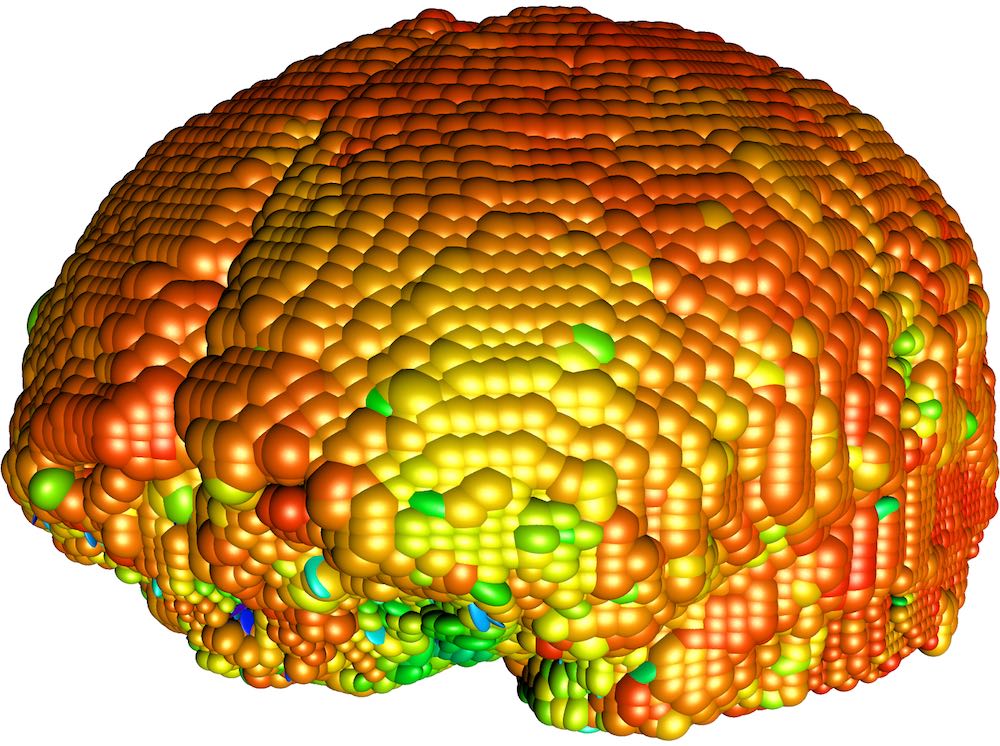}
    \caption{\(\varepsilon=2\), \(\lambda=10\).}\label{subfig:CaminoS:e2l10}
  \end{subfigure}
   \begin{subfigure}{.32\textwidth}\centering
\includegraphics[width=.98\textwidth]{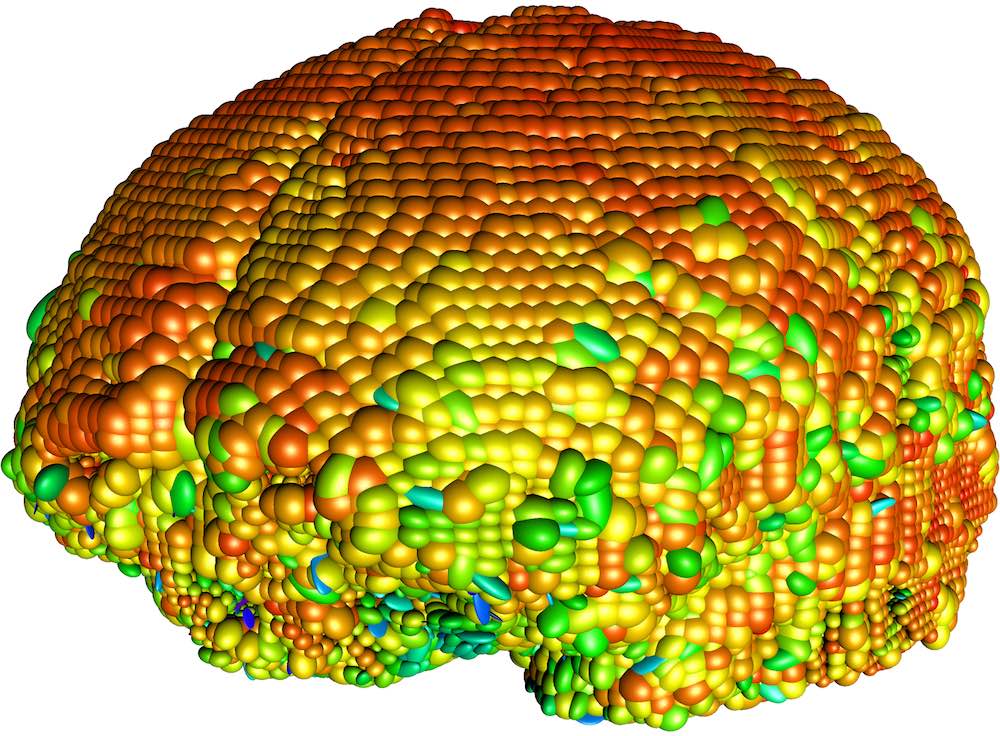}
    \caption{\(\varepsilon=2\), \(\lambda=25\).}\label{subfig:CaminoS:e2l25}
  \end{subfigure}
  \begin{subfigure}{.32\textwidth}\centering
    \includegraphics[width=.98\textwidth]{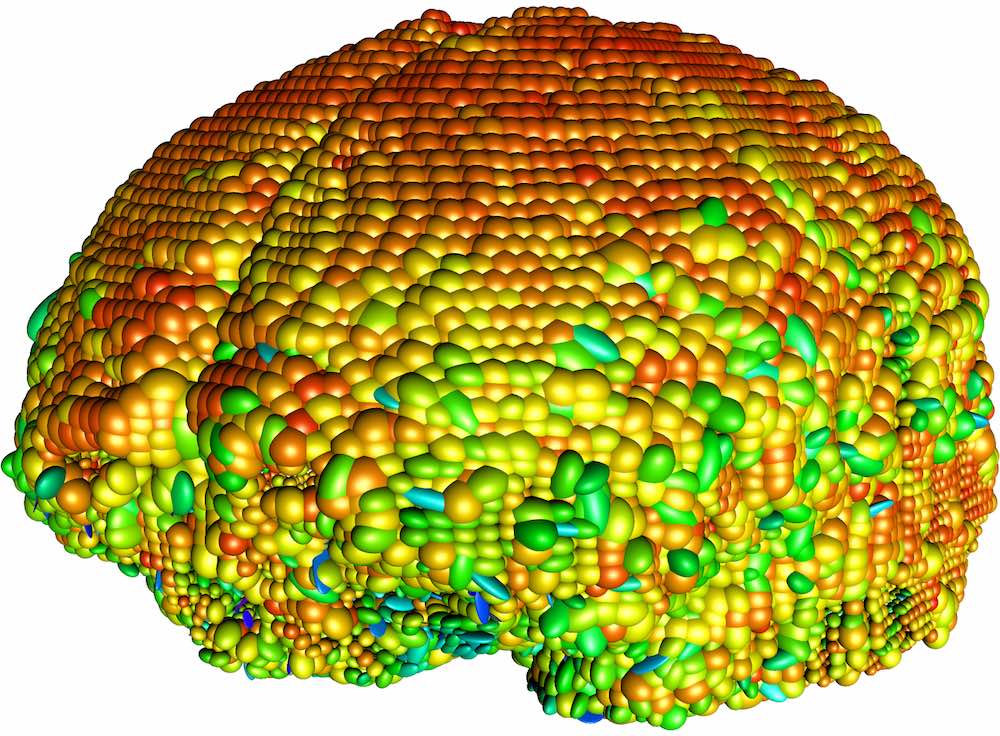}
    \caption{\(\varepsilon=2\), \(\lambda=50\).}\label{subfig:CaminoS:e2l50}
  \end{subfigure}
  \\
  \begin{subfigure}{.32\textwidth}\centering
    \includegraphics[width=.98\textwidth]{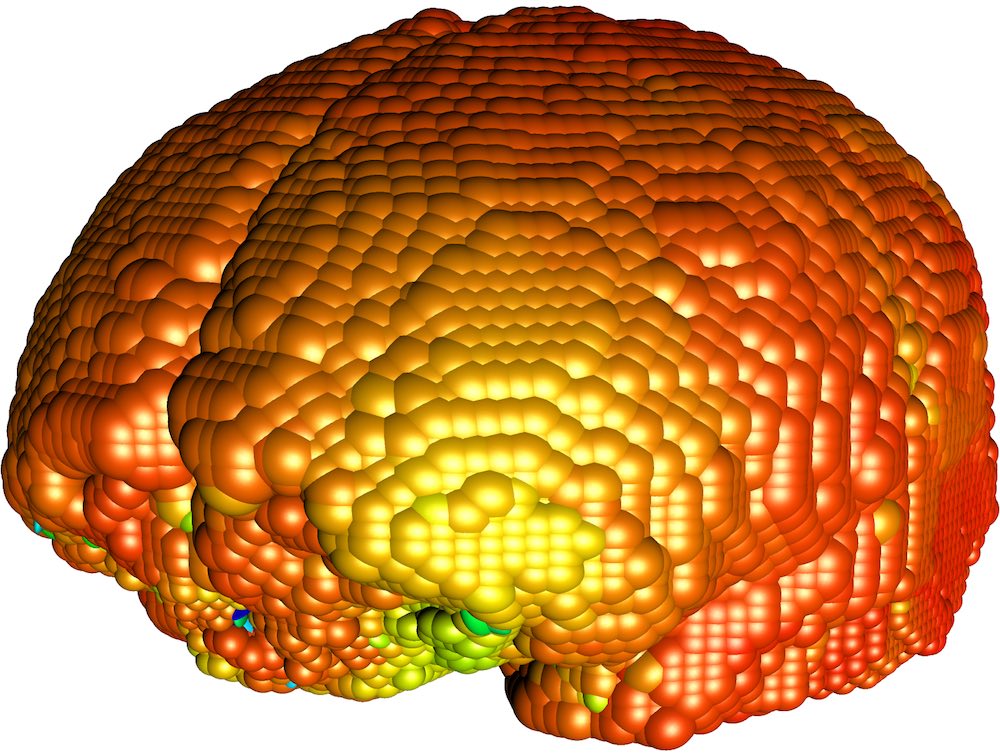}
    \caption{\(\varepsilon=3\), \(\lambda=10\).}\label{subfig:CaminoS:e3l10}
  \end{subfigure}
  \begin{subfigure}{.32\textwidth}\centering
    \includegraphics[width=.98\textwidth]{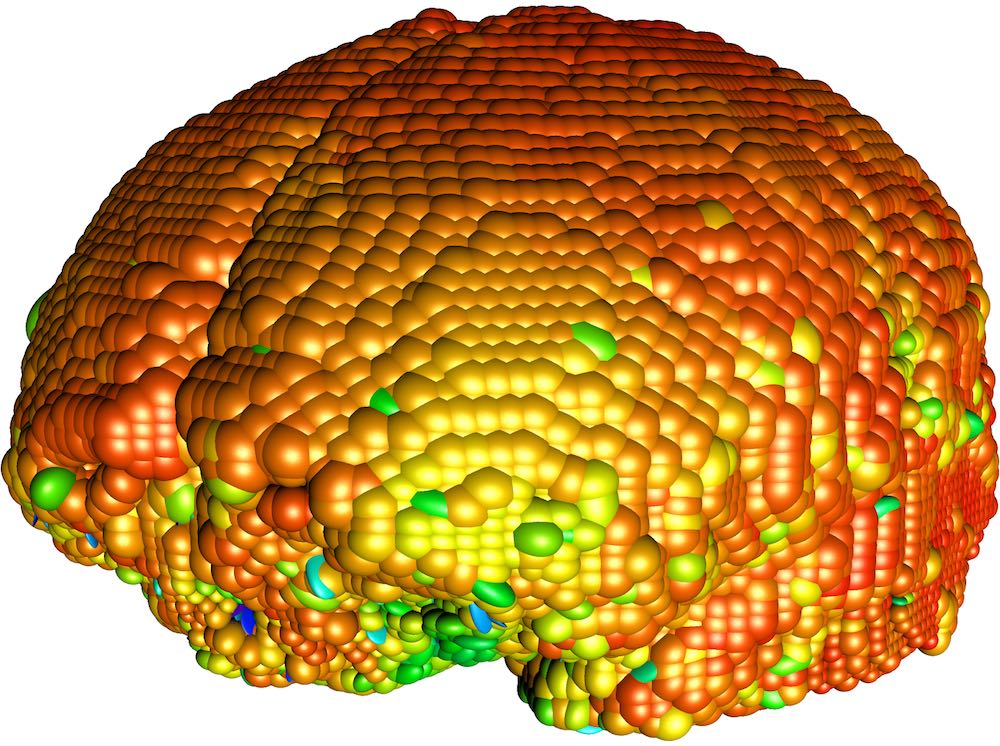}
    \caption{\(\varepsilon=3\), \(\lambda=25\).}\label{subfig:CaminoS:e3l25}
  \end{subfigure}
  \begin{subfigure}{.32\textwidth}\centering
    \includegraphics[width=.98\textwidth]{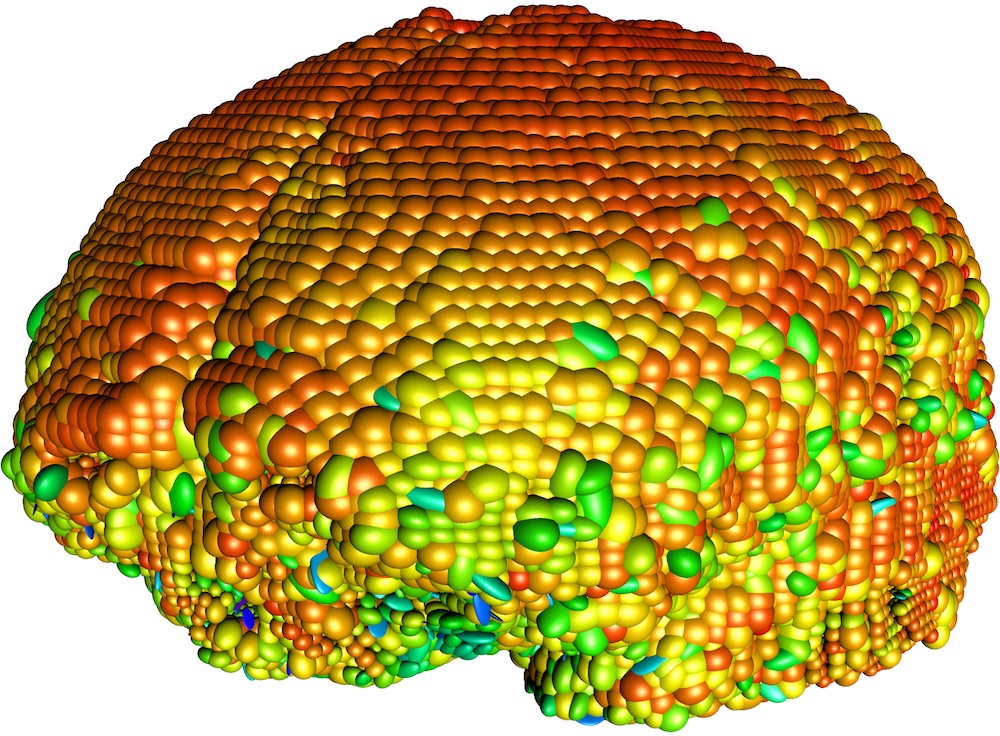}
    \caption{\(\varepsilon=3\), \(\lambda=50\).}\label{subfig:CaminoS:e3l50}
  \end{subfigure}
  \caption{Reconstruction results of manifold-valued data given on the implicit surface of the open Camino brain data set.}
  \label{fig:CaminoS}
\end{figure}
Furthermore, we can use the full 3D brain data set to demonstrate the advantages of using
a graph-based modeling of the data topology. Using our approach we can not only denoise manifold-valued
data given on a regular grid, but also on implicitly given surfaces.
Let~\(\mathcal G \coloneqq \{1,\ldots,112\}\times\{1,\ldots,112\}\times\{1,\ldots50\}\)
denote the index set of the Camino brain data set. We define the surface of the imaged human
brain as those voxels that belong to the brain and additionally have at least two surrounding background pixels 
in their direct neighborhood. By this we avoid 
single missing measurements already being treated as surface voxels.
We then construct a finite weighted \(\varepsilon-\)ball graph \(G=(V,E,w)\) from of 
all surface voxels that have a spatial distance less than \(\varepsilon\) voxels. 
The resulting graph consists of \(\lvert V\rvert = 14,384\) nodes
and~\(\lvert E\rvert = 141,978\) edges. Only \(23\) isolated open space pixels
were excluded to obtain a connected graph.

The resulting set of surface voxels is shown in Figure~\ref{subfig:CaminoS:orig}),
where one can see the front left side of the measured brain data slightly from above.
We then apply the semi-implicit minimization scheme based on Jacobi iterations~\eqref{eq:Jacobi-iterate} 
for the anisotropic \(p\)-graph Laplacian model~\eqref{eq:anisomodel},
\(\lambda\in\{10,25,50\}\) and neighborhoods of spatial
distance~\(\varepsilon\in\{2,3\}\) in the second and third row of
Figure~\ref{fig:CaminoS}, respectively.
For termination of our method we use the same stopping criterion as in the last experiment.
As observed before, for \(\lambda=10\), cf. Figures~\ref{subfig:CaminoS:e2l10})
and~\subref{subfig:CaminoS:e3l10}), the smoothing effect on the surface is strongest,
leading to nearly constant areas for the latter case. For \(\lambda=25\) both
neighborhood sizes, cf.~Figure~\ref{subfig:CaminoS:e2l10}) and~
\subref{subfig:CaminoS:e3l10}) still smooth the surface. For the smaller
neighborhood size some features already appear to look perturbed (most prominently the green
ellipsoids). This is also the case for \(\varepsilon=3,\lambda=50\),
cf.~Figure~\ref{subfig:CaminoS:e3l50}). However, all brain surface regions seem to be more
piecewise constant than for \(\varepsilon=2,\lambda=25\) in~Figure~\ref{subfig:CaminoS:e2l25}).
Both results for \(\lambda=50\) (high data fidelity) already look quite noisy in the lower part of the data.
\section{Conclusion and future work}\label{s:conclusion}
In this paper we proposed a graph framework for processing of manifold-valued
data. Using this framework we translated variational models and partial
differential equations defined over manifold-valued functions to discrete
domains of arbitrary topology. Modeling the relationship of data entities
by finite weighted graphs allows to handle both local as well as nonlocal
methods in a unified way. The nonlocal graph \(p\)-Laplacian provides
furthermore a new approach for manifold-valued denoising that leads to results 
comparable to state-of-the-art methods.
As the proposed framework consists of three independent parts it can be easily
adapted to different manifolds, other variational models and partial differential
equations, or more sophisticated numerical solvers.
After the introduction of a family of graph \(p\)-Laplacian operators in this
work, an application of our approach to other processing tasks, e.g.,
segmentation or inpainting of manifold-valued data, is immediate.
We will investigate these applications in future work and also translate
additional variational models and partial differential equations to our
framework. In particular, the nonlocal Eikonal equation and
the~\(\infty\)-Laplacian operator on finite weighted graphs will be interesting
to investigate for segmentation and interpolation tasks on manifold-valued data.
From a theoretical point-of-view there are various challenging questions to be explored in future works. First, investigating the transition from our proposed discrete model on finite weighted graphs to continuous domains will give further insights in underlying properties and characteristics of the involved equations and their solutions. For this, one has to define the right measures to analyze a sequence of finite weighted graphs which get denser and denser for $|V| \rightarrow +\infty$. Furthermore, using elaborated concepts from differential geometry will help to prove consistency of the involved operators and energy functionals with respect to existing continuous operators, e.g., the Dirichlet energy and the graph $p$-Laplacians.
An investigation of the rigorous stability conditions for the numerical algorithms proposed in this work is an interesting topic on its own. Instead of heuristically choosing the time step width $\Delta t$ in the explicit Euler time discretization, one could try to deduce sharp stability conditions. It becomes clear that this will be directly depending on the local curvature of the underlying manifold and hence involves estimations based on further tools from differential geometry, especially the curvature tensor and Christoffel symbols.

\paragraph{Acknowledgements.} The authors would like to thank
Jan Lellmann for helpful discussions on the visualization of the NED dataset,
Johannes Persch for fruitful discussions,
and Gabriele Steidl for carefully reading a preliminary version of this manuscript.
RB acknowledges support by the German Research Foundation (DFG) within the
project BE 5888/2-1.
DT acknowledges his funding by the ERC via grant EU FP7-ERC Consolidator
Grant~$615216$ LifeInverse.

\AtNextBibliography{\small}
\printbibliography
\end{document}